\documentclass[12pt]{amsart}
\usepackage{latexsym}
\usepackage{amssymb}
\usepackage{amscd}
\usepackage{mathrsfs}
\usepackage{bbm}
\usepackage{CJK}
\usepackage[all]{xy}
\usepackage{tikz-cd}

\renewcommand\a{\alpha}
\renewcommand\b{\beta}
\newcommand\g{\gamma}
\renewcommand\d{\delta}
\newcommand\la{\lambda}
\newcommand\z{\zeta}

\newcommand\io{\iota}

\newcommand\s{\sigma}

\newcommand\f{\phi}
\newcommand\vf{\varphi}

\renewcommand\t{\tau}

\newcommand\Om{\Omega}
\newcommand\w{\omega}

\newcommand\vD{\varDelta}

\newcommand\vL{\varLambda}

\newcommand\vG{\varGamma}

\newcommand{\QQ}{\mathbb Q}
\newcommand{\FF}{\mathbb F}

\newcommand{\ZZ}{\mathbb Z}
\newcommand{\NN}{\mathbb N}
\newcommand{\CC}{\mathbb C}

\newcommand\BA{\mathbf A}

\newcommand\BE{\mathbf E}

\newcommand\BD{\mathbf D}

\newcommand\BZ{\mathbf Z}

\newcommand\BJ{\mathbf J}
\newcommand\bB{\mathbf B}

\newcommand\BU{\mathbf U}
\newcommand\BV{\mathbf V}

\newcommand\Bk{\mathbf k}

\newcommand\Bh{\mathbf h}
\newcommand\Bi{\mathbf i}
\newcommand\Bj{\mathbf j}
\newcommand\Bn{\mathbf n}
\newcommand\Ba{\mathbf a}
\newcommand\Bc{\mathbf c}
\newcommand\Bd{\mathbf d}
\newcommand\Be{\mathbf e}

\newcommand\Bxi{\boldsymbol\xi}
\newcommand\Bz{\boldsymbol\z}

\newcommand\ZC{\mathcal{C}}

\newcommand\SC{\mathscr{C}}

\newcommand\SL{\mathscr{L}}

\newcommand\SM{\mathscr{M}}
 
\newcommand\SO{\mathscr{O}}
\newcommand\SP{\mathscr{P}}

\newcommand\SX{\mathscr{X}}

\newcommand\iv{^{-1}}

\newcommand\wt{\widetilde}

\newcommand\ol{\overline}
\newcommand\ul{\underline}

\newcommand\lra{\leftrightarrow}

\newcommand\ora{\overrightarrow}

\newcommand\Hom{\operatorname{Hom}}

\newcommand\Ind{\operatorname{Ind}}

\newcommand\Res{\operatorname{Res}}
\newcommand\Mod{\operatorname{-Mod}}

\newcommand\Tr{\operatorname{Tr}\,}
\newcommand\ch{\operatorname{ch}}

\newcommand\id{\operatorname{id}}
\newcommand\Id{\operatorname{Id}}

\renewcommand\Im{\operatorname{Im}}

\newcommand\weit{\operatorname{wt}}

\newcommand\Diag{\operatorname{Diag}}

\newcommand{\isom}{\,\raise2pt\hbox{$\underrightarrow{\sim}$}\,}
\numberwithin{equation}{section}

\newtheorem{thm}{Theorem}[section]
\newtheorem{lem}[thm]{Lemma}
\newtheorem{cor}[thm]{Corollary}
\newtheorem{prop}[thm]{Proposition}

\def \para#1{\par\medskip\textbf{#1}
              \addtocounter{thm}{1}}

\def \remark#1{\par\medskip\noindent
                \textbf{Remark #1}
                \addtocounter{thm}{1}}

\def \remarks#1{\par\medskip\noindent
                \textbf{Remarks #1}
                \addtocounter{thm}{1}}

\begin{document}
\setlength{\baselineskip}{4.9mm}
\setlength{\abovedisplayskip}{4.5mm}
\setlength{\belowdisplayskip}{4.5mm}
%%%
%%%
\renewcommand{\theenumi}{\roman{enumi}}
\renewcommand{\labelenumi}{(\theenumi)}
\renewcommand{\thefootnote}{\fnsymbol{footnote}}
%%%
\renewcommand{\thefootnote}{\fnsymbol{footnote}}
%\NoBlackBoxes
\parindent=20pt
%%%%%%%%%%%%%%%%%%%%
%%%%%%%%%%%%%%%%%%%%%%%%%%%%%%%%%%%
\medskip
\begin{center}
  {\bf Algorithm for computing canonical bases, \\
       and foldings of quantum groups} 
\par
\vspace{1cm}
Toshiaki Shoji and Zhiping Zhou
\\
\vspace{0.7cm}
\title{}
\end{center}

\begin{abstract} 
Let $\BU_q^-$ be the negative half of the quantum group of finite type.
Let $P$ be the transition matrix between the canonical basis and a PBW basis of $\BU_q^-$.
In the case where $\BU_q^-$ is symmetric, Antor gave a simple algorithm of computing
$P$ by making use of monomial bases. 
By the folding theory, $\BU_q^-$ (symmetric, with a certain
automorphism) is related to the quantum group $\ul\BU_q^-$ of non-symmetric type.
In this paper, we extend the results of Antor to the non-symmetric case, and discuss
the relationship between the algorithms for $\BU_q^-$ and for $\ul\BU_q^-$. 
\end{abstract}

\maketitle
\pagestyle{myheadings}
\markboth{SHOJI AND ZHOU}{ALGORITHM FOR CANONICAL BASES}

%%%%%%%%%%%%%%%%%%%%
%%%%%%%%%%%%%%%%%%%%%%%%%%%%%%%%%%%
\bigskip
\medskip

\begin{center}
{\sc Introduction}  
\end{center}  

Let $\BU_q^-$ be the negative half of the quantum group $\BU_q$ of finite type,
associated to the Cartan datum $X = (I, (\ ,\ ))$, with the vertex set $I$.
Hence $\BU_q^-$ is an associative algebra over $\QQ(q)$ with generators
$f_i$ ($i \in I$) satisfying $q$-Serre relations. 
Let $w_0$ be the longest element of the Weyl group $W$ of $X$, and
for a reduced expression $w_0 = s_{i_1}\cdots s_{iN}$ of $W$, consider the
sequence $\Bh = (i_1, \dots, i_N) \in I^N$.  For such $\Bh$, a PBW basis
$\SX_{\Bh} = \{ L(\Bc, \Bh) \mid \Bc \in \NN^N \}$, and a canonical basis
$\bB_{\Bh} = \{ b(\Bc, \Bh) \mid \Bc \in \NN^N\}$ of $\BU_q^-$ are defined.
PBW bases depend on the choice $\Bh$, but
the canonical basis does not depend on $\Bh$, which is written as $\bB = \bB_{\Bh}$. 
The canonical basis $\bB$ was constructed by Lusztig \cite{L-can} in the case where
$X$ is symmetric of finite type, by using the representation theory of quivers.
In a subsequent paper \cite{L-quiver}, it was generalized to $\BU_q^-$ of symmetric Kac-Moody type,
by making use of the geometry of quivers
(called a geometric categorification of $\BU_q^-$). Finally in \cite{L-book}, it was extended to
the case of arbitrary Kac-Moody type, by considering the geometry of quivers with automorphisms
(called a geometric categorification with automorphisms).
At the same time, Kashiwara \cite{Ka} constructed, for arbitrary $\BU_q^-$ of Kac-Moody type,
the global crystal basis by using the theory of crystal basis.
It is known that these two bases are equal (\cite{GL},\cite{MSZ2}). 
\par
We return to the case where $\BU_q^-$ is of finite type. 
Canonical bases (or global crystal bases) are important objects for the representation
theory of quantum groups, but the direct computation is very hard.
Meanwhile, PBW bases can be constructed in a more elementary way, and are appropriate for
explicit computations. 
Hence it is desirable to know the transition matrix $P$ between $\SX_{\Bh}$ and $\bB_{\Bh}$. 
In $\BU_q^-$, there exists another basis $\SM_{\Bh} = \{ m(\Bc, \Bh) \mid \Bc \in \NN^N\}$
associated to $\Bh$, called
a monomial basis,  consisting of ``monomials'' $f_{i_1}f_{i_2} \cdots f_{i_k}$ of generators.
(See 1.7 for the definition of monomial bases.
Note that monomial bases are not unique even if $\Bh$ is fixed.
We discuss just by fixing one such a basis.)
They were originally introduced in \cite{L-can} for the symmetric case,
as a step for constructing canonical bases.  
\par
In the case where $X$ is symmetric, Antor proposed an algorithm of computing
the matrix $P$, which is explained as follows. There exists a canonical inner product
$(\ ,\ )$ on $\BU_q^-$. We consider the matrix
$\vL = \Bigl((m(\Bc,\Bh), m(\Bc', \Bh))\Bigr)_{\Bc, \Bc'}$
of inner products. Then various base changes among $\bB_{\Bh}, \SX_{\Bh}, \SM_{\Bh}$
produce certain matrix equations involving $\vL$ and $P$. He showed that there exists
a simple algorithm of computing $P$ from $\vL$. So the problem of computing $P$ is reduced to
the computation of the inner products $(F, F')$ for arbitrary monomials $F,F'$.   
Based on the geometric categorification of $\BU_q^-$ due to Lusztig, Antor
gave a closed formula for the inner product $(F, F')$.  Thus $P$ is computable once
the monomial basis $\SM_{\Bh}$ is explicitly given.
\par
Here we explain about the folding theory of quantum groups. 
Assume that $X$ is a symmetric
Cartan datum, with an admissible automorphism $\s$ on $X$ (see 3.1 for the definition).
Then $\s$ induces an automorphism on $\BU_q^-$, which we also denote by $\s$.
A Cartan datum $\ul X$ is defined from the pair $(X, \s)$, and one can define a
quantum group $\ul\BU^-_q$ associated to $\ul X$.
The folding theory discusses the relationship between $\BU_q^-$ and $\ul\BU_q^-$
(see \cite{L-book}, \cite{SZ1}, \cite{MSZ1}).
For example, let $\bB$ (resp. $\ul\bB$) be the set of canonical basis of
$\BU_q^-$ (resp. $\ul\BU_q^-$).
Then $\s$ permutes $\bB$, and the set $\bB^{\s}$ of $\s$-stable elements in $\bB$ is in
bijection with $\ul\bB$.  If we choose a suitable $\Bh$ (resp. $\ul\Bh$)
for $\BU_q^-$ (resp. $\ul\BU_q^-$),  
then a similar result holds for PBW bases, namely, $\s$ permutes $\SX_{\Bh}$, and
the set $\SX_{\Bh}^{\s}$ is in bijection with $\SX_{\ul\Bh}$. 
\par
In this paper, we generalize the result of Antor for the non-symmetric case, 
in connection with the folding theory of quantum groups.  So we consider
$\BU_q^-$ and $\ul\BU_q^-$ as above. 
The algorithm of computing $P$ from $\vL$ works well for $\ul\BU_q^-$, once the monomial
basis is defined.  Then the computation of $P$ is reduced
to the computation of inner products $(F, F')$ for monomials in $\ul\BU_q^-$. 
\par
Let $\SM_{\Bh}$ be the set of
monomial basis of $\BU_q^-$ constructed in \cite{L-can}. By mimicking his construction,
and by the aide of the folding theory, we can construct the monomial basis $\SM_{\ul\Bh}$
for $\ul\BU_q^-$ (for a suitable choice of $\Bh$ and $\ul\Bh$).
But note that $\s$ does not act on the set $\SM_{\Bh}$.  We modify his construction, and
define a new monomial basis $\wt\SM_{\Bh}$ for $\BU_q^-$.  
We show in Proposition 4.18 that $\s$ permutes the set $\wt\SM_{\Bh}$,
and the set $\wt\SM_{\Bh}^{\s}$ of $\s$-fixed elements is in bijection
with the set $\SM_{\ul\Bh}$.
\par
So far, a good monomial basis of $\ul\BU_q^-$ was constructed,
from a view point of the folding theory,
next step is to find a closed formula for the inner products $(F, F')$ of monomials
in $\ul\BU_q^-$. The method by Antor based on the geometric categorification can not be applied
directly to our situation.  Probably, a geometric categorification with automorphisms will work
for our case. 
\par
However, in this paper, we use another type of categorification, the categorification
in terms of KLR algebras. KLR algebras associated to the Cartan datum $X$ were introduced by
Khovanov-Lauda \cite{KL1} and Rouquier \cite{R}, independently.  They have proved that
$\BU_q^-$ are categorified in terms of KLR algebras, in the case where $X$ is symmetric,
Kac-Moody type. This result was generalized by \cite{KL2} for the arbitrary Kac-Moody type. 
We now consider the folding theory setup, the Cartan data $X$ with automorphism $\s$,
and $\ul X$.  Then one can consider the KLR algebra with automorphism, associated to
$(X, \s)$. 
McNamara \cite{M} categorified $\ul\BU_q^-$ associated to $\ul X$ in terms of
the KLR algebra associated to $(X, \s)$, which is an analogue of Lusztig's geometric
categorification with automorphisms. 
\par
McNamara's categorification behaves well for the folding theory setup.  
We can compute the inner product $(F, F')$ by making use of the Mackey formula for
the representations of KLR algebras, and we obtain a closed formula for them (Theorem 5.20).
\par
As the referee pointed out, Theorem 5.20 can be proved, directly from the
basic properties of the inner product of quantum groups, without appealing
the theory of KLR algebras. See Remark 5.21 for details.
\par
Finally, returning to the original setting, we consider $\BU_q^-$ with automorphism
$\s$, and $\ul\BU_q^-$. 
Let $\wt m(\Bc, \Bh)$ be the (modified) monomial basis on $\BU_q^-$, and
let $\vL = \Bigl((\wt m(\Bc, \Bh), \wt m(\Bc', \Bh))\Bigr)$
be the matrix of the inner products.
Similarly, let $m(\ul\Bc, \ul\Bh)$ be the monomial basis of $\ul\BU_q^-$, and
consider the matrix $\ul\vL = \Bigl((m(\ul\Bc, \ul\Bh), m(\ul\Bc', \ul\Bh))\Bigr)$.
By using Theorem 5.20, we can compute $\vL$ and $\ul\vL$ explicitly. 
We denote by $\vL^{\s}$ the submatrix of $\vL$ obtained from
$\s$-stable $\wt m(\Bc, \Bh), \wt m(\Bc',\Bh)$. We compare these two matrices
$\vL^{\s}$ and $\ul\vL$. Our final result is
Theorem 6.7, which asserts that the matrix $\ul\vL$ is determined completely from the
matrix $\vL^{\s}$.
\par
Let $P$ be the transition matrix between the canonical basis
$\bB_{\Bh}$ and the PBW basis $\SX_{\Bh}$ of $\BU_q^-$, and $\ul P$ the transition matrix
between the canonical basis $\bB_{\ul\Bh}$ and the PBW basis $\SX_{\ul\Bh}$ of $\ul\BU_q^-$
(under a certain choice
of $\Bh$ and $\ul\Bh$). We can define a submatrix $P^{\s}$ similarly to $\vL^{\s}$.
Assume that $\BU_q^-$ is irreducible.  Then the order of $\s$ is equal to $p = 2$ or 3,
a prime number.  It is known by \cite{SZ1}
that $P^{\s} \equiv \ul P \mod p\ZZ[q]$. 
\par
By our algorithm, $\ul P$ is determined from $\ul\vL$.  Thus Theorem 6.7 implies that
the matrix $\ul P$ is determined from the matrix $\vL$ for $\BU_q^-$. 
In other words, $\ul P$ for $\ul\BU_q^-$
is determined completely (not for modulo $p$) from the information for a monomial basis
of $\BU_q^-$, though it is not certain whether $\ul P$ is determined from $P$. 
\par
In a recent preprint \cite{SZ3}, the monomial basis of affine quantum algebras of symmetric type,
with respect to the Beck-Nakajima's PBW basis, was constructed.
\par
The authors are grateful to the referee, for the suggestion to the improvement
of the proof of Lemma 4.13, which replaced the original proof by a simple and unified one. 
They are also indebted to the referee concerning the discussion on the direct proof of
Theorem 5.20. 
\par\bigskip
\begin{center}
{\sc Contents}
\end{center}
\par
1. Algorithm for computing canonical bases
\par
2. Commutation formulas for root vectors
\par
3. Foldings of quantum groups
\par
4. Monomial bases
\par
5. Mackey filtration for KLR algebras
\par
6. Comparison of algorithms
\par
7. Examples

\par\bigskip

\section{ Algorithm for computing canonical bases}

\para{1.1.}
Let $X = (I, (\ ,\ ))$ be a Cartan datum of finite type, where
$I$ is a finite set, and $(\ ,\ )$ is a symmetric bilinear form on the
$\QQ$-vector space $E = \bigoplus_{i \in I}\QQ \a_i$ with basis $\a_i$, satisfying
the property
\par\medskip
\par \ $(\a_i,\a_i) \in 2\ZZ_{> 0}$ for any $i \in I$,  \\
\par \ $\frac{2(\a_i,\a_j)}{(\a_i,\a_i)} \in \ZZ_{\le 0}$  for any $i \ne j \in I$.
\par\medskip
The Cartan datum $X$ is called symmetric if the Cartan matrix $(a_{ij})$ is
symmetric, where $a_{ij} = 2(\a_i,\a_j)/(\a_i,\a_i)$.   
Let $Q = \bigoplus_{i \in I}\ZZ \a_i$ be the root lattice of $X$.  We set
$Q_+ = \sum_{i \in I}\NN \a_i$, and $Q_- = -Q_+$.  
For $\g = \sum_i n_i\a_i \in Q$, we define $|\g| = \sum_i|n_i|$.  

\para{1.2.}
Let $q$ be an indeterminate. For an integer $n$, a positive integer $m$, set
\begin{equation*}
[n] = \frac{q^n- q^{-n}}{q - q\iv}, \qquad [m]^! = [1][2]\cdots [m], \quad [0]^! = 1.
\end{equation*}

Let $\BA = \ZZ[q,q\iv]$ be the ring of Laurent polynomials.  Then $[n], [m]^! \in \BA$. 
For each $i \in I$, set $q_i = q^{(\a_i,\a_i)/2}$.  We denote by $[n]_i$ 
the element obtained from $[n]$ by replacing $q$ by $q_i$.  $[m]^!_i$ is defined similarly.
\par
For the Cartan datum $X$, let $\BU_q^- = \BU_q^-(X)$ be the negative half of
the quantum group $\BU_q$ associated to $X$.  $\BU_q^-$ is an associative algebra over
$\QQ(q)$ generated by $f_i \ (i \in I)$ satisfying the $q$-Serre relations
\begin{equation*}
\sum_{k + k' = 1-a_{ij}}(-1)^kf_i^{(k)}f_jf_i^{(k')} = 0 \quad\text{ for } i \ne j \in I,
\end{equation*}
where we set $f_i^{(n)} = f_i^n/[n]_i^!$ for any $n \in \NN$. 
\par
Let ${}_{\BA}\BU_q^-$ be Lusztig's integral form of $\BU_q^-$,
namely, the $\BA$-subalgebra of $\BU_q^-$ generated by $f_i^{(n)}$ for $i \in I, n \in \NN$.
\par
Let ${}^* : \BU_q^- \to \BU_q^-$ be the anti-automorphism defined by
$f_i \mapsto f_i$ for any $i \in I$, which is called the anti-involution on $\BU_q^-$.  

\para{1.3.}
$\BU_q^-$ has a weight space decomposition $\bigoplus_{\g \in Q_-}(\BU_q^-)_{\g}$,
where $(\BU_q^-)_{\g}$ is a subspace of $\BU_q^-$ spanned by elements $f_{i_1}\cdots f_{i_k}$
such that $\a_{i_1} + \cdots + \a_{i_k} = -\g$. $x \in \BU_q^-$ is called homogeneous
with $\weit x = \g$ if $x \in (\BU_q^-)_{\g}$.  We define a multiplication
on $\BU_q^-\otimes \BU_q^-$ by
\begin{equation*}
\tag{1.3.1}  
(x_1\otimes x_2)\cdot (x_1'\otimes x_2') = q^{-(\weit x_2, \weit x_1')}x_1x_1'\otimes x_2x_2',
\end{equation*}
where $x_1, x_1', x_2, x_2'$ are homogeneous in $\BU_q^-$. Then $\BU_q^-\otimes \BU_q^-$
turns out to be an associative algebra with respect to this twisted product.
One can define a homomorphism $r : \BU_q^- \to \BU_q^-\otimes \BU_q^-$ by
$r(f_i) = f_i\otimes 1 + 1 \otimes f_i$ for each $i \in I$.
There exists a unique bilinear form $(\ ,\ )$ on $\BU_q^-$ satisfying the following
properties; $(1, 1) = 1$, and
\begin{align*}
\tag{1.3.2}  
  (f_i, f_j) &= \d_{ij}(1-q_i^2)\iv, \\
  (x, y'y'') &= (r(x), y'\otimes y''),  \\
  (x'x'', y) &= (x'\otimes x'', r(y)),
\end{align*}
where the bilinear form on $\BU_q^-\otimes \BU_q^-$ is defined by
$(x_1\otimes x_2, x_1'\otimes x_2') = (x_1,x_1')(x_2, x_2')$.
The thus defined bilinear form is symmetric and non-degenerate.
The bilinear form $(\ ,\ )$ is called the inner product on $\BU_q^-$.
\par
The inner product satisfies the property
\begin{equation*}
\tag{1.3.3}  
\bigl((\BU_q^-)_{\g}, (\BU_q^-)_{\g'}\bigr) = 0 \quad\text{ if } \quad \g \ne \g'.
\end{equation*}
\para{1.4.}
Let $\vD$ be the root system associated to $X$, and $\vD^+$ the set of
positive roots.
Let $W$ be the Weyl group of $X$ which is generated by
simple reflections $s_i$ corresponding to $\a_i$.
Let $w_0$ be the longest element in $W$, and
let $w_0 = s_{i_1}s_{i_2}\cdots s_{i_N}$ be a reduced expression of $w_0$,
where $N = |\vD^+|$.  We set $\Bh = (i_1, \dots, i_N)$, and denote it
a reduced sequence of $w_0$. We fix a reduced sequence $\Bh$.  For each
$1 \le k \le N$, set $\b_k = s_{i_1}s_{i_2}\cdots s_{i_{k-1}}(\a_{i_k})$.
It is known that $\b_k \in \vD^+$, and $\b_1, \dots, \b_N$ are all distinct.
Thus $\vD^+ = \{ \b_1, \dots, \b_N\}$, and $\Bh$ determines a total order on
the set $\vD^+$.
\par
Let $T_i : \BU_q \to \BU_q$ be the braid group action (which coincides with
Lusztig's braid group action $T_{i,1}''$ in \cite{L-book}), and set
\begin{equation*}
F_{\b_k}^{(n)} = T_{i_1}T_{i_2}\cdots T_{i_{k-1}}(f_{i_k}^{(n)}).
\end{equation*}
Then $F_{\b_k}^{(n)} \in \BU_q^-$, and $\weit(F_{\b_k}^{(n)}) = -n\b_k$.
If $n = 1$, we write it as $F_{\b_k}$. 
$F_{\b_k}$ is called a root vector in $\BU_q^-$ corresponding to
the root $-\b_k$.  For $\Bc = (c_1, \dots, c_N) \in \NN^N$,
we define
\begin{equation*}
L(\Bc, \Bh) = F_{\b_1}^{(c_1)}F_{\b_2}^{(c_2)}\cdots F_{\b_N}^{(c_N)}.
\end{equation*}  
Then the set $\SX_{\Bh} = \{ L(\Bc, \Bh) \mid \Bc \in \NN^N\}$ gives a basis
of $\BU_q^-$, which is called the PBW basis of $\BU_q^-$ associated to $\Bh$.
$\SX_{\Bh}$ gives an $\BA$-basis of ${}_{\BA}\BU_q^-$. 
\par
The inner products for PBW bases were computed in \cite[38.2]{L-book} as follows.
For $\Bc = (c_1, \dots, c_N), \Bc' = (c_1', \dots, c_N') \in \NN^N$, we have
\begin{equation*}
\tag{1.4.1}  
\bigl(L(\Bc,\Bh), L(\Bc',\Bh)\bigr)
      = \prod_{k=1}^N(f_{i_k}^{(c_k)}, f_{i_k}^{(c'_k)})
      = \prod_{k = 1}^N \d_{c_k, c_k'}\prod_{d=1}^{c_k} \frac{1}{1- q_{i_k}^{2d}}.
\end{equation*}  
In particular, $\SX_{\Bh}$ is an orthogonal basis of $\BU_q^-$
with respect to the inner product. 

\para{1.5.}
The bar involution ${}^- : \BU_q^- \to \BU_q^-$ is an $\QQ$-algebra
automorphism given by $q \mapsto q\iv, f_i \mapsto f_i \ (i \in I)$.
We define a total order on $\NN^N$ by $\Bc < \Bc'$, for
$\Bc = (c_1, \dots, c_N), \Bc' = (c_1', \dots, c_N')$,  if and only if
there exists $k$ such that $c_1 = c_1', \dots, c_{k-1} = c_{k-1}'$ and $c_k < c_k'$. 
It was proved by Lusztig in \cite{L-can} for the symmetric case, and by Nakajima in
\cite[Thm. 5.1]{N-lec} for the general case, that the following triangularity property
holds for the PBW basis. 

\begin{equation*}
\tag{1.5.1}  
  \ol{ L(\Bc, \Bh)} = L(\Bc,\Bh) + \sum_{\Bc' > \Bc}a_{\Bc', \Bc}L(\Bc',\Bh)
       \quad\text{ with } \quad a_{\Bc',\Bc} \in \BA.
\end{equation*}  
By making use of (1.5.1), the following result is proved.
\par\medskip\noindent
(1.5.2) \ For a given $L(\Bc,\Bh)$, there exists a unique element
$b(\Bc, \Bh)$ satisfying the properties
\par\medskip
(i) $\ol{b(\Bc,\Bh)} = b(\Bc,\Bh)$, 
\par
(ii) 
 $b(\Bc, \Bh) = L(\Bc,\Bh) + \sum_{\Bc' > \Bc}p_{\Bc',\Bc}L(\Bc', \Bh),
\quad \text{ with }\quad  p_{\Bc',\Bc} \in q\ZZ[q]$.

\par\medskip
Then $\bB_{\Bh} = \{ b(\Bc,\Bh) \mid \Bc \in \NN^N\}$ gives a basis of $\BU_q^-$,
and an $\BA$ basis of ${}_{\BA}\BU_q^-$.
$\bB_{\Bh}$ is called the canonical basis of $\BU_q^-$ associated to $\Bh$.

\remark{1.6.}
In the case where $\BU_q^-$ is of symmetric type,
the canonical basis $\bB_{\Bh}$ was constructed in \cite{L-can}, by making use of
the representation theory of quivers, and it was shown that $\bB_{\Bh}$
does not depend on $\Bh$, which
we denote as $\bB_{\Bh} = \bB$.  For the non-symmetric case, 
the independence of $\Bh$ was verified by Saito \cite{S} by making use of the theory
of crystal bases. Thus the canonical basis $\bB$ exists. 

\para{1.7.}
Besides the PBW basis $\SX_{\Bh}$, and the canonical basis $\bB_{\Bh}$,
we consider another type of basis associated to $\Bh$.
A basis $\SM_{\Bh} = \{ m(\Bc, \Bh) \mid \Bc \in \NN^N\}$
is called a monomial basis if it satisfies the properties;
\par\medskip\noindent
(1.7.1)
\begin{enumerate}
\item   \ $m(\Bc,\Bh)$ is written as a ``monomial'' of generators $f_i^{(n)}$
such as $m(\Bc,\Bh) = f_{i_1}^{(d_1)}f_{i_2}^{(d_2)}\cdots f_{i_k}^{(d_k)}$
for $i_k \in I, d_k \in \NN$.
\item  \ $m(\Bc,\Bh)$ is written as
\begin{equation*}
  m(\Bc,\Bh) = L(\Bc,\Bh) + \sum_{\Bc' > \Bc}h_{\Bc',\Bc}L(\Bc',\Bh)
                    \quad\text{ with }\quad h_{\Bc',\Bc} \in \BA. 
\end{equation*}  
\end{enumerate}
Note that any element in $\SM_{\Bh}$ is bar-invariant, hence
the monomial basis is regarded as an intermediate object between the PBW basis and
the canonical basis.  But $\SM_{\Bh}$ is not unique for a given $\Bh$, even if it exists. 

\remark{1.8.}
A monomial basis $\SM_{\Bh}$ was constructed by Lusztig \cite{L-can}, in the case
where $\BU_q^-$ is symmetric, by making use of the representation theory of
quivers.  In the symmetric case, there are related works \cite{DD}, \cite{Re}
for constructing monomial bases, based on the Ringel-Hall algebras associated to quivers. 
In the general case, a monomial basis was constructed by
Chari and Xi \cite{CX} by case by case consideration.
In the book \cite{DDPW}, the authors developed the theory of Ringel-Hall algebras
associated to quivers with automorphism, and generalized the result of
\cite{DD} to the non-symmetric case. 

\para{1.9.}
In the discussion below, we fix a monomial bases $\SM_{\Bh}$.
Following Antor \cite{A}, we give an algorithm of expressing the canonical basis
in terms of the PBW basis.
We denote by $P = (p_{\Bc, \Bc'})$ the transition matrix from
the basis $\SX_{\Bh}$ to the basis $\bB_{\Bh}$ with respect to the total order on $\NN^N$
given in 1.5.
Then $P$ is a lower unitriangular matrix with coefficients in $q\ZZ[q]$.
We are interested in computing the matrix $P$. 
\par
We define the matrices
\begin{align*}
  D = \Bigl((L(\Bc,\Bh), L(\Bc',\Bh))\Bigr),
         \qquad \Om = \Bigl((b(\Bc,\Bh), b(\Bc', \Bh))\Bigr) 
\end{align*}  
indexed by $\Bc, \Bc' \in \NN^N$.  Since PBW bases are mutually orthogonal with respect to
$(\ ,\ )$, $D$ is a diagonal matrix.
We have a relation of matrices 
\begin{equation*}
\tag{1.9.1}  
\Om = {}^tPDP.
\end{equation*}  

Let $H = (h_{\Bc,\Bc'})$ be the transition matrix from the basis $\SX_{\Bh}$ to the basis 
$\SM_{\Bh}$.  Then by (1.7.1) (ii), $H$ is a lower unitriangular matrix
with coefficients in $\BA$.  
\par
We define a matrix $\vL$ by
\begin{equation*}
\vL = \Bigl(\bigl(m(\Bc,\Bh), m(\Bc',\Bh)\bigr)\Bigr)
\end{equation*}  
Then one can write as
\begin{equation*}
\tag{1.9.2}  
\vL = {}^tHDH.
\end{equation*}

Let $Q = (q_{\Bc,\Bc'})$ be the transition matrix from the basis $\bB_{\Bh}$
to the basis $\SM_{\Bh}$.
Then one can write as 
\begin{equation*}
 m(\Bc, \Bh) = b(\Bc, \Bh) + \sum_{\Bc' > \Bc}q_{\Bc',\Bc}b(\Bc',\Bh),
\end{equation*}  
and $Q$ is a lower unitriangular matrix with coefficients in $\BA$.
We have
\begin{equation*}
\tag{1.9.3}
H = PQ.
\end{equation*}  
Note that $m(\Bc, \Bh)$ and $b(\Bc',\Bh)$ are bar-invariant, hence
$q_{\Bc,\Bc'}$ is bar-invariant, namely, we have
\begin{equation*}
\ol{q_{\Bc,\Bc'}} = q_{\Bc,\Bc'} \quad\text{ with }\quad q_{\Bc,\Bc'} \in \BA.
\end{equation*}  

The following result was proved by Antor \cite[Cor. 2.3]{A}.

\begin{prop}  %%%%   Prop. 1.10
There exist matrix equations
\begin{equation*}
\tag{1.10.1}  
\vL = {}^tHDH, \qquad H = PQ,
\end{equation*}
where
\par
$H$ : a lower unitriangular matrix with coefficients in $\BA$,

$D$ : a diagonal matrix with coefficients in $\QQ(q)$,

$P$ : a lower unitriangular matrix with coefficients in $q\ZZ[q]$, 

$Q$ : a lower unitriangular matrix where the coefficients are all
bar-invariant and belong to $\BA$.
\par\medskip
The matrices $H, D, Q$ and $P$ are determined uniquely from $\vL$ by these conditions,
and there exists an algorithm of computing $H, D, Q, P$ starting from $\vL$.  
\end{prop}  
\begin{proof}
Let $Z = (z_{\Bc,\Bc'})$ be one of the matrices $\vL, H, D, P, Q$
appearing in the proposition. For each $\g, \g' \in Q_+$,  let $Z_{\g, \g'}$
be the submatrix consisting of $z_{\Bc,\Bc'}$ such that
$\sum_{1 \le i \le N}c_i\b_i = \g, \sum_{1 \le i \le N}c'_i\b_i = \g'$,
where $\vD^+ = \{ \b_1, \dots, \b_N\}$ is as in 1.4. Then $Z$ is regarded as a block
matrix, indexed by $\g, \g'$.  Assume that $\g \ne \g'$.
On the one side, $Z_{\g,\g'} = 0$ for $Z = H, P, Q$ from the definition.
On the other hand, by (1.3.3), $Z_{\g, \g'} = 0$ for $Z = \vL, D$.
Hence in the computation of matrices $Z$, we may only consider the part $Z_{\g,\g}$. 
Note that $Z_{\g,\g}$ is a matrix of finite rank.
The matrix relation in (1.10.1) still holds by restricting to $Z_{\g,\g}$-part.
Hence in the discussion below, we fix $\g \in Q_+$, and consider the restriction
to $Z_{\g, \g}$.  We write them as $Z$ by omitting $\g$,
but assume that $Z$ is a matrix of finite rank. 
\par
Since $H$ is lower unitriangular and $D$ is diagonal, it is well-known that
the matrix equation $\vL = {}^tHDH$ determines $H$ and $D$ uniquely, and
there exists an algorithm,  based on the linear algebra,
of computing $H$ and $D$ starting from $\vL$. (Note that the computation is done
for the $\QQ(q)$ matrices, but the existence of a $\BA$-matrix $H$ and a $\QQ(q)$-matrix
$D$ is already guaranteed.)
\par
Thus it is enough to show that the matrix equation $H = PQ$ 
determines $P$ and $Q$ uniquely, and there exists an algorithm of computing
$P$ and $Q$ from $H$.
First we show the uniqueness of $P$ and $Q$.  Assume that
$H = PQ = P'Q'$, where $P',Q'$ satisfy similar conditions as $P,Q$.
Then $Q'Q\iv = {P'}\iv P$.
Off-diagonal coefficients of ${P'}\iv P$ are contained in $q\ZZ[q]$,
off-diagonal coefficients
of $Q'Q\iv $ are contained in $\ZZ[q,q\iv]$ and bar-invariant.  Hence they
are equal to 0.  Thus $P = P', Q = Q'$.
\par
Next we show there exists an algorithm of determining $P$ and $Q$.
Write, for $i > j$, 
\begin{align*}
  h_{ij} &= \sum_{j \le k \le i}p_{ik}q_{kj} = p_{ij} + q_{ij} + \sum_{j < k < i} p_{ik}q_{kj}.
\end{align*}
By induction on $i-j$, we may assume that $p_{ik}$ is known for $j < k$,
and $q_{kj}$ is known for $k < i$. Thus $a =  p_{ij} + q_{ij}$ is already determined.
We show this determines $p_{ij}, q_{ij}$ as required. 
Now $a \in \ZZ[q,q\iv]$ is written as $a = a_+ + a_0 + a_-$ with
$a_+ \in q\ZZ[q], a_- \in q\iv\ZZ[q\iv], a_0 \in \ZZ$.
Then $a$ is decomposed as 
\begin{equation*}
a = (a_- + \ol{a_-} + a_0) + (a_+ - \ol{a_-}). 
\end{equation*}
Here $a_+ - \ol{a_-} \in q\ZZ[q]$, and $a_- + \ol{a_-} + a_0$ is bar-invariant.
Hence $p_{ij} = a_+ - \ol{a_-}, q_{ij} = a_- + \ol{a_-} + a_0$ satisfies
the required decomposition. 
\end{proof}  

\para{1.11.}
Since the computation of PBW bases is relatively easy, it is important to know
the transition matrix $P$ for the computation of the canonical basis $\bB_{\Bh}$.
By Proposition 1.10, $P$ can be computed if one can
find a monomial basis $\SM_{\Bh}$ for which the matrix
$\vL = \Bigl((m(\Bc,\Bh), m(\Bc',\Bh))\Bigr)$ is explicitly computable. 
In later discussion, we show,  for a suitable  monomial basis $\SM_{\Bh}$,
that there exists a closed formula expressing the matrix $\vL$

\par\bigskip
\section { Commutation formulas for root vectors }
\para{2.1.}
Fix a reduced sequence $\Bh = (i_1, \dots, i_N)$ of $w_0$, and
let $\b_1, \dots, \b_N$ be the total order of $\vD^+$ associated to $\Bh$. 
For $\Bc = (c_1, \dots, c_N)$,
consider the PBW basis $L(\Bc, \Bh) = F_{\b_1}^{(c_1)}\cdots F_{\b_N}^{(c_N)}$,
where $F_{\b_k}^{(n)} = T_{i_1}T_{i_2}\cdots T_{i_{k-1}}(f_{i_k}^{(n)})$
is a root vector corresponding to the root $\b_k \in \vD^+$, and $n \in \NN$. 

\par
The following commutation formula for root vectors
was proved by Levendorskii-Soibelman \cite{LS}.

\begin{prop}  %%%% Prop. 2.2
Assume that $1 \le p' < p \le N$. Then we have 
\begin{equation*}
  F_{\b_p}F_{\b_{p'}} = q^{-(\b_p, \b_{p'})}F_{\b_{p'}}F_{\b_p}
  + \sum_{n_{p'+1}, \dots, n_{p-1} \ge 0}c(n_{p'+1}, \dots, n_{p-1})
           F_{\b_{p'+1}}^{(n_{p'+1})}\cdots F_{\b_{p-1}}^{(n_{p-1})} 
\end{equation*}
with $c(n_{p'+1}, \dots, n_{p-1}) \in \BA$. 
\end{prop}  

A generalization of Levendorskii-Soibelman's formula was proved by
Xi \cite{X1}. 

\begin{prop}  %%%%   Prop. 2.3
  For any $1 \le p' < p \le N$ and $s, t \in \NN$,
\begin{equation*}
  F_{\b_p}^{(s)}F_{\b_{p'}}^{(t)} = q^{-st(\b_p,\b_{p'})}F_{\b_{p'}}^{(t)}F_{\b_p}^{(s)}
        + \sum_{\substack{n_{p'}, \dots, n_p \ge 0 \\
            n_{p'} < t, n_p < s}}c(n_{p'}, \dots, n_p)
               F_{\b_{p'}}^{(n_{p'})}\cdots F_{\b_p}^{(n_p)} 
\end{equation*}
with $c(n_{p'}, \dots, n_p) \in \BA$. 
\end{prop}

\para{2.4.} By making use of Proposition 2.3, we shall compute
the expansion of $L(\Bc, \Bh)L(\Bc',\Bh)$ in terms of the basis $\SX_{\Bh}$.
For $\Bc = (c_1, \dots, c_N) \in \NN^N$, we define a subset $s(\Bc)$ of $[1,N]$ by
$s(\Bc) = \{ k \in [1,N] \mid c_k \ne 0\}$.
The following formula is used in later discussion in Section 4.

\begin{prop}  %%%%  Prop. 2.5 
\begin{enumerate}
\item 
For $\Bc, \Bc' \in \NN^N$, we have
\begin{equation*}
\tag{2.5.1}
  L(\Bc,\Bh)L(\Bc',\Bh) = \sum_{\Bc''}a_{\Bc,\Bc'}^{\Bc''}L(\Bc'',\Bh),
\end{equation*}
where $a_{\Bc,\Bc'}^{\Bc''} \in \BA$, and
$\Bc''$ runs over all the elements in $\NN^N$ such that $\Bc'' \ge \Bc$.
\item
Write $L(\Bc,\Bh) = F_{\b_s}^{(c_s)}\cdots F_{\b_t}^{(c_t)}$,
$L(\Bc',\Bh) = F_{\b_{p'}}^{(c'_{p'})}\cdots F_{\b_p}^{(c'_p)}$, where 
all of $c_s, c_t$, $c'_{p'}, c'_p$ are non-zero.    
Then $a_{\Bc,\Bc'}^{\Bc''} \ne 0$ implies that
$s(\Bc'') \subset [a, b]$, where $a = \min\{ s, p'\}, b = \max\{ t, p\}$.
\end{enumerate}  
\end{prop}

\begin{proof}
  We prove the proposition by induction on $|\weit(L(\Bc,\Bh))| + |\weit(L(\Bc',\Bh))|$.
We write $L(\Bc,\Bh) = F_{\b_s}^{(c_s)}\cdots F_{\b_t}^{(c_t)}$,
$L(\Bc',\Bh) = F_{\b_{p'}}^{(c'_{p'})}\cdots F_{\b_p}^{(c'_p)}$, where
$c_s, c_t,c'_{p'}, c'_p$ are all non-zero.
In the case where $s = t$ and $p' = p$, the proposition certainly holds by
Proposition 2.3.
\par
(a) \ First assume that $s< t$.  Then we have
\begin{align*}
\tag{2.5.2}  
  L(\Bc, \Bh)L(\Bc',\Bh) 
         &= F_{\b_s}^{(c_s)}
           (F_{\b_{s+1}}^{(c_{s+1})}\cdots F_{\b_t}^{(c_t)})L(\Bc', \Bh).
\end{align*}
Let $\Bd = (0, \dots, 0, c_{s+1}, \dots, c_t, 0, \dots, 0) \in \NN^N$, and 
write $F_{\b_{s+1}}^{(c_{s+1})}\cdots F_{\b_t}^{(c_t)} = L(\Bd, \Bh)$.
By induction, $L(\Bd, \Bh)L(\Bc',\Bh)$
is a linear combination of $L(\Be, \Bh)$ such that $\Be \ge \Bd$.
We can write $L(\Be,\Bh) = L'L''$ as a product of PBW bases,
where $L' = F_{\b_1}^{(e_1)}\cdots F_{\b_s}^{(e_s)}$,
and $L'' = F_{\b_{s+1}}^{(e_{s+1})}\cdots F_{\b_N}^{(e_N)}$.
Since $\Be \ge \Bd$, either $e_i \ne 0$ for some $i \le s$, 
or $e_1 = \dots = e_s = 0$ and
$(e_{s+1}, \dots, e_N) \ge (c_{s+1}, \dots, c_t, 0, \dots,0)$. 
In the latter case, $L(\Be,\Bh) = L''$. 
\par
Now, by applying (2.5.1) for $F_{\b_s}^{(c_s)}$ and $L'$,
$F_{\b_s}^{(c_s)}L'$ is a linear combination of
$F_{\b_1}^{(e'_1)}\cdots F_{\b_s}^{(e'_s)}$ such that
$(e'_1, \dots, e'_s) \ge (0, \dots, 0, c_s)$, and so,
$F_{\b_s}^{(c_s)}L(\Be,\Bh)$ is a linear combination of
$(F_{\b_1}^{(e'_1)}\cdots F_{\b_s}^{(e'_s)})L''$, which coincides with a PBW basis
$L(\Bc'',\Bh)$.  If $(e_1', \dots, e_s') > (0,\dots, 0, c_s)$, we have $\Bc''> \Bc$.
Thus we may assume that $(e_1', \dots, e_s') = (0, \dots, 0, c_s)$. 
In this case, $F_{\b_1}^{(e'_1)}\cdots F_{\b_s}^{(e'_s)} = F_{\b_s}^{(c_s)}$,
and by comparing the weights, $\weit L' = 0$, hence $e_1 = \dots = e_s = 0$.
Thus $(e_{s+1}, \dots, e_N) \ge (c_{s+1}, \dots, c_t, 0, \dots, 0)$.
We have
\begin{equation*}
L(\Bc'',\Bh) =
F_{\b_s}^{(c_s)}L'' = F_{\b_s}^{(c_s)}F_{\b_{s+1}}^{(e_{s+1})}\cdots F_{\b_N}^{(e_N)}.
\end{equation*}
This implies that $\Bc'' \ge \Bc$. 
Hence $L(\Bc,\Bh)L(\Bc',\Bh)$ is a linear combination of $L(\Bc'',\Bh)$ such that
$\Bc'' \ge \Bc$.
\par
(b) \ Next assume that $s = t$ and $p' < p$. Then $\Bc = (0, \dots, 0, c_t, 0, \dots, 0)$.
We have
\begin{align*}
\tag{2.5.3}  
  L(\Bc,\Bh)L(\Bc',\Bh) &= (F_{\b_t}^{(c_t)}F_{\b_{p'}}^{(c'_{p'})})
                             (F_{\b_{p'+1}}^{(c'_{p'+1})}\cdots F_{\b_p}^{(c'_p)})  \\
                        &= \Bigl(\sum_{\Be \ge \Bc}a_{\Be}L(\Be, \Bh)\Bigr)
                              (F_{\b_{p'+1}}^{(c'_{p'+1})}\cdots F_{\b_p}^{(c'_p)})
\end{align*}
by Proposition 2.3.
Then $L(\Be, \Bh)$ is of the form $F_{\b_k}^{(e_k)}\cdots F_{\b_{k'}}^{(e_{k'})}$.
If $k < k'$, the discussion in (a) for the case $s < t$ can be applied, and
for $L = F_{p'+1}^{(c'_{p'+1})}\cdots F_p^{(c'_p)}$,
$L(\Be, \Bh)L$ is a linear combination of
$L(\Bc'',\Bh)$ such that $\Bc'' \ge \Be$. 
Thus we may assume that $L(\Be,\Bh) = F_{\b_k}^{(e_k)}$ for some $k$, and
consider $F_{\b_k}^{(e_k)}L$, where $L$ is as above.
If $p'+1 < p$, by a similar discussion as above, using (2.5.3), one can eliminate
the term $F_{\b_{p'+1}}^{(c'_{p'+1})}$ from $L$. Hence by repeating this procedure, finally
the computation is reduced to the case where $L(\Be,\Bh) = F_{\b_k}^{(e_k)}$ and 
$L = F_{\b_{p'+1}}^{(c'_{p'+1})}\cdots F_{\b_p}^{(c'_p)} = F_{\b_p}^{(c'_p)}$. 
In this case, Proposition 2.3 can be applied.  Hence in any case, $L(\Be,\Bh)L$ is
a linear combination of $L(\Bc'',\Bh)$ such that $\Bc''\ge \Be$. 
Since $\Be \ge \Bc$, we have $\Bc'' \ge \Bc$.  Thus (i) holds for the case (b),
and (i) is proved. 
(ii) is clear from Proposition 2.3 and from the above discussion. 
\end{proof}  

\par\bigskip
\section{ Foldings of quantum groups }

\para{3.1.}
Let $X = (I, (\ ,\ ))$ be a Cartan datum of finite type as in 1.1.
We assume that $X$ is symmetric.
A permutation $\s : I \to I$ is called an admissible automorphism on $X$
if it satisfies the property that
$(\a_i, \a_{i'}) = (\a_{\s(i)}, \a_{\s(i')})$ for any $i,i' \in I$,
and that $(\a_i, \a_{i'}) = 0$ if $i \ne i'$ belong to the same $\s$-orbit in $I$.
Assume that $\s$ is admissible. 
We denote by $\Bn$ the order of $\s: I \to I$. 
Let $J$ be the set of
$\s$-orbits in $I$. For each $j \in J$, set $\a_j = \sum_{i \in j}\a_i$,
and consider the subspace $\bigoplus_{j \in J}\QQ \a_j$ of $E$ with basis $\a_j$.
We denote by $|j|$ the size of  the orbit $j$ in $I$.
The restriction of the form $(\ ,\ )$ on $\bigoplus_{j \in J}\QQ\a_j$ is given by

\begin{equation*}
(\a_j, \a_{j'}) = \begin{cases}
    (\a_i,\a_i)|j|, \quad (i \in j) &\quad\text{ if } j = j', \\
    \sum_{i \in j, i' \in j'}(\a_i, \a_{i'})   &\quad\text{ if } j \ne j'. 
                  \end{cases}
\end{equation*}
Then $\ul X = (J, (\ ,\ ))$ turns out to be a Cartan datum, which is called the Cartan datum
induced from $(X, \s)$.
Now  $\s$ acts naturally on the root lattice $Q$,
and the set $Q^{\s}$ of $\s$-fixed elements in $Q$
is identified with the root lattice $\ul Q$ of $\ul X$.
We have $Q^{\s}_+ = \sum_{j \in J}\NN \a_j$, which is identified with $\ul Q_+$.  
Let $\ul\vD^+$ be the set of positive roots associated to $\ul X$.
Then $\a_j \in Q_+^{\s}$ corresponds to a simple root in $\ul\vD^+$.
\par
The admissible automorphism $\s : I \to I$  induces 
an algebra automorphism on $\BU_q^-$ by $f_i \mapsto f_{\s(i)}$ for each $i \in I$,
which we also denote by $\s$.  The action of $\s$ leaves ${}_{\BA}\BU_q^-$ invariant, 
and we denote by ${}_{\BA}\BU_q^{-,\s}$ the fixed point subalgebra of ${}_{\BA}\BU_q^-$. 
\par
Let $\ul{\BU}_q^- = \BU_q^-(\ul X)$
be the negative half of the quantum group $\ul\BU_q$ associated to $\ul X$.
Thus $\ul\BU_q^-$ is an associative $\QQ(q)$-algebra, with generators $f_j \ (j \in J)$
and $q$-Serre relations as in 1.2.  The integral form ${}_{\BA}\ul\BU_q^-$ of $\ul\BU_q^-$
is the $\BA$-subalgebra generated by $f_j^{(n)}$ \ ($j \in J, n \in \NN$). 

\para{3.2.}
Let $\s: I \to I$ be as in 3.1, and assume that $\s$ is not the identity map. 
Hereafter, we assume that $X$ is irreducible.
Then $\Bn = 2$ or 3, which is a prime number $p$. Let $\FF = \ZZ/p\ZZ$
be the finite field of $p$ elements, and set $\BA' = \FF[q,q\iv] = \BA/p \BA$.
We consider the $\BA'$-algebra

\begin{equation*}
  {}_{\BA'}\BU_q^{-,\s} = \BA' \otimes_{\BA}{}_{\BA}\BU_q^{-,\s}
          \simeq {}_{\BA}\BU_q^{-,\s}/p({}_{\BA}\BU_q^{-,\s}).
\end{equation*}

Let $\BJ$ be an $\BA'$-submodule of $_{\BA'}\BU_q^{-,\s}$ generated by
$\sum_{0 \le i < p}\s^i(x)$. Then $\BJ$ is a two-sided ideal of
$_{\BA'}\BU_q^{-\s}$.  We define an $\BA'$-algebra $\BV_q$ as the quotient algebra
of ${}_{\BA'}\BU_q^{-,\s}$,

\begin{equation*}
\tag{3.2.1}  
\BV_q = {}_{\BA'}\BU_q^{-,\s}/\BJ.
\end{equation*}
Let $\pi : {}_{\BA'}\BU_q^{-,\s} \to \BV_q$ be the natural projection. 

\par
For each $j \in J$, and $a \in \NN$, set $\wt f^{(a)}_j = \prod_{i \in j}f_i^{(a)}$.
Since $\s$ is admissible, $f_i^{(a)}$ and $f_{i'}^{(a)}$ commute each other for $i,i' \in j$.
Hence $\wt f_j^{(a)}$ does not depend on the order of the product, and we have
$\wt f_j^{(a)} \in {}_{\BA}\BU_q^{-,\s}$.  We denote by the same symbol the image of
$\wt f_j^{(a)}$ on ${}_{\BA'}\BU_q^{-,\s}$. Thus $\pi(\wt f_j^{(a)}) \in \BV_q$ can be defined. 
In the case where $a = 1$, we set $\wt f_j^{(1)} = \wt f_j = \prod_{i \in j}f_j$.
\par
We define an $\BA'$-algebra ${}_{\BA'}\ul\BU_q^-$ by
${}_{\BA'}\ul\BU_q^- = \BA'\otimes_{\BA}{}_{\BA}\ul\BU_q^-$. 
\par
The following result holds.

\begin{thm}[{[SZ1]}] %%%%  Theorem 3.3
The assignment $f_j^{(a)} \to \pi(\wt f_j^{(a)})$ \ ($j \in J$)
gives an $\BA'$-algebra isomorphism $\Phi : {}_{\BA'}\ul\BU_q^- \simeq \BV_q$.
\end{thm}  

\par\noindent
\remark{3.4.}
The statement as in Theorem 3.3 was proved by \cite[Thm. 0.4]{SZ1} in the case of
finite type, and was generalized in \cite{SZ2} to the affine case, by using
the PBW-bases of $\BU_q^-$. Finally, it was proved in \cite{MSZ1} for the
case of Kac-Moody type, in general, by making use of canonical bases
in $\BU_q^-$.  In \cite{SZ1},\cite{SZ2}, the existence of canonical bases is
not assumed. 

\para{3.5.}
Let $W$ be the Weyl group generated by $\{ s_i \mid i \in I\}$ associated to $X$,
and $\ul{W}$ the Weyl group generated by $\{ s_j \mid j \in J\}$ associated to $\ul{X}$.
Let $\s : W \to W$ be the automorphism of $W$ defined by $s_i \mapsto s_{\s(i)}$, and
$W^{\s}$ the fixed point subgroup of $W$ by $\s$.
For $j \in J$, set $w_j = \prod_{i \in j}s_i$. Since $\s$ is admissible, $w_j$
is independent from the order of the product, and so $w_j \in W^{\s}$. 
It is known that the correspondence $s_j \mapsto w_j$ gives an isomorphism of groups
$\ul W \isom W^{\s}$.
Let $w_0$ be the longest element in $W$ with $l(w_0) = N$, and $\ul{w}_0$ the longest
element of $\ul{W}$ with $l(\ul{w}_0) = \ul N$. 
Let $\ul{w_0} = s_{j_1}\cdots s_{j_{\ul N}}$ be a reduced expression of $\ul{w}_0$, then
we have $w_0 = w_{j_1}\cdots w_{j_{\ul N}}$ and $\sum_{1 \le k \le \ul N}l(w_{j_k}) = N$,
where $w_{j_k} = \prod_{i \in j_k}s_i$. 
It follows that
\begin{equation*}
\tag{3.5.1}  
w_0 = \biggl(\prod_{k_1 \in j_1}s_{k_1}\biggr)\cdots
          \biggl(\prod_{k_{\ul N} \in j_{\ul N}}s_{k_{\ul N}}\biggr) = s_{i_1}\cdots s_{i_N}
\end{equation*}
gives a reduced expression of $w_0$.
For a given reduced sequence $\ul \Bh = (j_1, \dots, j_{\ul N})$ of $\ul{w}_0$,
we define a reduced sequence $\Bh = (i_1, \dots, i_N)$ of $w_0$ by (3.5.1), i.e.,
\begin{equation*}
\tag{3.5.2}  
  \Bh = (\underbrace{i_1, \dots, i_{|j_1|}}_{j_1},
         \underbrace{i_{|j_1|+1}, \dots, i_{|j_1|+|j_2|}}_{j_2},
            \dots, \underbrace{i_{|j_1| + \cdots + |j_{\ul N-1}|+1}, \dots, i_N}_{j_{\ul N}}).
\end{equation*}  
(Although the expression of $w_j$ is not unique, we ignore the difference of
mutually commuting factors.)

\para{3.6.}
Take reduced sequences $\Bh$ of $w_0$, and $\ul\Bh$ of $\ul{w}_0$.  
Then $\ul\Bh$ gives a total order of $\ul\vD^+$ as
$\ul\vD^+ = \{ \ul\b_1, \dots, \ul\b_{\ul N} \}$.   Let
\begin{align*}
\SX_{\Bh} &= \{ L(\Bc, \Bh) \mid \Bc \in \NN^N\}, \\
\SX_{\ul\Bh} &= \{ L(\ul\Bc, \ul\Bh) \mid \ul\Bc \in \NN^{\ul N} \}
\end{align*}
be the PBW-basis of $\BU_q^-$
associated to $\Bh$, and the PBW-basis of $\ul\BU_q^-$ associated to $\ul\Bh$.
Here we have

\begin{align*}
L(\Bc,\Bh) &= F_{\b_1}^{(c_1)}F_{\b_2}^{(c_2)}\cdots F_{\b_N}^{(c_N)}, \\
L(\ul\Bc, \ul\Bh) &= F_{\ul\b_1}^{(\ul c_1)}F_{\ul\b_2}^{(\ul c_2)}
\cdots F_{\ul\b_{\ul N}}^{(\ul c_{\ul N})},
\end{align*}
where $F_{\ul\b_k}^{(\ul c_k)} = T_{j_1}\cdots T_{j_{k-1}}(f_{j_k}^{(\ul c_k)})$
and $T_j$ ($j \in J)$ is a braid group action on $\ul\BU_q$. 

Now choose $\Bh, \ul\Bh$ as in 3.5.  Following the expression in (3.5.2),
we write $\Bc = (c_1, \dots, c_N) \in \NN^N$ as $\Bc = (\Bc^{(1)}, \dots, \Bc^{(\ul N)})$,
where $\Bc^{(k)} \in \NN^{|j_k|}$, namely,
\begin{equation*}
  \Bc = (\underbrace{c_1, \dots, c_{|j_1|}}_{\Bc^{(1)}},
          \underbrace{c_{|j_1|+1}, \dots, c_{|j_1| + |j_2|}}_{\Bc^{(2)}}, 
          \dots, \underbrace{c_{|j_1| + \cdots + |j_{\ul N-1}|+1},\dots, c_N}_{\Bc^{(\ul N)}}).
\end{equation*}  
Corresponding to the action of $\s$ on the orbit
$j_k \subset I$,
we define an action of $\s$ on $\Bc^{(k)}$ as a permutation of coordinates,
thus the action of $\s$ on $\NN^N$
can be defined. Let $\NN^{N,\s}$ be the set of $\s$-fixed elements in $\NN^N$.
For $\ul\Bc = (\ul c_1, \dots, \ul c_{\ul N}) \in \NN^{\ul N}$, define
$\Bc = (\Bc^{(1)}, \dots, \Bc^{(\ul N)}) \in \NN^N$ by the condition that
$\Bc^{(k)} = (\ul c_k, \dots, \ul c_k) \in \NN^{|j_k|}$ for each $k$.  Then
$\Bc \in \NN^{N,\s}$, and the assignment $\ul\Bc \mapsto \Bc$ gives a bijection
$\NN^{\ul N} \isom \NN^{N,\s}$. 

\par
The following result was proved in \cite[Thm. 1.14]{SZ1} in the course of the proof
of Theorem 0.4 in \cite{SZ1}.

\begin{prop}  %%%%  Prop. 3.7
Under the setup in 3.6, we have
\begin{enumerate}
\item \ $\s$ acts on $\SX_{\Bh}$ as a permutation,
$\s : L(\Bc,\Bh) \mapsto L(\s(\Bc), \Bh)$.
$L(\Bc, \Bh)$ is $\s$-invariant if and only if $\Bc \in \NN^{N,\s}$.
\item \ Under the bijection $\NN^{\ul N} \isom \NN^{N,\s}, \ul\Bc \mapsto \Bc$,
the assignment $L(\ul\Bc, \ul\Bh) \to L(\Bc, \Bh)$ gives a bijection
\begin{equation*}
\SX_{\ul\Bh} \isom \SX_{\Bh}^{\s},
\end{equation*}
where $\SX_{\Bh}^{\s}$ is the set of $\s$-fixed PBW bases in $\SX_{\Bh}$. 
\item \ The bijection in (ii) is compatible with the isomorphism $\Phi$ in
Theorem 3.3, namely on $\BV_q$, we have
\begin{equation*}  
\Phi(L(\ul\Bc, \ul\Bh)) = \pi(L(\Bc, \Bh))
\end{equation*}
\item \ Similar statements as in (i) $\sim$ (iii) also hold for 
  the canonical basis $\bB_{\Bh}$ on $\BU_q^-$ and $\bB_{\ul\Bh}$ on $\ul\BU_q^-$.  
\end{enumerate}  
\end{prop}

\para{3.8.}
The following result is a special case of Porposition 3.7 (iii), and is worth
for mentioning. 
From the relations between $\Bh$ and $\ul\Bh$, we have a partition
\begin{equation*}
\tag{3.8.1}
\vD^+ = \{ \underbrace{\b_1, \dots, \b_{|j_1|}}_{\ul\b_1\text{-part}},
             \underbrace{\b_{|j_1|+1}, \dots, \b_{|j_1| + |j_2|}}_{\ul\b_2\text{-part}},
             \dots, \underbrace{\b_{|j_1|+ \cdots + |j_{\ul N-1}| +1},
                          \dots, \b_N}_{\ul\b_{\ul N}\text{-part}}\}.
\end{equation*}  

Write the $\ul\b_k$-part of $\vD^+$ as $\b_s, \dots, \b_t$.
Then $F_{\b_s}^{(c_s)}, \dots, F_{\b_t}^{(c_t)}$ are mutually commuting.
Assume that 
$F_{\b_s}^{(c_s)}\cdots F_{\b_t}^{(c_t)}$ is $\s$-stable.  Then 
$(c_s, \dots, c_t)$ is written as $(\ul c_k, \dots, \ul c_k)$ ($|j_k|$-times) for
some $\ul c_k$,  and we have
\begin{equation*}
\tag{3.8.2}  
\Phi(F_{\ul\b_k}^{(\ul c_k)}) = \pi(F_{\b_s}^{(c_s)}\cdots F_{\b_t}^{(c_t)}).
\end{equation*}

\par\bigskip
\section{ Monomial bases }

\para{4.1.}  
In the case where the Cartan datum $X$ is symmetric, a monomial basis of
$\BU_q^-$ was constructed by Lusztig \cite{L-can} in a uniform way,
by making use of the representation theory of quivers.  For $\BU_q^-$ of general type,
a monomial basis was constructed by Chari-Xi \cite{CX}, but the relationship with Lusztig's basis
is not clear, and their description for the exceptional case is rather complicated.
(For example, Lusztig's monomial basis (and our basis) consists of monomials of the form 
$f_{i_1}^{(c_1)}\cdots f_{i_k}^{(c_k)}$ such that $k = nN$ with $n = |I|$ and $N = |\vD^+|$,
while, the basis of Chari-Xi consists of monomials as above such that $k = N$.) 
In this section, we construct a monomial basis for non-symmetric $X$, by mimicking
Lusztig's construction, and by using Proposition 3.7, which has a good connection 
with Lusztig's monomial basis through the folding theory.

\para{4.2.}
Assume that $X = (I, (\ ,\ ))$ is symmetric.  Let $\ora{Q} = (I,\Om )$ be a quiver,
namely, an oriented graph, where $I$ is the set of vertices, and $\Om$ is the set of edges,
with a collection of arrows $i \to j$ for each edge $(i,j)$.
If $I$ is fixed, we just say $\Om$ is a quiver on $I$. 
A vertex $i \in I$ is called a sink if
there is no arrow $i \to j$, and is called a source if there is no arrow $j \to i$.
Assume that $i$ is a sink or a source. The quiver $s_i\Om$ is defined by reversing
the arrow to $i$ or from $i$.
Let $\Bh = (i_1, \dots, i_N)$ be a reduced sequence of $w_0 \in W$
as given in 1.4.
We say that $\Bh = (i_1, \dots, i_N)$
is adapted to $\Om$ if $i_k$ is a sink of $\Om_k$ for $k = 1,2, \dots, N$,
where $\Om_k = s_{i_{k-1}}s_{i_{k-2}}\cdots s_{i_2}s_{i_1}\Om$.
It is known by \cite [Prop. 4.12]{L-can}, for a given orientation $\Om$, there exists
some $\Bh$ such that $\Bh$ is adapted to $\Om$. 

\para{4.3.}
Let $\s : I \to I$ be an admissible automorphism, and $\ul X = (J, (\ ,\ ))$
the induced Cartan datum as given in 3.1.  We assume that $X$ is irreducible.
Let $\ul\Bh$ be a reduced sequence for $\ul w_0$, and $\Bh$ the reduced sequence for
$w_0$ obtained from $\ul\Bh$ as in 3.5.  We shall find $\ul\Bh$ and an orientation $\Om$ such that
$\Bh$ is adapted to $\Om$. 
\par
First we give a general remark. 
We consider a Dynkin diagram of symmetric type.
Let $I = I_0 \sqcup I_1$ be a partition of $I$ such that for any edge, one of
the endpoints will be in $I_0$ and the other in $I_1$.
Set $c_0 = \prod_{i \in I_0}s_i, c_1 = \prod_{i \in I_1}s_i$ (in any order).
Then $c_0c_1, c_1c_0$ give Coxeter elements in $W$. The following result was
proved by Steinberg (see \cite [3.17]{H}).
\par\medskip\noindent
(4.3.1) \ Let $c_0, c_1$ be as above. Let $h$ be the Coxeter number. Then 
\begin{align*}
w_0 &= \cdots c_0c_1c_0 \ \text{($h$-factors)} \\
    &= \cdots c_1c_0c_1 \ \text{($h$-factors)}
\end{align*}  
are reduced expressions for the longest element $w_0 \in W$.
In particular, if $h$ is even, then $w_0 = c^{h/2}$ gives a reduced
expression of $w_0$ for a Coxeter element $c = c_0c_1$ or $c = c_1c_0$. 
\par\medskip
By Kirillov \cite[3.33]{Ki}, we have the following.
\par\medskip\noindent
(4.3.2) \ Define an orientation of $\Om$ so that
$i \in I_0$ is a sink, and $i \in I_1$ is a source.
Assume that $h$ is even.  Then the reduced expression
$w_0 = c_0c_1\cdots c_0c_1$ is adapted to $\Om$. 

\par\medskip
The irreducible Cartan datum of symmetric type is
$X = A_{2n-1}, D_n, E_6$ with $\Bn = 2$, or $D_4$ with $\Bn = 3$. 
In that case, $\ul X = B_n, C_{n-1}, F_4$ or $G_2$. 
We consider each case separately.  Let $J$ be the set of $\s$-orbits in $I$.
In the discussion below, we consider a condition on the total order of $I$,
\par\medskip\noindent
(*) \  The total order $i < j$ on $I$ satisfies the property 
   `` $i \to j$ implies that $j < i$ ''. 

\par\medskip

(A) The case $A_{2n-1}$.
We fix a labeling of $I$ such that
\begin{equation*}
  I = \{ 1,2, \dots, n-1, n, (n-1)', \dots, 2',1'\},
\end{equation*}
where
$\s: i \lra i'$ and $\s(n) = n$.
Then $J$ is given by $J = \{ \ul 1, \ul 2, \dots, \ul{n-1}, \ul n\}$,
where $\ul 1 = \{1,1'\}, \ul 2 =  \{ 2,2'\} , \dots \ul{n-1} = \{ n-1, (n-1)'\},
\ul n = \{ n \}$. 
Let
\begin{align*}
\tag{4.3.3}
I_0 = \{ 1,1', 3,3', \dots \}, \qquad I_1 = \{ 2,2', 4,4', \dots\},
\end{align*}
where $n \in I_0$ (resp. $n \in I_1$) if $n$ is odd (resp. $n$ is even).  
We define an orientation on $\Om$ as in (4.3.2). 
Let $c_0 = \prod_{i \in I_0}s_i, c_1 = \prod_{i \in I_1}s_i$ (in any order).
Then $c = c_0c_1$ gives a Coxeter element of $W = W(A_{2n-1})$.
We choose a total order on $I_0$ and $I_1$ as in (4.3.3), where $n$ is the last element
in $I_0$ or $I_1$, and define a total order on $I$ by $I = \{ i_1, \dots, i_{2n-1} \}$
by juxtaposition $I = I_0\sqcup I_1$. 
This total order satisfies the property (*). 
\par
Let $h$ be the Coxeter number of $W$.
Then $h = (2n-1) + 1 = 2n$ is even.  Hence $w_0 = c^{h/2}$ gives a reduced
expression of $w_0 \in W$ by (4.3.1).
We have $|\vD^+(A_{2n-1})| = n(2n-1)$, and $h/2 = n$. 
Set
\par\smallskip
\begin{equation*}
\tag{4.3.4}  
  \Bh = (\underbrace{1,1',3,3', \dots, 2,2',4,4'\dots}_{\text{corresp. to } c},
             \dots, \underbrace{1,1',3,3',\dots, 2,2',4,4' \dots}_{\text{corresp. to } c})
            \quad \text{ ($h/2$-times)}.
\end{equation*}
Then $\Bh$ gives a reduced sequence of $w_0$, which is adapted to $\Om$.
We have

\begin{equation*}
\tag{4.3.5}  
  \ul\Bh = (\underbrace{\ul 1, \ul 3, \dots, \ul 2, \ul 4\dots}_{\text{corresp. to }c},
             \dots, \underbrace{\ul 1, \ul 3,\dots, \ul 2, \ul 4, \dots}_{\text{corresp. to }c})
            \quad \text{ ($h/2$-times)}.
\end{equation*}

\par\medskip
(B) \ The case $D_n$. 
We fix a labeling of $I$ as
$I = \{ 1,2,\dots, n-2, n, n' \}$, where $\s : n \lra n'$, and
$\s(i) = i$ for $i = 1, \dots, n-2$.
Then $J = \{ \ul 1, \ul 2, \dots, \ul{n-2}, \ul{n-1}\}$,
where $\ul i = \{i \}$ for $i = 1, \dots, n-2$, and $\ul{n-1} = \{ n, n'\}$. 
We define an orientation
as in (4.3.2), where $I_0 = \{ 1,3, \dots, \}$ and $I_1 = \{ 2,4, \dots\}$.
Then $n, n'$ appear in the end of the sequence $I_0$ if $n$ is even,
and appear in the end of the sequence $I_1$ if $n$ is odd. 

We define a total order on $I = I_0 \sqcup I_1$ by
\begin{equation*}
I = \{ 1,3,\dots, 2,4, \dots \}. 
\end{equation*}  
Then this total order on $I$ satisfies the property (*).
Let $c = c_0c_1$ be a Coxeter element in $W = W(D_n)$.
Since $h = 2n-2$ is even, we have $w_0 = c^{h/2} = c^{n-1}$ gives a reduced expression of
$w_0$.
Set
\begin{equation*}
\tag{4.3.6}  
  \Bh = (\underbrace{1,3,\dots, 2,4,\dots}_{\text{corresp. to }c},
                  \dots, \underbrace{1,3,\dots 2,4, \dots}_{\text{corresp. to } c})
                \quad \text{ ($h/2$-times)}.
\end{equation*}
Then $\Bh$ gives a reduced sequence of $w_0$, which is adapted to $\Om$. 
We have
\begin{equation*}
\tag{4.3.7}  
  \ul\Bh = (\underbrace{\ul 1,\ul 3, \dots, \ul 2,\ul 4, \dots}_{\text{corresp. to }c}, \dots,
             \underbrace{\ul 1, \ul 3,\dots \ul 2, \ul 4, \dots}_{\text{corresp. to }c})
                \quad \text{ ($h/2$-times)},
\end{equation*}
where $\ul {n-1}$ appears in the end of the sequence $\ul 1, \ul 3, \dots$ if $n$ is even, and
appears in the end of the sequence $\ul 2, \ul 4, \dots$ if $n$ is odd. 

\par\medskip
(C) \ The case $E_6$.  We define a labeling of $I$ as $I = \{ 1,2,3,2',1',4\}$
as in (4.3.8). 
We have $\s : 1 \lra 1', 2 \lra 2',$ and $\s(3) = 3, \s(4) = 4$.
Then $J = \{ \ul 1, \ul 2, \ul 3, \ul 4\}$, where
$\ul 1 = \{ 1,1'\}, \ul 2 = \{ 2,2'\}, \ul 3 = \{ 3\}, \ul 4 = \{ 4\}$. 
We define an orientation $\Om$ as in (4.3.2), where
$I_0 = \{ 1,1',3\}, I_1 = \{ 2,2',4\}$.  Hence we have 

\begin{equation*}
\tag{4.3.8}
  \xymatrix@C=20pt@ R=16pt@ M =6pt{
   &                  &  4  \ar[d]  &                    &  \\  
1  &  2 \ar[l] \ar[r] &  3          &  2' \ar[l] \ar [r] &  1'   
}
\end{equation*}

Define a total order on $I = I_0 \sqcup I_1$ by the total order of $I_0, I_1$ as above, 
\begin{equation*}
I = \{ 1,1',3,2,2',4 \}.
\end{equation*}  
This total order satisfies the property (*). 
Let $c = s_1s_{1'}s_3s_2s_{2'}s_4$ be a Coxeter element in $W$.
Since $h = 12$ is even, we see that $w_0 = c^{h/2}$ gives a reduced expression
of $w_0 \in W(E_6)$. Set
\begin{equation*}
\tag{4.3.9}  
  \Bh = (\underbrace{1,1',3,,2,2',4}_{\text{corresp. to }c},
             \dots \underbrace{1,1',3,2,2',4}_{\text{corresp. to }c})
           \quad \text{ ($h/2$-times).}
\end{equation*}
Then $\Bh$ gives a reduced expression of $w_0$, which is adapted to $\Om$. 
We have
\begin{equation*}
\tag{4.3.10}  
  \ul\Bh = (\underbrace{\ul 1, \ul 3, \ul 2, \ul 4}_{\text{corresp. to }c},
             \dots \underbrace{\ul 1, \ul 3, \ul 2, \ul 4}_{\text{corresp. to }c})
           \quad \text{ ($h/2$-times).}
\end{equation*}

\par\medskip
(D) \ The case $D_4$.  We fix a labeling of $I$ as $I = \{ 1,1',1'',2\}$,
where $\s: 1 \mapsto 1' \mapsto 1'' \mapsto 1$, and $\s(2) = 2$.
Then $J = \{ \ul 1, \ul 2\}$ with $\ul 1= \{ 1,1',1''\}, \ul 2 = \{ 2\}$. 
We define an orientation $\Om$ as in (4.3.2), where
$I = I_0 \sqcup I_1$ with $I_0 = \{ 1,1',1''\}$ and $I_1 = \{ 2\}$. 
Hence
\begin{equation*}
\tag{4.3.11}
  \xymatrix@C=20pt@ R=16pt@ M =6pt{
         &  1'  &   \\
1  &  2 \ar[l] \ar[r] \ar[u] &  1''   
}
\end{equation*}

Define a total order on $I = I_0 \sqcup I_1$ by the total order of $I_0$ and $I_1$
as above, hence we have 
\begin{equation*}
I = \{ 1,1',1'',2\}.
\end{equation*}   
This total order satisfies the property (*).
Let $c = s_1s_{1'}s_{1''}s_2$ be a Coxeter element in $W(D_4)$.
Since $h = 6$ is even, $w_0 = c^{h/2} = c^3$ gives a reduced expression of $w_0 \in W(D_4)$.
Set
\begin{equation*}
\tag{4.3.12}  
  \Bh = (\underbrace{1,1',1'',2}_{\text{corresp. to }c},
          \underbrace{1,1',1'',2}_{\text{corresp. to }c},
         \underbrace{1,1',1'',2}_{\text{corresp. to }c}).
\end{equation*}  
Then $\Bh$ gives a reduced sequence of $w_0$, which is adapted to $\Om$. 
We have

\begin{equation*}
\tag{4.3.13}  
  \ul\Bh = (\ul 1, \ul 2, \ul 1, \ul 2, \ul 1, \ul 2).
\end{equation*}  

\para{4.4.}
Under the setting in 4.2, 
let $\BV = \bigoplus_{i \in I}V_i$ be an $I$-graded vector space over
an algebraically closed field $\Bk$. 
The dimension vector $\dim \BV = \Bd = (d_i)_{i \in I}$ is defined by $d_i = \dim V_i$. 
Let
\begin{equation*}
\BE_{\BV} = \bigoplus_{i \to j}\Hom (V_i, V_j)
\end{equation*}
be the representation space associated to $\BV$.  Then $G_{\BV} = \prod_{i \in I}GL(V_i)$
acts on $\BE_{\BV}$.   Each $x \in \BE_{\BV}$ corresponds to a representation of
$\ora{Q}$, and the isomorphism classes of representations correspond to the
$G_{\BV}$-orbits in $\BE_{\BV}$.  Now the isomorphism classes of
those representations are in bijection with
$\Bc = (c_1, \dots, c_N)\in \NN^N$ such that $\sum_{k=1}^Nc_k\b_k = \sum_{i \in I}d_i\a_i$.  
We denote by $\SO_{\Bc}$ the orbit in $\BE_{\BV}$ corresponding to $\Bc$. 
\par
First we construct a monomial basis of
$\BU_q^-$ following \cite{L-can}. Let $\Bh = (i_1, \dots, i_N)$ be a reduced sequence
of $w_0 \in W$ such that $\Bh$ is adapted to $\Om$.
$\Bh$ determines a total order of $\vD^+$ such that $\vD^+ = \{ \b_1, \dots, \b_N\}$.
Now take a simple root $\a_i$.  Then $\a_i$ appears in the sequence $(\b_1, \dots, \b_N)$.
The following important property was proved in \cite [4.14 (b)]{L-can}.
\par\medskip\noindent
(4.4.1) \  Assume that $\Bh$ is adapted to $\Om$. If $i \to j$ in $\Om$,
then $\a_j$ precedes $\a_i$ in the sequence $(\b_1, \dots, \b_N)$, namely, we have
$(\b_1, \dots, \b_N) = (\dots, \a_j, \dots, \a_i, \dots)$. 
\par\medskip
We fix a total order of $I$ as $I = \{ 1, 2, \dots, n\}$ which satisfies
the condition (*)  in 4.3, namely, $i \to j$ implies that $j < i$. 
\par
Recall that PBW basis $L(\Bc, \Bh) \in \SX_{\Bh}$
is given as $L(\Bc, \Bh) = F_{\b_1}^{(c_1)}\cdots F_{\b_N}^{(c_N)}$
for $\Bc = (c_1, \dots, c_N) \in \NN^N$.  We have $\weit(L(\Bc,\Bh)) = -\sum_{k=1}^Nc_k\b_k$.  
We fix $\Bd = (d_1, \dots, d_n) \in \NN^n$, and consider the representation space
$\BE_{\BV}$ such that $\dim \BV = \Bd$. We define a ``monomial'' 
\begin{equation*}
\tag{4.4.2}  
F(\Bd) = f_n^{(d_n)}f_{n-1}^{(d_{n-1})}\cdots f_1^{(d_1)}. 
\end{equation*}

The following result was proved in Lemma 7.4 and Proposition 7.7 in \cite{L-can}.

\begin{lem}  %%%%  Lemma 4.5.
The monomial $F(\Bd)$ is written as
\begin{equation*}
F(\Bd) = \sum_{\Bc}q^{\d(\Bc)}L(\Bc, \Bh),  
\end{equation*}  
where the sum is taken over $\Bc \in \NN^N$ such that
$\SO_{\Bc} \subset \BE_{\BV}$.  Here $\d(\Bc) = \dim \BE_{\BV} - \dim \SO_{\Bc}$,
and $\d(\Bc) = 0$ for exactly one element $\Bc$. 
\end{lem}

\para{4.6.}
Assume given $\Bc = (c_1, \dots, c_N) \in \NN^N$.  For $k \in [1,N]$,
define $\Bc_k \in \NN^N$ by the condition that the $k$-th coordinate of $\Bc_k$
coincides with that of $\Bc$, and all other coordinates are zero, namely,
$\Bc_k = (0, \dots, 0, c_k, 0, \dots, 0)$. 
Define $\Bd^k = (d_1^k, \dots, d_n^k) \in \NN^n$ by the condition that
$c_k\b_k = \sum_{j = 1}^nd_j^k\a_j$.
Let
\begin{equation*}
\tag{4.6.1}  
F((\Bc)) = F(\Bd^1)F(\Bd^2)\cdots F(\Bd^N)
\end{equation*}  
where
\begin{equation*}
\tag{4.6.2}  
F(\Bd^k) = f_n^{(d_n^k)}f_{n-1}^{(d_{n-1}^k)}\cdots f_1^{(d_1^k)}.
\end{equation*}  

Note that $\weit(F(\Bd^k)) = -\sum_{j=1}^n d_j^k\a_j = -c_k\b_k$,
hence
\begin{equation*}
  \weit(F((\Bc))) = -\sum_{k = 1}^N c_k\b_k = \weit(L(\Bc,\Bh)).
\end{equation*}  
The following result was proved in \cite [7.8]{L-can}.

\begin{thm}  %%%%   Theorem 4.7 
\begin{enumerate}
\item 
For each $\Bc \in \NN^N$, $F((\Bc))$ can be written as
\begin{equation*}
  F((\Bc)) = L(\Bc, \Bh) + \sum_{\Bc' > \Bc}h_{\Bc',\Bc}L(\Bc',\Bh)
                 \quad\text{ with  }  \quad h_{\Bc,\Bc'} \in \BA,
\end{equation*}
where $\Bc'$ satisfies the condition $\weit(L(\Bc',\Bh)) = \weit(L(\Bc,\Bh))$.
\\
Hence  $\{ F((\Bc)) \mid \Bc \in \NN^N\}$
gives an $\BA$-basis of ${}_{\BA}\BU_q^-$ and a $\QQ(q)$-basis of $\BU_q^-$. 
\item
We set $\SM_{\Bh} = \{ m(\Bc,\Bh) \mid \Bc \in \NN^N\}$, with $m(\Bc, \Bh) = F((\Bc))$.  
Then $\SM_{\Bh}$ is a monomial basis of $\BU_q^-$. 
\end{enumerate}
\end{thm}

\remark{4.8.}
In the original form of Theorem 4.7, Lusztig uses
the condition that $\dim \SO_{\Bc'} < \dim \SO_{\Bc}$ instead of $\Bc < \Bc'$.
We note that
\par\medskip\noindent
(4.8.1) \ If $\dim \SO_{\Bc'} <  \dim \SO_{\Bc}$, then we have $\Bc < \Bc'$. 
\par\medskip
In fact, it is proved in \cite[Prop. 7.9]{L-can} that
\begin{equation*}
  \ol{L(\Bc,\Bh)} =  L(\Bc,\Bh) + \sum_{\Bc'}\w_{\Bc'}^{\Bc}L(\Bc',\Bh)
     \quad\text{ with }\quad \w_{\Bc'}^{\Bc} \in \BA, 
\end{equation*}
where 
$\Bc'$ runs over the elements in $\NN^N$ such that $\dim \SO_{\Bc'} < \dim \SO_{\Bc}$,
and $\w^{\Bc}_{\Bc'} \ne 0$ occurs exactly when $\dim \SO_{\Bc'} < \dim \SO_{\Bc}$.
On the other hand, it is known by (1.5.1) that
\begin{equation*}
\ol{L(\Bc,\Bh)} = L(\Bc,\Bh) + \sum_{\Bc' > \Bc}a_{\Bc',\Bc}L(\Bc,\Bh).
\end{equation*}  
(Here it may happen that $a_{\Bc',\Bc} = 0$.)
Comparing these two formulas, we obtain (4.8.1). 

\para{4.9.}
Let $\ul\BU_q^-$ be associated to $\ul X = (J, (\ ,\ ))$.
We want to consider an analogue of Theorem 4.7 for $\ul\BU_q^-$.
But the representation theory of quivers cannot be applied directly.
In the following, by modifying the discussion in \cite{L-can}, and by applying the
folding theory, we shall construct monomial bases of $\ul\BU_q^-$.
We follow the setup in 4.3.  Then $\Bh, \ul\Bh$ can be defined, and
an orientation $\Om$ is determined as in 4.3.  Let $I = \{ 1,2, \dots, n\}$
be the total order of $I$ as given in 4.3.  Then $\Bh$ is adapted to $\Om$,
and the total order on $I$ satisfies the condition (*) in 4.3.
The total order of $I$ induces a total order on $J$, which we write as
$J = \{ 1, 2, \dots, m \}$, where $m = |J|$.
The generators of $\ul\BU_q^-$ are denoted by $\ul f_j$ for $j \in J$. 
\par
Assume that $j \ne j' \in J$ are joined in the Dynkin diagram of $\ul X$.
Since the orientation $\Om$ constructed in 4.3 is compatible with the action of $\s$,
one can define an arrow $j \to j'$
by the condition that $i \to i'$ for some $i \in j, i' \in j'$ such that $i$ and $i'$ are
joined in the Dynkin digram of $X$.
\par
Assume that $j = \{ i_1, \dots, i_k\}$.  Then $\a_{i_1}, \dots, \a_{i_k}$ appears
successively in the sequence $\b_1, \dots, \b_N$
(the order on $\a_{i_1}, \dots, \a_{i_k}$ is ignored).
By applying (4.4.1), we have the following.
\par\medskip\noindent
(4.9.1) \  Assume that $j' \to j$ in $J$, with
$j = \{ i_1, \dots, i_k\}, j' = \{ i'_1, \dots, i'_{\ell}\}$.
Then in the sequence $\b_1, \dots, \b_N$, we have
\begin{equation*}
  \dots, \underbrace{\a_{i_1}, \dots, \a_{i_k}}_{j\text{-parts}} \dots,
     \underbrace{\a_{i'_1}, \dots, \a_{i'_{\ell}}}_{j'\text{-parts}}, \dots.
\end{equation*}  

For the proof of (4.9.1), we use a partition of $\vD^+$ as given in (3.8.1).
In this partition, if $\ul\b_k$ is a simple root in $\ul \vD^+$,
then the $\ul\b_k$-part in $\vD^+$ consists  of simple roots in $\vD^+$, and any simple
root in $\vD^+$ belongs to some $\ul\b_k$-part for a simple root $\ul\b_k \in \ul\vD^+$.

\par
For any $\ul\Bd = (\ul d_1, \dots, \ul d_m) \in \NN^m$, we define a monomial
$F(\ul\Bd) \in \ul\BU_q^-$ by 
\begin{equation*}
\tag{4.9.2}  
F(\ul\Bd) = \ul f_{m}^{(\ul d_m)}\ul f_{m-1}^{(\ul d_{m-1})}\dots \ul f_1^{(\ul d_1)}. 
\end{equation*}  

\par
Recall that $\SX_{\ul\Bh} = \{ L(\ul\Bc, \ul\Bh) \mid \ul\Bc \in \NN^{\ul N}\}$
is a PBW basis of $\ul\BU_q^-$.  We show a lemma

\begin{lem}  %%%%  Lemma 4.10.
$\ul f_1^{(\ul d_1)}\ul f_2^{(\ul d_2)}\cdots \ul f_m^{(\ul d_m)}$ coincides with
$L(\ul\Be, \ul\Bh)  \in \SX_{\ul\Bh}$,    
where $\ul\Be = (\ul e_1, \dots, \ul e_{\ul N})$ 
is a sequence such that $\ul d_1, \dots, \ul d_m$ appears in $\ul\Be$ in this order,
and all other coordinates in $\ul\Be$ are zero. 
\end{lem}  

\begin{proof}
Let $\ul\b_1, \dots, \ul\b_{\ul N}$ be the total order of $\ul\vD^+$
obtained from $\ul\Bh$. Since $\ul\b_k$ is defined by
$\ul\b_k = s_{j_1}\cdots s_{j_{k-1}}(\a_{j_k})$, there exists a sequence
$l(1) < l(2) < \cdots < l(m)$ in $[1,\ul N]$, and a permutation
$p$ on $[1, m]$ such that $\a_j$ ($j \in J$) appears in the sequence
$\ul\b_1, \dots, \ul\b_{\ul N}$
along the following order 
\begin{equation*}
\tag{4.10.1}
\dots, \a_{p(1)}, \dots, \a_{p(2)}, \dots \dots, \a_{p(m)}, \dots, 
\end{equation*}
where $\a_{p(k)}$ coincides with $\ul\b_{l(k)}$ for $k = 1, \dots,m$. 
Thus for any $\ul\Be \in \NN^{\ul N}$, we have 
\begin{equation*}
\tag{4.10.2}  
\ul f_{p(1)}^{(\ul e_{l(1)})}\cdots \ul f_{p(m)}^{(\ul e_{l(m)})} =   
F_{\ul\b_{l(1)}}^{(\ul e_{l(1)})}\cdots F_{\ul\b_{l(m)}}^{(\ul e_{l(m)})} \in \SX_{\ul\Bh}.
\end{equation*}  
We consider the corresponding sequence in $\b_1, \dots, \b_N$ of $\vD^+$.
Then $\ul\b_{l(k)}$-part in $\b_1, \dots, \b_N$ consists of $\a_{s_k},\dots, \a_{t_k}$ of
successive simple roots in $\vD^+$, mutually orthogonal,
and by (4.9.1), they are given by the following order
\begin{equation*}
\tag{4.10.3}
  \dots, \underbrace{\a_{s_1}, \dots, \a_{t_1}}_{\ul\b_{l(1)}\text{-part}},
       \dots, \underbrace{\a_{s_2}, \dots, \a_{t_2}}_{\ul\b_{l(2)}\text{-part}},
       \dots\dots, \underbrace{\a_{s_m}, \dots, \a_{t_m}}_{\ul\b_{l(m)}\text{-part}}, \dots. 
\end{equation*}  
This implies, by (4.4.1), that if $i$ belongs to $\ul\b_k$-part,
and $i'$ belongs to $\ul\b_{k'}$-part
with $k \ne k'$, and if $i \to i'$, then $\a_{i'}$ precedes $\a_i$. 
Now in the sequence (4.10.1), if $p(s) \to p(t)$ with $s < t$,
then there exist $i$ belonging to
$\ul\a_{p(s)}$-part, and $i'$ belonging to $\ul\a_{p(t)}$-part
such that $i \to i'$, and so $\a_{i'}$ precedes $\a_i$ by (4.9.1). This is a contradiction.
Hence either $p(s) \to p(t)$ with $t < s$, or $p(s)$ and $p(t)$
are not joined.
If $p(s)$ and $p(t)$ are not joined, then $\ul f_{p(s)}^{(\ul e_{l(s)})}$
and $\ul f_{p(t)}^{(\ul e_{l(t)})}$ commute
each other.  In particular, one can take $p$ as the identity permutation (by reordering the
commuting factors for $f_{p(s)}^{\ul e_l(s)}$), thus  
there exists $\ul\Be = (\ul e_1, \dots, \ul e_{\ul N})$ such that
\begin{equation*}
\ul f_1^{(\ul d_1)}\cdots \ul f_m^{(\ul d_m)}
= \ul F_{\ul\b_{l(1)}}^{(\ul e_{l(1)})}\cdots \ul F_{\ul\b_{l(m)}}^{(\ul e_{l(m)})}
      \in \SX_{\ul\Bh},
\end{equation*}
where $\ul e_{l(k)} = \ul d_{k}$, and $\ul e_l = 0$ if $l \ne l(k)$. 
The lemma is proved. 
\end{proof}

As a corollary, we have

\begin{cor}  %%%%  Cor. 4.11
Let $\ul\Bd = (\ul d_1, \dots, \ul d_m) \in \NN^m$.
There exists $\ul\Bc \in \NN^{\ul N}$ such that 
\begin{equation*}  
\ul f_m^{(\ul d_m)}\cdots \ul f_1^{(\ul d_1)}
    = L(\ul\Bc, \ul\Bh) + \sum_{\ul \Bc'> \ul\Bc}a_{\ul\Bc', \ul\Bc}L(\ul\Bc', \ul\Bh),
\end{equation*}
where $a_{\ul\Bc', \ul\Bc} \in q\ZZ[q]$.  
Thus $\ul f_m^{(\ul d_m)}\cdots \ul f_1^{(\ul d_1)}$ coincides with
the canonical basis $b(\ul\Bc, \ul\Bh)$. 
\end{cor}  

\begin{proof}
Let ${}^* : \ul\BU_q^- \to \ul\BU_q^-$ be the anti-involution.
By Lemma 4.10, we have  
\begin{equation*}
\ul f_m^{(\ul d_m)}\cdots \ul f_1^{(\ul d_1)}
   = (\ul f_1^{(\ul d_1)}\cdots \ul f_m^{(\ul d_m)})^* = L(\ul\Be, \ul\Bh)^*
\end{equation*}
for some $\ul\Be \in \NN^{\ul N}$.
It is known that $L(\ul\Be, \ul\Bh)^* = L({\ul\Be}^*, \ul\Bh^*)$, where
for $\ul\Bh = (j_1, \dots, j_{\ul N})$, $\ul\Be = (\ul e_1, \dots, \ul e_{\ul N})$,
we set
\begin{equation*}
\ul\Bh^* = (\io(j_{\ul N}), \io(j_{\ul N-1}), \dots, \io(j_1)),
   \quad \ul{\Be}^* = (\ul e_{\ul N}, \ul e_{\ul N-1}, \dots, \ul e_1).  
\end{equation*}
Here $\io$ is an involution on $J$ defined by $\ul w_0(\ul \a_j) = -\ul \a_{\io(j)}$. 
\par
Now $\ul f_m^{(\ul d_m)}\cdots \ul f_1^{(\ul d_1)}$ is a PBW-basis, which is
bar-invariant, hence it coincides with the canonical basis $b(\ul{\Be}^*, \ul\Bh^*)$.
Thus its expansion by the PBW basis $\SX_{\ul\Bh}$ is given, for a certain
$\ul\Bc \in \NN^{\ul N}$  as 
\begin{equation*}
  b(\ul{\Be}^*, \ul\Bh^*)
      = L(\ul\Bc, \ul\Bh) + \sum_{\ul\Bc' > \ul\Bc}a_{\ul\Bc', \ul\Bc}L(\ul\Bc', \ul\Bh)
\end{equation*}
with $a_{\ul\Bc', \ul\Bc} \in q\ZZ[q]$.  We have $b(\ul\Be^*, \ul\Bh^*) = b(\ul\Bc, \ul\Bh)$.
The corollary is proved. 
\end{proof}  

\para{4.12.}
Let $\Bc = (c_1, \dots, c_N) \in \NN^N$ be such that $\sum_k c_k\b_k = \sum_id_i \a_i$
with $\Bd = (d_1, \dots, d_n)$.  Then $\SO_{\Bc} \subset \BE_{\BV}$.  We define
$\Bd^k = (d_1^k, \dots, d_n^k) \in \NN^n$ for $k \in [1, N]$ associated to $\Bc$ as in 4.6.
The following formula was proved in \cite[Cor. 6.9]{L-can}.
\begin{equation*}
\tag{4.12.1}  
  \d(\Bc) = \dim \BE_{\BV} - \dim \SO_{\Bc} = -\sum_{h < k; i}d_i^hd_i^k
                  + \sum_{h < k ; i \to j}d_j^hd_i^k.
\end{equation*}

By (4.12.1), it is easy to see that $\d(\Bc_k) = 0$ for any $k$.  Thus by Lemma 4,5,
we have
\begin{equation*}
\tag{4.12.2}  
F(\Bd^k) = b(\Bc_k, \Bh).
\end{equation*}  

\par
We consider a generalization of (4.12.2) to the case where $\Bc$ has more than
one non-zero coordinates, in connection with the folding theory.
Take $\Bc \in \NN^N$. 
For each $k \in [1, \ul N]$, define $\wt\Bc_k \in \NN^N$ by 
\begin{equation*}
\tag{4.12.3}  
  \wt\Bc_k  = (0, \dots, 0,
     \underbrace{c_s, c_{s+1}, \dots, c_t}_{\ul\b_k\text{-part}}, 0, \dots, 0). 
\end{equation*}  
We define $\wt\Bd^k = (\wt d_1^k, \dots, \wt d_n^k) \in \NN^n$, as $\Bd^k$ is
defined from $\Bc$, by the condition that
$\sum_{k' = s}^t c_{k'}\b_{k'} = \sum_{i = 1}^n\wt d_i^k\a_i$, where
$j_k = \{ s, s+1, \dots, t\}$. 
Then we have a lemma.

\begin{lem}  %%%%   Lemma 4.13
Under the notation above, we have
$F(\wt\Bd^k) = b(\wt\Bc_k, \Bh)$. 
\end{lem}  

\para{4.14.}
The lemma will be proved in 4.18 after some preliminaries.
\par
By Lemma 4.5, $F(\wt\Bd^k)$ is written as a linear combination of
$L(\Bc, \Bh)$ with coefficient $q^{\d(\Bc)}$, where $\SO_{\Bc} \subset \BE_{\BV}$.
The latter condition is equivalent to the condition that $\weit(\Bc) = \dim \BV$,
with $\dim \BV = \sum_{i=1}^n\wt d_i^k\a_i$. 
Since $\weit(\wt \Bc_k) = \dim \BV$, it is enough to show that $\d(\wt\Bc_k) = 0$. 
In the original proof, we have verified $\d(\wt\Bc_k) = 0$ by case by case
computation using the formula (4.12.1). The following discussion is based on the
referee's suggestion, which gives a simple and unified proof.    
\par
Since we already know  that $F(\wt \Bd^k)$ is a canonical basis, we have only to show
that
\par\medskip\noindent
(4.14.1)  \ $\wt\Bc_k$ is the smallest element 
among the elements in $\{ \Bc' \in \NN^N \mid \weit(\Bc') = \dim \BV \}$ with respect to
the total order on $\NN^N$.
\par\medskip
In the discussion below, we verify (4.14.1) by a general argument, avoiding
the use of the formula (4.12.1). The argument works also in the original case, and
we can show (4.12.2) without appealing (4.12.1). 

\para{4.15.}
We prepare some general facts. 
Assume that $w = s_{i_1}\cdots s_{i_m}$ is a reduced expression of $w \in W$.
We define $\b_k = s_{i_1}\cdots s_{i_{k-1}}(\a_{i_k})$ for $k = 1, \dots, m$.
Set $\vD^+(w) = \vD^+ \cap w(\vD^-)$.  
Then it is known that $\b_1, \dots, \b_m$ are all distinct, and 

\begin{equation*}
\vD^+(w) = \{ \b_1, \dots, \b_m\}.
\end{equation*}

We have 
\begin{equation*}
\tag{4.15.1}  
\b_{p} = w'(\a_{i_p}), \ \b_{p+1} = w'(s_{i_p}(\a_{i_{p+1}})), \ \dots, \
   \b_m = w'(s_{i_p}\cdots s_{i_{m-1}}(\a_{i_m})),
\end{equation*}
where $w' = s_{i_1}\cdots s_{i_{p-1}}$.
Then for $w'' = s_{i_p}\cdots s_{i_m}$, we have

\begin{equation*}
\tag{4.15.2}  
\{ \b_p, \dots, \b_m \} = w' \Bigl(\vD^+(w'')\Bigr) \subset \vD^+(w). 
\end{equation*}

Returning to the setting in Lemma 4.13, we show the following.

\begin{lem}  %%%%  Lemma 4.16
Assume that $c_k\b_k = \sum_{j=a}^bc_j'\b_j$ with $a < b$, where $c_k, c_j' \in \NN$.
Then $\b_k$ appears in the sequence $\b_a, \dots, \b_b$.    
\end{lem}  
\begin{proof}
We can write as $\{\b_a, \dots, \b_b\} = w'\Bigl(\vD^+(w'')\Bigr)$ for some
$w', w'' \in W$ as in (4.15.2). Put $\b_j' = {w'}\iv(\b_j)$.  Then $\b_j' \in \vD^+$
by (4.15.1).  
Thus ${w'}\iv(c_k\b_k) = c_k\b_k'$ is an $\NN$-linear combination of $\b_j'$
for $j = a, \dots, b$. In particular, $\b_k' \in \vD^+$.
Since $\b_j' \in \vD^+(w'')$, ${w''}\iv(\b_j') < 0$.  Hence ${w''}\iv (\b_k')< 0$,
and so $\b_k'$ appears in the sequence $\b_a', \dots, \b_b'$.
It follows that $\b_k$ appears in the sequence $\b_a, \dots, \b_b$.
\end{proof}

\para{4.17.}
We now prove (4.14.1) in the case where $c_k\b_k = \sum_{i=1}^nd_i^k\a_i$
(the original case in (4.12.1)).
Let $\Bc' = (c'_j) \in \NN^N$ such that $\sum_{j=a}^b c_j'\b_j = c_k\b_k$,
and $\Bc_k = (0, \dots, 0, c_k, 0, \dots, 0)$.
By Lemma 4.16, $\b_k$ appears in the sequence $\b_a, \dots, \b_b$. 
If $a < k$ with $c'_a \ne 0$, then $\Bc' > \Bc_k$. 
Thus we may assume that $c'_i = 0$ for $i < k$. If $c'_k > c_k$, then $\Bc' > \Bc_k$.
Hence we assume that $c'_k \le c_k$.
Since $\sum_{j = k}^bc_j'\b_j = c_k\b_k$, if $c'_k = c_k$, then $\Bc' = \Bc_k$.
So we may assume that $c'_k < c_k$.
In this case,
\begin{equation*}
\tag{4.17.1}  
(c_k - c'_k)\b_k = \sum_{j = k+1}^bc_j'\b_j.
\end{equation*}  
By applying the previous argument to this equation, we see that
$\b_k$ appears in the sequence $\b_{k+1}, \dots, \b_b$.  This is absurd, and
the case $c_k' < c_k$ does not occur.
Hence (4.14.1) holds in this case.

\para{4.18.}
We consider the general case, namely, $\wt \Bc_k$ is as in (4.12.3), satisfying 
the relation $\sum_{u = s}^tc_{u}\b_{u} = \sum_{j = a}^bc'_j\b_j$ for $\Bc' = (c'_i) \in \NN^N$. 
Set $\nu = \sum_{u = s}^tc_{u}\b_{u}$.  We prove (4.14.1) by induction on $|\nu|$,
where $|\nu| = \sum_{j=1}^nn_j$ for $\nu = \sum_{j=1}^nn_j\a_j \in Q_+$. 
We may assume that the interval $[a, b]$ is $\s$-stable with respect to the action of
$\s$ on $\NN^N$, by extending the original $[a,b]$.
Let $w', w'' \in W$ such that $\{ \b_a, \dots, \b_b\} = w'\Bigl(\vD^+(w'')\Bigr)$
as in (4.15.2). 
Then $w'$ and $w''$ are $\s$-stable.
Let $\b_j' = {w'}\iv (\b_j)$. As in the proof of Lemma 4.16,    
we have ${w''}\iv(\sum_{u = s}^tc_u\b_u') \in Q_-$. 
Then there exists $\b'_{u_1}$ such that ${w''}\iv(\b'_{u_1}) \in \vD^-$ for $s \le u_1 \le t$.
But since $w', w''$ are $\s$-stable, if ${w''}\iv(\b'_{u_1}) < 0$, then all the $\s$-conjugate
of $\b'_{u_1}$ have a similar property.
This implies that all the $\b_s, \dots, \b_t$ appear, as an interval,  in the sequence
$\b_a, \dots, \b_b$. 
\par
If $a < s$, and $c'_a \ne 0$, then we have $\Bc' > \wt \Bc_k$.
Hence we may assume that $c'_u = 0$ for any $u < s$.  
Now suppose that $\wt\Bc_k > \Bc'$. Then there exists $u_1 \in [s,t]$ such that
$c_u = c'_u$ for any $u < u_1$ and that $c_{u_1} > c'_{u_1}$.
In this case, we have
\begin{equation*}
\tag{4.18.1}  
(c_{u_1}- c'_{u_1})\b_{u_1} + \sum_{u = u_1+1}^tc_u\b_u = \sum_{j > u_1}c'_j\b_j. 
\end{equation*}  

For $ u_1 + 1 \le u \le t$, if $c_u \ne 0$ and $c'_u \ne 0$, by subtracting $\b_u$
from both side of (4.18.1), we obtain a formula, where $c_u, c'_u$ are replaced by
$c_u-1, c'_u-1$. 
Let $\nu$ be the left hand side of (4.18.1), and $\nu'$ the element in $Q_+$
obtained from $\nu$ by replacing $c_u$ by $c_u-1$. Since
$|\nu'| < |\nu|$, by induction on the weights, we may assume that (4.14.1) holds for $\nu'$.
But this contradicts our assumption that $\wt\Bc_k > \Bc'$.
Hence if $c_u \ne 0$, then we must have $c'_u = 0$.
\par
Since $F_{\b_s}, \dots, F_{\b_t}$ are mutually commuting, we may assume that
there exists $u_2 > u_1$ such that $c_{u_1+1} \ne 0, \dots, c_{u_2} \ne 0$ and that
$c_u = 0$ for $u > u_2$. It follows that $c'_u = 0$ for $u = u_1 +1, \dots, u_2$. 
In this case, (4.18.1) can be rewritten as

\begin{equation*}
\tag{4.18.2}  
  (c_{u_1}- c'_{u_1})\b_{u_1} + \sum_{u = u_1+1}^{u_2}c_u\b_u
                = \sum_{j > u_2}c'_j\b_j. 
\end{equation*}  

By applying the discussion of the proof of Lemma 4.16, together with the discussion
in the first paragraph of this subsection (here $w', w''$ are not necessarily $\s$-stable)
for the equation (4.18.2), one can find 
$u \le u_2$ such that $\b_u$ appears in the sequence
$\b_{u_2+1}, \b_{u_2+2}, \dots$. 
This is absurd, and we obtain $\wt\Bc_k \le \Bc'$.
Hence if $\wt\Bc_k \ne \Bc'$, then we have $\wt\Bc_k < \Bc'$.
Thus (4.14.1) is proved, and Lemma 4.13 follows.

\para{4.19.}
Let $\ul\Bc = (\ul c_1, \dots, \ul c_{\ul N}) \in \NN^{\ul N}$.
For each $k \in [1, \ul N]$, we define $\ul\Bc_k \in \NN^{\ul N}$, similarly to
$\Bc_k \in \NN^N$ in 3.6,
by the condition
that the $k$-th coordinate of $\ul \Bc$ coincides with $\ul c_k$, and all other coordinates are 0.
We define $\ul\Bd^k = (\ul d^k_1, \dots, \ul d^k_m) \in \NN^m$ in a similar way
as $\Bd^k$ in 4.6. 
Hence we have $\ul c_k\ul\b_k = \sum_{j=1}^m\ul d^k_j\a_j$ for $k \in [1, \ul N]$. 
We define
\begin{equation*}
\tag{4.19.1}  
F((\ul\Bc)) = F(\ul \Bd^1)F(\ul \Bd^2) \cdots F(\ul \Bd^{\ul N}),
\end{equation*}
where
\begin{equation*}
\tag{4.19.2}  
F(\ul\Bd^k) = \ul f_{m}^{(\ul d_m^k)}\ul f_{m-1}^{(\ul d_{m-1}^k)}
                      \cdots \ul f_1^{(\ul d_1^k)}. 
\end{equation*}

We apply Corollary 4.11 for the case where $\ul\Bd = \ul\Bd^k$.
We show

\begin{lem}  %%%%  Lemma 4.20
The expansion of $F(\ul \Bd^k)$ in terms of $\SX_{\ul\Bh}$ is given as 
\begin{equation*}  
  F(\ul \Bd^k) = L(\ul\Bc_k, \ul\Bh) + \sum_{\ul\Bc' >  \ul\Bc_k}
                   a_{\ul\Bc',\ul\Bc_k}L(\ul\Bc',\ul\Bh)
\end{equation*}
with $a_{\ul\Bc',\ul\Bc_k} \in q\ZZ[q]$.  Hence
$F(\ul \Bd^k)$ coincides with $b(\ul\Bc_k, \ul\Bh)$.
\end{lem}

\begin{proof}
By Corollary 4.11, $F(\ul\Bd^k)$ is written as a linear combination of $L(\ul\Bc',\ul\Bh)$
with coefficient $a_{\ul\Bc',\ul\Bc_k} \in \ZZ[q]$. Again by Corollary 4.11,
it is enough to show that $a_{\ul\Bc_k, \ul\Bc_k}$ is not contained in $q\ZZ[q]$.
\par
Passing to ${}_{\BA'}\ul\BU_q^-$, we consider the isomorphism
$\Phi : {}_{\BA'}\ul\BU_q^- \isom \BV_q$ as in Theorem 3.3.
Let $\Bc = (c_1, \dots, c_N) \in \NN^{N,\s}$ be the sequence corresponding to
$\ul\Bc \in \NN^{\ul N}$
such that $\Phi(L(\ul\Bc, \ul\Bh)) = \pi(L(\Bc, \Bh))$.
For each $k \in [1,\ul N]$,
define $\wt\Bc_k$ and $\wt \Bd^k$ as in 4.12.
In particular,
\begin{equation*}
\wt\Bc_k = (0, \dots, 0, \underbrace{\ul c_k, \dots, \ul c_k}_{\ul\b_k\text{-part}}, 0, \dots, 0).
\end{equation*}  
Then we have a relation in $\BV_q$, 
\begin{equation*}
\tag{4.20.1}  
\Phi(\ul f_m^{(\ul d_m^{k})}\cdots \ul f_1^{(\ul d_1^{k})})
    = \pi(f_n^{(\wt d_n^k)}f_{n-1}^{(\wt d_{n-1}^k)}\cdots f_1^{(\wt d_1^k)}).
\end{equation*}  

\par
On the other hand, by Lemma 4.13, we have
\begin{equation*}
\tag{4.20.2}  
f_n^{(\wt d_n^k)}f_{n-1}^{(\wt d_{n-1}^k)}\cdots f_1^{(\wt d_1^k)} = 
     L(\wt\Bc_k, \Bh) + \sum_{\Bc' > \wt\Bc_k}a_{\Bc',\wt\Bc_k}L(\Bc', \Bh) 
\end{equation*}
with $a_{\Bc',\wt\Bc_k} \in q\ZZ[q]$. 
Note that we have $\pi(L(\wt\Bc_k, \Bh)) = \Phi(L(\ul\Bc_{k}, \ul\Bh))$.
\par

It follows from (4.20.1) and (4.20.2)
that the following formula holds in ${}_{\BA'}\ul\BU_q^-$.

\begin{equation*}
\tag{4.20.3}  
  \ul f_m^{(\ul d_m^{k})}\ul f_{m-1}^{(\ul d_{m-1}^{k})}\cdots \ul f_1^{(\ul d_1^{k})}
    = L(\ul\Bc_{k}, \ul\Bh) + \sum_{\ul\Bc' > \ul\Bc_{k}}
            a'_{\ul\Bc', \ul\Bc_{k}}L(\ul\Bc', \ul\Bh)
\end{equation*}
with $a'_{\ul\Bc', \ul\Bc_{k}} \in q\FF[q]$, where $\FF$ is a finite field $\ZZ/p\ZZ$.  
This means that under the expansion of $\ul f_m^{(\ul d_m^k)}\cdots \ul f_1^{(\ul d_1^k)}$
in terms of $\SX_{\ul\Bh}$, the coefficient of $L(\ul\Bc_k, \ul\Bh)$ is
contained in $a + p\ZZ[q]$, where $a \in \ZZ$ is equal to 1 modulo $p$.
In particular, this coefficient is not contained in $q\ZZ[q]$.
The lemma is proved.
\end{proof}

The following result is a generalization of Theorem 4.7
to the non-symmetric case.

\begin{thm}  %%%%   Theorem 4.21
For each $\ul\Bc \in \NN^{\ul N}$, $F((\ul\Bc))$ can be written as
\begin{equation*}  
F((\ul\Bc)) = L(\ul\Bc, \ul\Bh)
     + \sum_{\ul\Bc' > \ul\Bc}h_{\ul\Bc',\ul\Bc}L(\ul\Bc', \ul\Bh)
  \quad\text{ with }\quad h_{\ul\Bc',\ul\Bc} \in \BA,  
\end{equation*}
where $\ul\Bc'$ satisfies the relation $\weit(L(\ul\Bc',\ul\Bh)) = \weit(L(\ul\Bc, \ul\Bh))$.
\end{thm}  

\begin{proof}
By induction on $k$, we shall prove
\begin{equation*}
\tag{4.21.1}  
  \ul F(\ul \Bd^1)\ul F(\ul \Bd^2)\cdots \ul F(\ul \Bd^k)
  = L(\ul \Bc_{\le k}, \ul\Bh) + \sum_{\ul\Bc' > \ul\Bc_{\le k}}
          h_{\ul\Bc', \ul\Bc_{\le k}}L(\ul\Bc',\ul\Bh),
\end{equation*}
where $\ul\Bc_{\le k} = (\ul c_1, \dots, \ul c_k, 0, \dots, 0)$, and
$h_{\ul\Bc', \ul\Bc_{\le k}} \in \BA$.  
The formula (4.21.1) certainly holds for $k = 1$ by Lemma 4.20. If it holds for $k = \ul N$,
it gives the assertion of the theorem. 
We assume that (4.21.1) hods for $k$.  Then by Lemma 4.20, we have

\begin{align*}
\tag{4.21.2}
\ul F&(\ul \Bd^1)\ul F(\ul \Bd^2)\cdots \ul F(\ul \Bd^k)\ul F(\ul \Bd^{k+1})  \\
&= \Bigl(L(\ul\Bc_{\le k}, \ul\Bh) + \sum_{\ul\Bc' > \ul\Bc_{\le k}}
        h_{\ul\Bc',\ul\Bc_{\le k}}L(\ul\Bc', \ul\Bh)\Bigr)
        \Bigl(L(\ul\Bc_{k+1}, \ul\Bh) + \sum_{\ul\Bc'' > \ul\Bc_{k+1}}
             a_{\ul\Bc'',\ul\Bc_{k+1}}L(\ul\Bc'',\ul\Bh)\Bigr).
\end{align*}

\par
We consider the relation $\ul\Bc' > \ul\Bc_{\le k}$ for a fixed $k$.
This condition is given by either $(\ul c'_1, \dots, \ul c'_k) > (\ul c_1, \dots, \ul c_k)$ 
or $(\ul c'_1, \dots, \ul c'_k) = (\ul c_1, \dots, \ul c_k)$ with $\ul c'_{\ell} > 0$
for some $\ell > k$.
But since $L(\ul\Bc',\ul\Bh)$ has the same weight as $L(\ul\Bc_{\le k}, \ul\Bh)$,
which is equal to $\ul c_1\ul\b_1 + \cdots + \ul c_k\ul\b_k$, the latter case does not occur,
and we must have $(\ul c_1', \dots, \ul c'_k) > (\ul c_1, \dots, \ul c_k)$. 
\par
We also consider the relation $\ul\Bc'' > \ul\Bc_{k+1}$ for a fixed $k$.
This condition is given by
either $(\ul c''_1, \dots, \ul c''_{k+1}) > (0, \dots, 0, \ul c_{k+1})$ or
$(\ul c''_1, \dots, \ul c''_{k+1}) =  (0, \dots, 0, \ul c_{k+1})$ with $\ul c''_{\ell} > 0$
for some $\ell > {k+1}$.
But since $L(\ul\Bc'', \ul\Bh)$ has the same weight as
$\ul F_{\ul\b_{k+1}}^{(\ul c_{k+1})}$,
which is equal to $\ul c_{k+1}\ul\b_{k+1}$,
the latter case does not occur.  In the former
case, either $\ul c''_{\ell} > 0$ for some $\ell < k+1$, or
$\ul c''_1 = \cdots = \ul c''_k = 0$ and $\ul c''_{k+1} > \ul c_{k+1}$.  
Again the latter case does not occur. 
\par
In the following, we show each factor $L(\ul\Bd, \ul\Bh)L(\ul\Bd',\ul\Bh)$ in (4.21.2)
is a linear combination of $L(\ul\Be, \ul\Bh)$ such that $\ul\Be > \ul \Bc_{\le k+1}$.
\par\medskip
(a) First note that
$L(\ul\Bc_{\le k}, \ul\Bh) L(\ul\Bc_{k+1}, \ul\Bh) = L(\ul\Bc_{\le k+1}, \ul\Bh)$.
This is clear from the definition. 
\par\medskip
(b) Next consider $L(\ul\Bc',\ul\Bh)L(\ul c_{k+1}, \ul\Bh)$ such that
$\ul\Bc' > \ul\Bc_{\le k}$. 
By the above remark, we have $(\ul c'_1, \dots, \ul c'_k) > (\ul c_1, \dots, \ul c_k)$.
Moreover, 
$(\ul F_{\ul\b_{k+1}}^{(\ul c'_{k+1})}\cdots \ul F_{\ul\b_{\ul N}}^{(\ul c'_{\ul N})})
\ul F_{\ul\b_{k+1}}^{(\ul c_{k+1})}$
is a linear combination of $L(\ul\Bd, \ul\Bh)$ such that $s(\ul\Bd) \subset [k+1, \ul N]$
(see 2.4 for the definition $s(\ul\Bd)$).     
It follows that $L(\ul\Bc',\ul\Bh)\ul F_{\ul\b_{k+1}}^{(\ul c_{k+1})}$
is a linear combination of $L(\ul\Be, \ul\Bh)$, where $\ul\Be = (\ul e_1, \dots, \ul e_{\ul N})$
is such that $\ul e_{\ell} = \ul c'_{\ell}$ for $\ell = 1, \dots, k$.
In particular, $\ul\Be > \ul\Bc_{\le k+1}$.  Thus $L(\ul\Bc',\ul\Bh)L(\ul\Bc_{k+1}, \ul\Bh)$
is a linear combination of $L(\ul\Be, \ul\Bh)$ such that $\ul\Be > \ul\Bc_{\le k+1}$.

\par\medskip
(c) \ Next consider $L(\ul\Bc', \ul\Bh)L(\ul \Bc'',\ul\Bh)$, where
$\ul\Bc' > \ul\Bc_{\le k}$ and $\ul\Bc'' > \ul\Bc_{k+1}$.
By the above discussion, $(\ul c_1', \dots, \ul c_k') > (c_1, \dots, c_k)$.
In particular, we have $\ul\Bc' > \ul\Bc_{\le k+1}$. 
By Proposition 2.5, $L(\ul\Bc', \ul\Bh)L(\ul \Bc'',\ul\Bh)$ is a linear combination of
$L(\ul\Be, \ul\Bh)$ such that $\ul\Be \ge \ul\Bc'$.  Hence $\ul\Be > \ul\Bc_{\le k+1}$. 
Thus $L(\ul\Bc', \ul\Bh)L(\ul \Bc'',\ul\Bh)$ is a linear combination of $L(\ul\Be,\ul\Bh)$
such that $\ul\Be > \ul\Bc_{\le k+1}$. 

\par\medskip
(d) \ Finally consider $L(\ul\Bc_{\le k}, \ul\Bh)L(\ul\Bc'',\ul\Bh)$, where
$\ul\Bc''> \ul\Bc_{k+1}$. Then by the above remark,
$\ul c''_{\ell} > 0$ for some $\ell\le k$. 
Set $\ul\Bc''_{\le k} = (\ul c''_1, \dots, \ul c''_k, 0, \dots, 0)$.  
By Proposition 2.5, $L(\ul\Bc_{\le k}, \ul\Bh)L(\ul \Bc''_{\le k}, \ul\Bh)$ is
a linear combination of $L(\ul\Be, \ul\Bh)$ such that $\ul\Be \ge \ul\Bc_{\le k}$.
Moreover, $S(\ul\Be)$ is contained in $[1,k]$.
If there exists $1 \le \ell \le k$ such that $\ul e_{\ell} > \ul c_{\ell}$ and that
$\ul e_i = \ul c_i$ for $1 \le i < \ell$, then
we have $\ul\Be > \ul\Bc_{\le k+1}$.  Hence we may assume that
$(\ul e_1, \dots, \ul e_k) = (\ul c_1, \dots, \ul c_k)$. 
But since the weight of $L(\ul\Be, \ul\Bh)$ coincides with
$\weit(L(\ul\Bc_{\le k}, \ul\Bh)) + \weit(L(\ul\Bc''_{\le k}, \ul\Bh))$, 
we must have $\weit(L(\ul\Bc''_{\le k}, \ul\Bh)) = 0$, and $\ul c_1'' = \cdots = \ul c_k'' = 0$.
This contradicts the condition that $c''_{\ell} > 0$ for some  $\ell \le k$. 
Thus we have $\ul\Be > \ul\Bc_{\le k+1}$.
It follows that $L(\ul\Bc_{\le k}, \ul\Bh)L(\ul\Bc'',\ul\Bh)$ is a linear combination of
$L(\ul\Be, \ul\Bh)$ such that $\ul\Be > \ul\Bc_{\le k+1}$.
\par\medskip
By (a) $\sim$ (d), (4.21.1) holds for $k +1$.  The theorem is proved. 
\end{proof}  

\para{4.22.}
By Theorem 4.21, the set $\{ F((\ul\Bc)) \mid \ul\Bc \in \NN^{\ul N} \}$
gives an $\BA$-basis of ${}_{\BA}\ul\BU_q^-$, and a basis of $\ul\BU_q^-$.
We write $F((\ul \Bc)) = m(\ul\Bc, \ul\Bh)$, and set
$\SM_{\ul\Bh} = \{ m(\ul\Bc, \ul\Bh) \mid \ul \Bc \in \NN^{\ul N} \}$. 
$\SM_{\ul\Bh}$ gives a monomial basis of $\ul\BU_q^-$. 
\par
Let $\Bh$ be the reduced sequence of $w_0$ obtained from $\ul\Bh$.
We have the monomial basis $\SM_{\Bh} = \{ m(\Bc, \Bh) \mid \Bc \in \NN^N\}$
defined by Theorem 4.7.  But since $\s$ does not preserve the set $\SM_{\Bh}$,
this basis is not appropriate for 
the folding theory. 
In the following, we define a new basis $\wt\SM_{\Bh}$ of $\BU_q^-$
by modifying $\SM_{\Bh}$.
\par
Take $\Bc \in \NN^N$.
For each $k \in [1, \ul N]$, define $\wt\Bc_k$ and $\wt\Bd^k$ as in 4.12.
As before, we set
\begin{equation*}
\tag{4.22.1}  
F(\wt\Bd^k) = f_n^{(\wt d_n^k)}f_{n-1}^{(\wt d_{n-1}^k)}\cdots f_1^{(\wt d_1^k)},
\end{equation*}  
and define
\begin{equation*}
\tag{4.22.2}  
\wt m(\Bc, \Bh) = F(\wt\Bd^1)F(\wt\Bd^2) \cdots F(\wt\Bd^{\ul N}). 
\end{equation*}  

Note that
\begin{equation*}
L(\wt\Bc_1,\Bh)L(\wt\Bc_2, \Bh) \cdots L(\wt\Bc_{\ul N}, \Bh) = L(\Bc, \Bh).
\end{equation*}
Then by Lemma 4.13, together with a similar discussion as in the proof of
Theorem 4.21, we have

\begin{equation*}
\tag{4.22.3}  
\wt m(\Bc, \Bh) = L(\Bc, \Bh) + \sum_{\Bc' > \Bc}h_{\Bc',\Bc}L(\Bc', \Bh)
\end{equation*}
with $h_{\Bc',\Bc} \in \BA$.
\par
Set $\wt\SM_{\Bh} = \{ \wt m(\Bc,\Bh) \mid \Bc \in \NN^N\}$. 
Then $\wt\SM_{\Bh}$ gives a monomial basis of $\BU_q^-$, which is different from
$\SM_{\Bh}$.  The following result is immediate from the previous discussion,
which gives an analogue of Proposition 3.7 for the case of monomial bases.

\begin{prop}  %%%%   Prop. 4.23
Let $\SM_{\ul\Bh}$ be the monomial basis of $\ul\BU_q^-$, and 
$\wt\SM_{\Bh}$ the monomial basis of $\BU_q^-$.
\begin{enumerate}
\item \ $\s$ acts on $\wt\SM_{\Bh}$ as a permutation
$\s : \wt m(\Bc,\Bh) \mapsto \wt m(\s(\Bc), \Bh)$. $\wt m(\Bc,\Bh)$ is $\s$-invariant
if and only if $\Bc \in \NN^{N,\s}$. 
\item
Under the bijection $\NN^{\ul N} \simeq \NN^{N,\s}$, $\ul\Bc \mapsto \Bc$,
the assignment $m(\ul\Bc,\ul\Bh) \mapsto \wt m(\Bc, \Bh)$ gives a bijection
$\SM_{\ul\Bh} \isom \wt\SM_{\Bh}^{\s}$, where $\wt\SM_{\Bh}^{\s}$ is the set of
$\s$-fixed elements in $\wt\SM_{\Bh}$.
\item
The bijection in (ii) satisfies the property
\begin{equation*}
\Phi(m(\ul\Bc, \ul\Bh)) = \pi(\wt m(\Bc, \Bh)). 
\end{equation*}  
\end{enumerate}  
\end{prop}

\par\bigskip
\section{ Mackey filtration for KLR algebras }

\para{5.1.}
Let $X = (I, (\,\ ))$ be a symmetric Cartan datum, and $\s$ an admissible automorphism on $X$. 
(Here we consider $X$ of finite type, following the setup in 1.1.  But the discussion
below works for $X$ of general type.)
Assume that $\Bk$ is an algebraically closed field.  We define 
a KLR algebra associated to $(X, \s)$ as follows. For each $i, i' \in I$, 
choose  a polynomial $Q_{i,i'}(u,v) \in \Bk[u,v]$ 
satisfying the following properties;
\begin{enumerate}
\item \ $Q_{i,i}(u,v) = 0$,  
\item  \ $Q_{i,i'}(u,v)$ is written as 
$Q_{i,i'}(u,v) = \sum_{p,q}t_{ii';pq}u^pv^q$, where the 
coefficients $t_{ii';pq} \in \Bk$ satisfy the condition 
\begin{equation*}
(\a_i,\a_i)p + (\a_{i'},\a_{i'})q = -2(\a_i,\a_{i'}) 
        \quad\text{ if $t_{ii';pq} \ne 0$. }
\end{equation*}
Moreover, $t_{ii';-a_{ii'},0} \ne 0, t_{ii';0,-a_{ii'}} \ne 0$.  
\item  \ $Q_{i,i'}(u,v) = Q_{i',i}(v,u)$, 
\item \ $Q_{\s(i),\s(i')}(u,v) = Q_{i,i'}(u,v)$.  
\end{enumerate}
\par
For $\b \in Q_+$  such that $|\b| = n$, define $I^{\b}$ by 
\begin{equation*}
I^{\b} = \{ (i_1, \dots, i_n) \in I^n \mid \sum_{1 \le k \le n}\a_{i_k} = \b\}.
\end{equation*}
The KLR algebra $R(\b)$ is an associative $\Bk$-algebra 
defined by the generators, $e(\nu)$ ($\nu = (\nu_1, \dots, \nu_n) \in I^{\b}$), 
$x_k$ $(1 \le k \le n)$, $\tau_k \ (1 \le k < n)$,  with relations
\begin{align*}
\tag{1}
e(\nu)e(\nu') &= \d_{\nu,\nu'}e(\nu), \quad \sum_{\nu \in I^{\b}}e(\nu) = 1, \\
\tag{2}
x_kx_l &= x_lx_k, \quad  x_ke(\nu) = e(\nu)x_k, \\ 
\tag{3}
\tau_ke(\nu) &= e(s_k\nu)\tau_k, \quad 
           \tau_k\tau_l = \tau_l\tau_k, \quad\text{ if } |k - l| > 1, \\
\tag{4}
\tau_k^2e(\nu) &= Q_{\nu_k, \nu_{k+1}}(x_k, x_{k+1})e(\nu), \\
\tag{5}
(\tau_k x_l &- x_{s_k l}\tau_k)e(\nu) = 
                    \begin{cases}
                      -e(\nu)  &\quad\text{ if } l = k, \nu_k = \nu_{k+1}, \\
                       e(\nu)  &\quad\text{ if } l = k+1, \nu_k = \nu_{k+1}, \\
                       0       &\quad\text{ otherwise, }  
                    \end{cases}  \\
\tag{6}
(\tau_{k+1}&\tau_{k}\tau_{k+1} - \tau_{k}\tau_{k+1}\tau_{k})e(\nu)  \\
        &= \begin{cases}
     \displaystyle
     \frac{Q_{\nu_k,\nu_{k+1}}(x_k, x_{k+1}) - Q_{\nu_k,\nu_{k+1}}(x_{k+2}, x_{k+1})}
          {x_k - x_{k+2}}e(\nu)  &\quad\text{ if } \nu_k = \nu_{k+2}, \\
     0                           &\quad\text{ otherwise. }
          \end{cases}
\end{align*}  
(Here $s_k$ ($1 \le k < n$) is the transvection $(k, k+1)$ in the symmetric group
$S_n$. $S_n$ acts on $I^n$ by 
$w : (\nu_1, \dots, \nu_n) \mapsto (\nu_{w\iv(1)}, \dots, \nu_{w\iv (n)})$.)
\par
The algebra $R(\b)$ is a $\BZ$-graded algebra, where the degree is 
defined, for $\nu = (\nu_1, \dots, \nu_n) \in I^{\b}$, 
 by 
\begin{equation*}
\deg e(\nu) = 0, \quad \deg x_ke(\nu) = (\a_{\nu_k}, \a_{\nu_k}), \quad 
\deg\tau_ke(\nu) = -(\a_{\nu_k}, \a_{\nu_{k+1}}).
\end{equation*}  

For a graded $R(\b)$-module $M = \bigoplus_{i \in \BZ}M_i$, the grading shift
by $-1$ is denoted by $qM$, namely $(qM)_i = M_{i-1}$, where $q$ is an indeterminate.

There exists an anti-involution $\psi : R(\b) \to R(\b)$, which leaves all the
generators $e(\nu), x_k$ and $\t_k$ invariant. 

\para{5.2.}
We consider the KLR algebra $R(\b)$ associated to
$(X, \s)$. Let $\Bn$ be the order of $\s$, and assume that $\ch \Bk$ does not
divide $\Bn$.
By the condition $Q_{\s(i), \s(i')}(u,v) = Q_{ii'}(u,v)$ on the polynomials $Q_{ii'}$, 
$\s$ induces an isomorphism $R(\b) \isom R(\s(\b))$, by
$e(\nu) \mapsto e(\s(\nu)),x_k \mapsto x_k, \tau_k \mapsto \tau_k$,
where $\s(\nu_1, \dots, \nu_n) = (\s(\nu_1), \dots, \s(\nu_n))$.   
Assume that $\b$ is $\s$-stable. Then $\s$ induces an automorphism
on $R(\b)$, which we denote also by the same symbol $\s$.
Let $R(\b)\Mod$ be the abelian category of graded $R(\b)$-modules. 
We obtain a functor
$\s^*: R(\b)\Mod \to R(\b)\Mod$, where for an $R(\b)$-module $M$, $\s^*M$ has
the same underlying space $M$, but the action of $R(\b)$ is given by $\s(x)$ on $M$
for $x \in R(\b)$.
\par
Following McNamara \cite{M}, we consider an analogue of Lusztig's category
$\wt \ZC$ associated to the periodic functor $\s^*$ on $\ZC = R(\b)\Mod$ (see \cite{L-book}).
Assume that $\b \in Q_+^{\s}$, and let
$\SC_{\b}$ be the category whose objects are the pairs $(M, \f)$,
where $M \in R(\b)\Mod$,
and $\f : \s^*M  \isom M$ is an isomorphism  such that
the composition satisfies the relation
\begin{equation*}
\f\circ \s^*\f\circ \cdots \circ (\s^*)^{\Bn-1}\f = \id_M.
\end{equation*}
A morphism from $(M, \f) \to (M', \f')$ in $\SC_{\b}$ is a morphism $f : M \to M'$
satisfying the following commutative diagram
\begin{equation*}
\begin{CD}
  \s^*M @>\s^* f>> \s^* M '  \\
  @V\f VV        @VV\f'V    \\
  M   @>f>> M '
\end{CD}    
\end{equation*}  
Then the category $\SC_{\b}$ is equivalent to the category of
graded representations of the algebra
\begin{equation*}
R(\b)\sharp(\ZZ/\Bn\ZZ) = R(\b) \otimes_{\Bk}\Bk[\ZZ/\Bn\ZZ],
\end{equation*}
where $\Bk[\ZZ/\Bn\ZZ]$ is the group algebra of $\ZZ/\Bn\ZZ$, and
$\Bk[\ZZ/\Bn\ZZ]$ acts on $R(\b)$ through the action of $\s$.
Hence $\SC_{\b}$ is an abelian category.
We denote by $\SP_{\b}$ the full subcategory of $\SC_{\b}$ consisting of
finitely generated projective objects in $\SC_{\b}$.
Let $\SL_{\b}$ be the full subcategory of $\SC_{\b}$ consisting of $(M, \phi)$
such that $M$ is a finite dimensional $R(\b)$-module. 
Then $\SP_{\b}$ is an additive category and $\SL_{\b}$ is an abelian category. 

\para{5.3.}
Let $\z_{\Bn}$ be a primitive $\Bn$-th root of unity in $\CC$.
We fix, once and for all, a ring homomorphism $\ZZ[\z_{\Bn}] \to \Bk$,
which maps $\z_{\Bn}$ to a primitive $\Bn$-th root of unity in $\Bk$.
If $(M, \f) \in \SC_{\b}$, then $(M, \z_{\Bn}\f) \in \SC_{\b}$, which we denote by
$\z_{\Bn}(M, \f)$. 
\par
An object $(A, \f)$ of $\SC_{\b}$ is said to be traceless if there is
a representation $M$ of $R(\b)$, an integer $t \ge 2$ dividing $\Bn$ such that
$(\s^*)^tM  \simeq M$, and an isomorphism
\begin{equation*}
A \simeq M \oplus \s^*M \oplus \cdots \oplus (\s^*)^{t-1}M,
\end{equation*}
where $\f : \s^*A \isom A$ is given by the identity maps $(\s^*)^kM \isom (\s^*)^kM$
$(1 \le k \le t-1)$, and an isomorphism $(\s^*)^tM \isom M$ on each
direct summand, under the above isomorphism.
\par
For an additive category $\SP_{\b}$, we define $K(\SP_{\b})$ 
as a $\ZZ[\z_{\Bn}]$-module generated by
symbols $[A, \f]$ associated to the isomorphism class of the object $(A, \f) \in \SP_{\b}$
subject to the relations
\begin{enumerate}
\item \ $[X] = [X'] + [X'']$ if $X \isom X' \oplus X''$,
\item \ $[A, \z_{\Bn}\f] = \z_{\Bn}[A, \f]$,
\item \ $[X] = 0$ if $X$ is traceless.  
\end{enumerate}  

For an abelian category $\SL_{\b}$, $K(\SL_{\b})$ is defined similarly,
but by replacing the condition (i) by
\par\medskip
(i$'$) \ $[X] = [X'] + [X'']$ if there exists a short exact sequence
\begin{equation*}
  \begin{CD}
    0 @>>>  X' @>>>  X  @>>>  X'' @>>> 0.
  \end{CD}  
\end{equation*}  
\par
$K(\SP_{\b})$ and $K(\SL_{\b})$ are analogues of the Grothendieck group. 
$K(\SP_{\b})$ and $K(\SL_{\b})$ have structures of $\ZZ[\z_{\Bn}, q, q\iv]$-modules.
We set
\begin{equation*}
  K(\SP) = \bigoplus_{\b \in Q_+^{\s}}K(\SP_{\b}), \qquad
  K(\SL) = \bigoplus_{\b \in Q_+^{\s}}K(\SL_{\b}). 
\end{equation*}  
Thus $K(\SP)$ and $K(\SL)$ have structure of $\ZZ[\z_{\Bn}, q, q\iv]$-modules.
\para{5.4.}
Assume that $(P, \f), (Q, \f') \in \SP_{\b}$. Then for
$\f: \s^*P \isom P, \f': \s^*Q \isom Q$, $\f \otimes \f'$ induces a map
$(\s^*P)^{\psi} \otimes_{R(\b)} \s^*Q \to P^{\psi}\otimes_{R(\b)} Q$. 
Since we have canonical isomorphisms $\s^*P \simeq P, \s^*Q \simeq Q$ as vector spaces,
$\f\otimes \f'$ is regarded as a linear transformation on $P^{\psi}\otimes Q$. 
We set
\begin{equation*}
\tag{5.4.1}  
  ([P,\f], [Q, \f']) = \sum_{d \in \BZ}
         \Tr\bigl(\f \otimes \f', (P^{\psi}\otimes_{R(\b)}Q)_d\bigr)q^d.
\end{equation*}
(5.4.1) induces a symmetric bilinear pairing
$(\ ,\ ) : K(\SP_{\b}) \times K(\SP_{\b}) \to \BZ[\z_{\Bn}]((q))$.

\par
The following result, due to \cite[Thm. 3.6]{M}, gives a classification
of simple objects in $\SC_{\b}$.  

\begin{prop}  %%%  Prop. 5.5
Any simple object in $\SC_{\b}$ is either traceless or
of the form $(L, \f)$ with $L$ a simple $R(\b)$-module.
\end{prop}  

\para{5.6.}
For an $R(\b)$-module $M$ and an $R(\g)$-module $N$, define an $R(\b + \g)$-module 
$M\circ N$ by
\begin{equation*}
\tag{5.6.1}  
M\circ N = R(\b + \g)e(\b,\g)\otimes_{R(\b)\otimes R(\g)}M \otimes N, 
\end{equation*}
where $e(\b,\g) = \sum_{\nu \in I^{\b}, \nu' \in I^{\g}}e(\nu,\nu')$. 
\par
Now assume that $\b, \g \in Q_+^{\s}$, and take
$(M,\f) \in \SC_{\b}, (N, \f') \in \SC_{\g}$.
We define a convolution product of $(M, \f)$ and $N, \f')$ by 
\begin{equation*}
\tag{5.6.2}  
(M,\f) \circ (N,\f') = (M\circ N, \f\circ \f'),
\end{equation*}  
where $\f\circ \f'$ is the composite of the map
$\s^*M \circ \s^*N \to M \circ N$ induced from $\f \otimes \f'$ by (5.6.1), 
and the canonical isomorphism $\s^*(M \circ N) \simeq \s^*M \circ \s^*N$. 
\par
For $\b,\g \in Q_+^{\s}$, the automorphism $\s$ on $R(\b)$ and $R(\g)$ induces
an automorphism $\s$ on $R(\b) \otimes R(\g)$ by $\s(u \otimes w) = \s(u)\otimes \s(w)$.
Thus one can consider the category of graded representations of
$(R(\b)\otimes R(\g))\sharp(\ZZ/\Bn\ZZ)$, which we denote by $\SC_{\b \sqcup \g}$.
We can define full subcategories $\SP_{\b \sqcup \g}, \SL_{\b \sqcup \g}$
of $\SC_{\b\sqcup \g}$ similarly to $\SP_{\b}$ and $\SL_{\b}$ for $\SC_{\b}$. 
\par
The convolution product (5.6.2) gives an induction functor
$\Ind_{\b,\g} : \SC_{\b\sqcup\g} \to \SC_{\b+\g}$.
On the other hand, for each $(M, \f) \in \SC_{\b+\g}$, we define
\begin{equation*}
\tag{5.6.3}  
\Res_{\b,\g}(M,\f) = (e(\b,\g)M, \f').
\end{equation*}
Here $\Res M = e(\b,\g)M$ has a natural structure of an $R(\b)\otimes R(\g)$-module,
and we have the canonical isomorphism $\s^*(\Res M) \isom \Res (\s^*M)$ since
$e(\b,\g)$ is $\s$-invariant, which induces an isomorphism
$\f' : \s^*(\Res M) \isom \Res (\s^*M) \isom \Res M$.  
One can check that $\Res_{\b,\g}(M,\f)$ belongs to $\SC_{\b \sqcup \g}$, and we obtain
the restriction functor $\Res_{\b, \g} : \SC_{\b + \g} \to \SC_{\b \sqcup \g}$. 
\par
The induction functor induces functors,
$\Ind_{\b,\g}: \SP_{\b \sqcup \g} \to \SP_{\b + \g}$ and
$\Ind_{\b,\g}: \SL_{\b \sqcup \g} \to \SL_{\b + \g}$.
Also the restriction functor induces functors,
$\Res_{\b,\g}: \SP_{\b + \g} \to \SP_{\b \sqcup \g}$ and
$\Res_{\b,\g}: \SL_{\b + \g} \to \SL_{\b \sqcup \g}$ (see \cite [Thm. 4.6]{M}).
\par
The following is known by \cite [Thm. 4.3]{M}.

\begin{prop}  %%%%  Prop. 5.7. 
The induction functor and the restriction
functor form an adjoint pair of exact functors between
$\SC_{\b \sqcup \g}$ and $\SC_{\b + \g}$, namely we have
\begin{align*}
\tag{5.7.1}  
 (\Res_{\b,\g}\wt M, \wt N) = (\wt M, \Ind_{\b,\g}\wt N)
\end{align*}
for $\wt M  = (M,\f) \in \SC_{\b+\g}, \wt N  = (N, \f') \in \SC_{\b \sqcup \g}$. 
\end{prop}

\para{5.8.}
We have isomorphisms of $\ZZ[\z_{\Bn}, q,q\iv]$-modules
\begin{equation*}
  K(\SP_{\b})\otimes_{\ZZ[\z_{\Bn},q^{\pm 1}]}K(\SP_{\g}) \simeq K(\SP_{\b \sqcup \g}), \quad
  K(\SL_{\b})\otimes_{\ZZ[\z_{\Bn},q^{\pm 1}]}K(\SL_{\g}) \simeq K(\SL_{\b \sqcup \g}),   
\end{equation*}
and the induction functors $\Ind_{\b,\g}$ induce homomorphisms
\begin{equation*}
  K(\SP_{\b})\otimes_{\ZZ[\z_{\Bn}, q^{\pm 1}]}K(\SP_{\g}) \to K(\SP_{\b+\g}), \quad
  K(\SL_{\b}) \otimes_{\ZZ[\z_{\Bn}, q^{\pm 1}]}K(\SL_{\g}) \to K(\SL_{\b+\g}).
\end{equation*}

$K(\SP)$ and $K(\SL)$ turn out to be associative algebras over $\ZZ[\z_{\Bn}, q, q\iv]$
with respect to this product. 

\para{5.9.}
Let $P$ be a finitely generated projective $R(\b)$-module. We define a dual
module $\BD P$ by
\begin{equation*}
\BD P = \Hom_{R(\b)}(P, R(\b)) := \bigoplus_{n \in \ZZ}\Hom_{R(\b)}(q^nP, R(\b))_0,
\end{equation*}
where $\Hom_{R(\b)}(-,-)_0$ is the space of degree preserving homomorphisms.
$\BD P$ is a graded $\Bk$-vector space. $R(\b)$ acts on $\BD P$ by
$x(\la)(m) = \la(\psi(x)m)$ for $x \in R(\b), \la \in \BD P, m \in P$. 
Then $\BD P$ is a finitely generated projective $R(\b)$-module. 

\par
Let $M$ be a finite dimensional $R(\b)$-module. We define a dual $\BD M$ by
\begin{equation*}
\BD M = \Hom_{\Bk}(M, \Bk) := \bigoplus_{n \in \ZZ}\Hom_{\Bk}(q^nM, \Bk)_0.
\end{equation*}
Then $\BD M$ has a structure of $R(\b)$-module as in the case $\BD P$, and
$\BD M $ is a finite dimensional $R(\b)$-module.
\par
$\BD$ is a contravariant functor on $R(\b)\Mod$, and for any morphism
$f : M \to N$ between finitely generated projective, or finite dimensional
$R(\b)$-modules, there is an induced morphism $\BD(f) : \BD N \to \BD M$.
Now we define the dual of an object $(M, \f)$ in $\SL_{\b}$ or $\SP_{\b}$
by the formula
\begin{equation*}
\BD(M,\f) = (\BD M, (\BD(\f)\iv),
\end{equation*}
where $\BD(\f) : \BD M \isom \BD(\s^*M) = \s^*(\BD M)$.
Then $\BD$ gives a contravariant functor on $\SL_{\b}$ and $\SP_{\b}$
(see \cite[5]{M}).

\par
$(M,\f)$ in $\SL_{\b}$ or in $\SP_{\b}$ is said to be self-dual if
$\BD(M,\f) \simeq (M,\f)$ in that category. Note that if $(M,\f)$ is self-dual,
then $\f$ is uniquely determined by $M$ in the case where $\Bn$ is odd, is unique up to
$\pm 1$ in the case where $\Bn$ is even.
\par
Let $\wt{\ul\bB}_{\b}^*$ be the set of elements in $K(\SL_{\b})$ consisting of
isomorphism classes of self-dual objects $(M,\f)$ such that $M$ is a finite dimensional
simple $R(\b)$-module. It was proved by \cite[Thm.10.8]{M}, by investigating
the crystal structure of
$\wt{\ul\bB}^* = \bigsqcup_{\b \in Q_+^{\s}}\wt{\ul\bB}_{\b}^*$,
that there exists a $\ZZ[\z_{\Bn}, q,q\iv]$-basis $\ul\bB^*_{\b}$ of $K(\SL_{\b})$
such that $\ul\bB^*_{\b} \subset \wt{\ul\bB}_{\b}^*$.
Then the $\ZZ[\z_{\Bn}, q,q\iv]$-basis $\ul\bB_{\b}$ of $K(\SP_{\b})$ is defined
as the set of projective cover of elements in $\ul\bB_{\b}^*$, which  
consist of finitely generated projective indecomposable, self-dual objects in $\SP_{\b}$.  
We set $\ul\bB = \bigoplus_{\b \in Q_+^{\s}}\ul\bB_{\b}$. 

\para{5.10.}
Let $j= \{i_1, \dots, i_t\} \in J$ be a $\s$-orbit in $I$, and set
$\a_j = \a_{i_1} + \cdots + \a_{i_t}$. Then $\a_j \in Q_+^{\s}$.
The algebra $R(\a_j)$ has a unique irreducible module $L_j$.  Let $P_j$
be the projective cover of $L_j$.  $L_j$ and $P_j$ are constructed
explicitly in \cite[7]{M}.
In particular, $P_j$ is written as $P_j = \bigoplus_{w \in S_t}\Bk[x_1, \dots, x_t][w]$,
where $[w] = \t_we(i_1, \dots, i_t)$.  Thus $P_j$ is a free $\Bk[x_1, \dots, x_t]$-module
with basis $[w]$, where $e(\nu')$ acts on $[w]$ by 1 if $\nu' = w\nu$, and by 0 otherwise,
$\t_i$ acts on $[w]$ by $[s_i(w)]$.
$L_j \simeq \bigoplus_{w \in S_t}\Bk[w]$, where $e(\nu')$ and $\t_i$ acts on $[w]$ as above,
and $x_k$ acts on $[w]$ as 0.
\par
Since $\a_j$ is $\s$-stable, one can consider the categories $\SC_{\a_j}, \SL_{\a_j}$ and
$\SP_{\a_j}$. 
We define $\f_j : \s^*L_j \isom L_j$ by $\f_j[w] = [\s(w)]$, (here $[w]$ is identified
with the permutation $(i_1', \dots, i_t')$ of $(i_1, \dots, i_t)$ on which $\s$ acts).
Define $L(j) = (L_j, \f_j)$, which is a simple object in $\SL_{\a_j}$. We define
$P(j)$ as the projective cover of $L(j)$, which lies in $\SP_{\a_j}$.
$P(j)$ is written as $P(j) = (P_j, \f'_j)$, where $P_j$ is the projective cover of
$L_j$, and $\f'_j :\s^*P_j \isom P_j$
is induced from  $\f_j : \s^*L_j \isom L_j$.
\par
For $n \ge 1$, we define $L(j)^{(n)} \in \SL_{n\a_j}$ by
\begin{equation*}
\tag{5.10.1}  
L(j)^{(n)} = q_j^{\binom{n}{2}}L(j)^{\circ n},
\end{equation*}
where $A^{\circ n}$ denotes the convolution product
$A \circ A \circ \cdots \circ A$ ($n$-times).
By \cite[Lemma 7.2]{M}, $L(j)^{(n)}$ is a self-dual object in $\SL_{n\a_j}$.
Note that $L(j)^{(n)} = (L_j^{(n)}, \f_{nj})$, where
$L_j^{(n)} = q_j^{\binom{n}{2}}L_j^{\circ n}$, and $\f_{nj}: \s^*L_j^{(n)} \isom L_j^{(n)}$
is induced from $\f_j : \s^*L_j \isom L_j$. 
Let $P(j)^{(n)}$ be the projective cover of $L(j)^{(n)}$.
We have $P(j)^{(n)} = (P_j^{(n)}, \f'_{nj})$, where
$P_j^{(n)}$ is the projective cover of $L_j^{(n)}$,
and $\f'_{nj}: \s^*P_j^{(n)} \isom P_j^{(n)}$ is induced from
$\f_{nj}: \s^*L_j^{(n)} \isom L_j^{(n)}$. 

\para{5.11.}
Recall that $K(\SP)$ is an associative algebra over $\ZZ[\z_{\Bn}, q,q\iv]$
with basis $\ul\bB$. Let $\wt\BA$ be the smallest subring of $\ZZ[\z_{\Bn}, q, q\iv]$
containing $\BA = \ZZ[q,q\iv]$
such that all the structure constants of the algebra $K(\SP)$ with respect to $\bB$
are contained in $\wt\BA$.  Let ${}_{\wt\BA}K(\SP)$ be the $\wt\BA$-submodule of $K(\SP)$
spanned by $\ul\bB$.  Then ${}_{\wt\BA}K(\SP)$ is the $\wt\BA$-subalgebra of
$K(\SP)$.
Set ${}_{\wt\BA}\ul\BU_q^- = \wt\BA\otimes_{\BA}{}_{\BA}\ul\BU_q^-$.
\par
We extend the bilinear pairing $(\ ,\ )$ on $K(\SP_{\b})$ to the bilinear pairing
on $K(\SP)$ by defining $([M], [N]) = 0$ if $M \in \SP_{\b}, N \in \SP_{\g}$
with $\b \ne \g$. 
\par
The following result was proved by \cite [Thm.6.1]{M}.

%%%
%%%
\begin{thm}  %%%%  Thm. 5.12
There exists a unique $\wt\BA$-algebra isomorphism
\begin{equation*}  
\g : {}_{\wt\BA}\ul\BU_q^- \isom {}_{\wt\BA}K(\SP)
\end{equation*}
such that $\g(f_j^{(n)}) = [P(j)^{(n)}]$. Moreover, $\g$ is an
isometry with respect to the pairing on $K(\SP)$ given in 5.11 and
the inner product on $\ul\BU_q^-$ defined in 1.3.  
\end{thm}

\remarks{5.13.} 
(i) \ It was proved in \cite[Lemma 11.3]{M} that $\wt\BA$ is contained
in $\ZZ[\z_{\Bn} + \z_{\Bn}\iv, q,q\iv]$. Hence if $\ul\BU_q^-$ is
irreducible of finite type, then we have $\wt\BA = \BA$. 
In the general case, for any $\ul\BU_q^-$ of Kac-Moody type, it was
proved by \cite{MSZ2} that $\wt\BA = \BA$.  Thus we have an isomorphism
\begin{equation*}
\tag{5.13.1}  
  \g : {}_{\BA}\ul\BU_q^- \isom {}_{\BA}K(\SP).
\end{equation*}  
\par
(ii)  \ Assume that $\Bk$ is an algebraically closed field of
characteristic zero.  It was shown in \cite [Thm. 5.12]{MSZ2} that, under the isomorphism
$\g$ in (5.13.1), 
$\g\iv(\ul\bB)$ coincides with Kashiwara's global crystal basis of $\ul\BU_q^-$
given in \cite {Ka},
which also coincides with Lusztig's canonical basis given in \cite {L-book}.
Hence for $\ul\BU_q^-$ of Kac-Moody type, Lusztig's canonical basis coincides with
Kashiwara's global crystal basis. 

\para{5.14.}
Let $\b_1, \dots, \b_s, \g_1, \dots, \g_t \in Q_+^{\s}$ be such that
$\sum_i\b_i = \sum_j\g_j$.
The induction functor $\Ind_{\b_1, \dots, \b_s} :
\SC_{\b_1 \sqcup \cdots \sqcup \b_s} \to  \SC_{\b_1 + \cdots + \b_s}$,
and the restriction functor
$\Res_{\g_1, \dots, \g_t}: \SC_{\g_1 + \cdots + \g_t} \to \SC_{\g_1\sqcup \cdots \sqcup \g_t}$
can be defined. 
Thus we have an exact functor

\begin{equation*}
\tag{5.14.1}
\Res_{\g_1,\dots, \g_t}\circ \Ind_{\b_1, \dots, \b_s}
  : \SC_{\b_1\sqcup \cdots \sqcup \b_s} \to \SC_{\g_1 \sqcup \cdots \sqcup \g_t}. 
\end{equation*}  

In the case where $\s$ is trivial, it was proved in \cite[Prop. 2.18]{KL1} that
this composite functor has a natural filtration by exact functors, which is
called the Mackey filtration.  The Mackey filtration was extended by
McNamara \cite{M} to the case where $\s$ is non-trivial.
Note that a functor $\SC_{\b} \to \SC_{\g}$
is said to be traceless if its image lies in the full subcategory
of traceless objects in $\SC_{\g}$.  

\begin{thm} [{\cite[ Thm. 4.5]{M}}]  %%%%  Theorem 5.15
Assume that $\sum_{1 \le i \le s} \b_i = \sum_{1 \le j \le t}\g_j$.  
Then the functor $\Res_{\g_1,\dots, \g_t}\circ \Ind_{\b_1, \dots, \b_s}$ has a natural filtration
by exact functors.   
The subquotient functors in this filtration which are not traceless
are indexed by a matrix $\Bxi = (\xi_{ij})_{1 \le i \le s, 1 \le j \le t}$, with
$\xi_{ij} \in Q_+^{\s}$, satisfying the condition that
\begin{equation*}
\tag{5.15.1}  
\b_i = \sum_{1 \le j \le t}\xi_{ij},   \qquad \g_j = \sum_{1 \le i \le s}\xi_{ij}. 
\end{equation*}
The functor attached to $\Bxi$ is isomorphic, up to a grading shift, to the composition
\begin{equation*}
\tag{5.15.2}
\Ind_{\Bxi}^{\g} \circ \tau \circ \Res_{\Bxi}^{\b}.
\end{equation*}
Here, $\Res^{\b}_{\Bxi} : \otimes_i \SC_{\b_i} \to \otimes_i(\otimes_j\SC_{\xi_{ij}})$
is given by the tensor product of
$\Res_{\xi_{i\bullet}} : \SC_{\b_i} \to \otimes_j\SC_{\xi_{ij}}$,
$\tau : \otimes_i(\otimes_j\SC_{\xi_{ij}}) \to \otimes_j(\otimes_i\SC_{\xi_{ij}})$ is
given by permuting the tensor factors,
and $\Ind_{\Bxi}^{\g} : \otimes_j(\otimes_i\SC_{\xi_{ij}}) \to \otimes_j\SC_{\g_j}$
is the tensor product of $\Ind_{\xi_{\bullet j}} : \otimes_i\SC_{\xi_{ij}} \to \SC_{\g_j}$.
\end{thm}
\begin{proof}
Following McNamara \cite{M}, we give an outline of the proof of Theorem 5.15.
Set $\a = \sum_{1 \le i \le s}\b_i = \sum_{1 \le j \le t}\g_j$ and $n = |\a|$.
Let $B = R(\a)$, and regard it as an
$(\otimes_jR(\g_j), \otimes_iR(\b_i))$-bimodule,
where $\otimes_iR(\b_i) = R(\b_1)\otimes \cdots \otimes R(\b_s)$, and
$\otimes_jR(\g_j) = R(\g_1)\otimes\cdots\otimes R(\g_t)$. 
Then for $M \in \SC_{\b_1\sqcup\cdots\sqcup\b_s}$, we have

\begin{align*}
  \Res_{\g_1, \dots,\g_t}\circ \Ind_{\b_1, \dots,\b_s}M
  &=  e(\g_1,\dots, \g_t)R(\a)e(\b_1,\dots, \b_s)
             \otimes_{\otimes_iR(\b_i)}M  \\
  &= B \otimes_{\otimes_iR(\b_i)}M,
\end{align*}
where
\begin{equation*}
e(\b_1, \dots, \b_s) = \sum_{\mu_1, \dots, \mu_s} e(\mu_1, \dots, \mu_s)
\end{equation*}
and $e(\mu_1, \dots, \mu_s) \in I^{\a}$ is a juxtaposition of
$\mu_k \in I^{\b_k}$.  $e(\g_1, \dots, \g_t)$ is defined similarly. 
\par
Let $\Bxi = (\xi_{ij})$ be the matrix given in (5.15.1).
We define a permutation $w = w(\Bxi) \in S_n$
as follows.
We arrange $\Bxi = (\xi_{ij})$ along the lexicographic order on $(i,j) \in \NN^2$,
\begin{equation*}
\tag{5.15.3}  
\xi_{11}, \dots, \xi_{1t}, \xi_{21}, \dots, \xi_{2t}, \dots, \xi_{s1}, \dots, \xi_{st},
\end{equation*}
and consider the partition $[1,n] = \bigsqcup_{i,j}A_{ij}$ into segments
associated to the sequence (5.15.3) with
\begin{equation*}
  A_{ij} = \{ x_{ij} +1, x_{ij}+ 2, \dots, x_{ij} + |\xi_{ij}|\}
\end{equation*}
in this order,
where $x_{ij} = \sum_{(i',j') < (i,j)}|\xi_{i'j'}|$.
We define $w = w(\Bxi) \in S_n$ by
a permutation sending the sequence $1,2, \dots, n$ to the sequence
\begin{equation*}
  A_{11}, \dots, A_{s1}, A_{12}, \dots, A_{s2},
       \dots, A_{1t}, \dots, A_{st}.
\end{equation*}
Define the element $u(\Bxi) = \tau_{w(\xi)} \in R(\a)$. 
Since $e(\g_1,\dots, \g_t)R(\a)e(\b_1, \dots, \b_s)$ is generated
by $\t_w$, as an $(\otimes_jR(\g_j), \otimes_iR(\b_i))$-bimodule, 
where
$w \in S_{\g_1, \dots, \g_t}\backslash S_n/S_{\b_1, \dots, \b_s}$
is given as $w = w(\Bxi)$,  
$B$ is generated as a bimodule by the elements $u(\Bxi)$.
\par
We define a partial order on the set of all the matrices $\Bxi = (\xi_{ij})$
as above, by the condition that   
$\Bxi \le \Bxi'$ if and only if there exists $(i,j)$ such that
$\xi_{i'j'} = \xi'_{i'j'}$
for any $(i',j') < (i,j)$ and that $\xi_{ij} \ne \xi'_{ij}$ with 
$\xi_{ij} - \xi'_{ij} \in Q_+$.

Let $B(\le \Bxi)$ (resp. $B(< \Bxi)$)  be the subbimodule of $B$
generated by all $u(\Bz)$ with $\Bz \le \Bxi$
(resp. $\Bz < \Bxi$). 
We define $B(\Bxi) = B(\le \Bxi)/B(< \Bxi)$. 
Then the functor of tensoring with the subquotient $B(\Bxi)$ coincides with
the composition $\Ind_{\xi}^{\g}\circ \tau \circ \Res_{\xi}^{\b}$, up to the grading
shift.  In particular, it is exact, and we have a filtration by exact subfunctors. 
In the case where $\s(\Bxi) = \Bxi$, the functor $B(\Bxi)\otimes -$ coincides with
the functor $\Ind_{\Bxi}^{\g}\circ \tau\circ \Res_{\Bxi}^{\b}$ from
$\SC_{\b_1\sqcup \dots \sqcup \b_s}$ to $\SC_{\g_1\sqcup\cdots \sqcup \g_t}$.
If $\s(\Bxi) \ne \Bxi$, consider $B(\Bxi) \oplus B(\s\Bxi) \oplus \cdots \oplus B(\s^{t-1}\Bxi)$,
where $t$ is the minimal positive integer such that $\s^t\Bxi = \Bxi$. 
As in \cite{M}, this produces a subquotient functor which is traceless. 
\end{proof}

\para{5.16}
Let $P(j)^{(n)}$ be as in 5.10.
Then we have $P(j)^{(n)} \in \SP_{n\a_j}$. 
We fix an integer $m \ge 1$. 
For $\Bj = (j_1, \dots, j_m) \in J^m, \Bc =  (c_1, \dots, c_m) \in \NN^m$,
define
\begin{equation*}
\tag{5.16.1}  
P_{\Bj,\Bc} = P(j_1)^{(c_1)}P(j_2)^{(c_2)}\cdots P(j_m)^{(c_m)}. 
\end{equation*}
Choose another $\Bj' = (j'_1, \dots, j'_m) \in J^m, \Bc' = (c'_1, \dots, c'_m) \in \NN^m$,
and define $P_{\Bj',\Bc'}$ similarly.
Note that $P_{\Bj,\Bc} \in \SP_{\b}$ with $\b = \sum_{k=1}^mc_k\a_{j_k}$. 
Set $\b = \weit(P_{\Bj,\Bc})$. 
We shall compute $([P_{\Bj,\Bc}], [P_{\Bj'\Bc'}])$ in $K(\SP)$. 
By (1.3.1) and Theorem 5.12, $([P_{\Bj, \Bc}], [P_{\Bj',\Bc'}]) = 0$
if $\weit(P_{\Bj,\Bc}) \ne \weit(P_{\Bj',\Bc'})$.
Hence we may assume that $\weit(P_{\Bj,\Bc}) = \weit(P_{\Bj',\Bc'})$,
namely, $\sum_kc_k\a_{j_k} = \sum_kc'_k\a_{j'_k}$. 
\par
\par
For a given $(\Bj, \Bc)$ as above, we define
$\ul\nu = (\ul\nu_1, \dots, \ul\nu_n) \in J^n$
by
\begin{equation*}
\tag{5.16.2}  
(\ul\nu_1, \ul\nu_2, \dots, \ul\nu_n) =
  (\underbrace {j_1, \dots, j_1}_{c_1\text{-times}},
  \underbrace{j_2, \dots, j_2}_{c_2\text{-times}}, \dots,
  \underbrace{j_m, \dots, j_m}_{c_m\text{-times}}),
\end{equation*}
where $n = \sum_{1 \le k \le m}c_k$, and set 
\begin{equation*}
\tag{5.16.3}
P_{\ul\nu} = P(\ul\nu_1)P(\ul\nu_2) \cdots P(\ul\nu_n). 
\end{equation*}  
Then $P_{\ul\nu} \in \SP_{\b}$ with $\b = \sum_{1 \le i \le n}\a_{\ul\nu_i} \in Q_+^{\s}$.
For another pair $(\Bj',\Bc')$ such that
$\sum_kc_k\a_{j_k} = \sum_kc'_k\a_{j'_k}$,
we define
$\ul\nu' = (\ul\nu'_1, \dots, \ul\nu'_n) \in J^n$ in a similar way as in (5.16.2),
where $\sum_{1 \le k \le m}c'_k = n$, and define $P_{\ul\nu'}$ in a similar way as in (5.16.3). 
Note that $f_j^{(n)} = ([n]^!_j)\iv f_j^n$ in $\ul\BU_q^-$. 
Hence by Theorem 5.12,
we have
\begin{equation*}
[P(j)^{(n)}] = ([n]^!_j)\iv [P(j)^{\circ n}]
\end{equation*}
in $\QQ(q)\otimes_{\BA}{}_{\BA}K(\SP)$. 
  Thus we have a relation  in $\QQ(q)\otimes_{\BA}{}_{\BA}K(\SP)$,
\begin{equation*}
\tag{5.16.4}  
[P_{\Bj, \Bc}] = \biggl(\prod_{k = 1}^m [c_k]_{j_k}^!\biggr)\iv [P_{\ul\nu}].
\end{equation*}
A similar formula holds also for $[P_{\Bj',\Bc'}]$.
Thus for $\Bc, \Bc' \in \NN^m$, we have
\begin{equation*}
\tag{5.16.5}  
([P_{\Bj, \Bc}, [P_{\Bj', \Bc'}])
    = \biggl(\prod_{k=1}^m [c_k]_{j_k}^![c'_k]_{j'_k}^!\biggr)\iv
       ([P_{\ul\nu}], [P_{\ul\nu'}]).
\end{equation*}  

\para{5.17.}
We now compute $(P_{\ul\nu}, P_{\ul\nu'})$. By the adjointness  property of
induction functors and restriction functors (Proposition 5.7), we have
\begin{align*}
\tag{5.17.1}  
  (P_{\ul\nu}, P_{\ul\nu'}) &= \bigl(\Ind_{\a_{\ul\nu_1}, \dots, \a_{\ul\nu_n}}
      P({\ul\nu_1})\otimes \cdots \otimes P(\ul\nu_n),
      \Ind_{\a_{\ul\nu'_1}, \dots, \a_{\ul\nu'_n}} P(\ul\nu'_1)\otimes
                     \cdots\otimes P(\ul\nu'_n)\bigr) \\
  &= \bigl(\Res_{\a_{\ul\nu'_1}, \dots, \a_{\ul\nu'_n}}
               \Ind_{\a_{\ul\nu_1}, \dots, \a_{\ul\nu_n}}
                 P(\ul\nu_1)\otimes \cdots \otimes P(\ul\nu_n),
                 P(\ul\nu'_1)\otimes\cdots\otimes P(\ul\nu'_n)\bigr).
\end{align*}  

We compute $\Res_{\a_{\ul\nu'_1}, \dots, \a_{\ul\nu'_n}}
  \Ind_{\a_{\ul\nu_1}, \dots, \a_{\ul\nu_n}}P(\ul\nu_1)\otimes \cdots \otimes P(\ul\nu_n)$
by using the Mackey filtration.
Take 
$\ul\nu = (\ul\nu_1, \dots, \ul\nu_n),\nu' = (\ul\nu'_1, \dots, \ul\nu'_n) \in J^n$.  
Let $\Xi = \Xi(\ul\nu, \ul\nu')$ be the set of matrices
$\Bxi = (\xi_{ij})_{1\le i,j \le n}$,
where $\xi_{ij} \in Q_+^{\s}$ satisfies the condition that
\begin{equation*}
\tag{5.17.2}  
\sum_{1 \le j \le n} \xi_{ij} = \a_{\ul\nu_i},  \quad \sum_{1 \le i \le n}\xi_{ij} = \a_{\ul\nu'_j}. 
\end{equation*}  

Take $\Bxi \in \Xi$. Since $\a_{\ul\nu_i}, \a_{\ul\nu'_j}$ are simple roots in $\ul\vD_+$,
$\Bxi$ is a permutation matrix, namely,
for each $1 \le i \le n$, there exists a unique $j$ such that $\xi_{ij} \ne 0$,
in which case $\xi_{ij}$ coincides with $\a_{\ul\nu_i} = \a_{\ul\nu'_j}$. 
Thus $i \mapsto j = w(i)$ determines a permutation $w = w(\Bxi) \in S_n$.
\par
In applying the Mackey filtration in Theorem 5.15, we need to determine
the grading shift of the functor $\tau$ explicitly.
In general, the following result is known.
\par\medskip\noindent
(5.17.3) \ Assume that $\nu = (\nu_1, \dots, \nu_n) \in I^{\b}$
with $|\b| = n$  for $\b \in Q_+$.  Take $w \in S_n$.  Then
\begin{equation*}
  \deg e(\nu)\t_w = -\sum_{\substack{k < \ell  \\ w(k) > w(\ell)}}
                        (\a_{\nu_{w(k)}}, \a_{\nu_{w(\ell)}}).
\end{equation*}  

In order to apply (5.17.3) for our case, 
we modify $w = w(\Bxi) \in S_n$ as follows.
Set $\ul \a = \sum_{i=1}^n\a_{\ul\nu_i} \in \ul Q_+$, and let $\a \in Q_+$
be the element corresponding to $\ul \a$ under the identification $\ul Q_+ = Q^{\s}_+$.
We set $n' = |\a|$. 

We consider $\nu = (\nu_1, \dots, \nu_{n'}) \in I^{\a}$, and 
$e(\nu) \in R(\a)$.
For each $\ul\nu = (\ul\nu_1, \dots, \ul\nu_n) \in J^n$ as above,
we define $e(\ul\nu) \in R(\a)$ as

\begin{equation*}
\tag{5.17.4}  
e(\ul\nu) = \sum_{\mu_1, \dots, \mu_n}e(\mu_1, \dots, \mu_n),
\end{equation*}
where $e(\mu_1, \dots, \mu_n) \in R(\a)$ is a juxtaposition of
$\mu_i \in I^{\a_{\ul\nu_i}}$,
and the sum runs over all such $\mu_1, \dots, \mu_n$. 
We choose $e(\nu) = e(\nu_1, \dots, \nu_{n'}) \in I^{\a}$
appearing as $e(\mu_1, \dots, \mu_n)$ in $e(\ul\nu)$ above,  
and choose $e(\nu') = e(\nu'_1, \dots,\nu'_{n'}) \in I^{\a}$
appearing in $e(\ul\nu')$, similarly. 
\par
Associated to the matrix $\Bxi = (\xi_{ij}) \in \Xi(\ul\nu, \ul\nu')$,
we construct a block matrix $\vD(\Bxi) = (\vD_{ij})$
of size $n'$ as follows; $\vD_{ij}$ is a matrix of size $|\ul\nu_i| \times |\ul\nu_j'|$ 
such that
\begin{equation*}
\vD_{ij} = \begin{cases}
              \Id_{|\ul\nu_i|}   &\quad\text{ if } \xi_{ij} = \ul\nu_i = \ul\nu_j', \\
              0                  &\quad\text{ if } \xi_{ij} = 0. 
           \end{cases}      
\end{equation*}
Then $\vD(\Bxi)$ is a permutation matrix, and one can define $w' = w'(\Bxi) \in S_{n'}$
in a similar way as $w = w(\Bxi)$ is defined from $\Bxi$.
Note that $w'$ gives a distinguished representative of the double coset
$S_{\nu_1', \dots, \nu'_{n'}}\backslash S_{n'}/S_{\nu_1, \dots, \nu_{n'}}$.  
\par
We fix $e(\nu) = e(\nu_1, \dots, \nu_{n'}) \in R(\a)$ appearing in (5.17.4).
Then by (5.17.3), for $w' = w'(\Bxi) \in S_{n'}$, we have
\begin{equation*}
\tag{5.17.5}  
  \deg \t_{w'}e(\nu) = \deg e({w'}\iv(\nu))\t_{w'}
  = -\sum_{\substack{1 \le i < j \le n'  \\ w'(i) > w'(j)}}
        (\a_{\nu_i}, \a_{\nu_j}). 
\end{equation*}  

We define, for $\Bxi \in \Xi(\ul\nu, \ul\nu')$ and $w = w(\Bxi) \in S_n$, 
\begin{equation*}
\tag{5.17.6}  
  A(\Bxi) = \sum_{\substack{1 \le k < \ell \le n \\ w(k) > w(\ell)}}
                 (\a_{\ul\nu_k}, \a_{\ul\nu_{\ell}}).
\end{equation*}  

\begin{lem}   %%%%  Lemma 5.18
For any $e(\nu)$ appearing in $e(\ul\nu)$, we have
\begin{equation*}
\tag{5.18.1}
\deg\tau_{w'}e(\nu) = -A(\Bxi).
\end{equation*}  
In particular, $\deg\t_{w'}e(\nu)$ does not depend on the choice of $e(\nu)$,
and $\t_{w'}e(\ul\nu)$ is a homogeneous element, whose degree is given by $-A(\Bxi)$. 
\end{lem}  
\begin{proof}
Let $\ul\nu = (\ul\nu_1,\dots, \ul\nu_n)$.  Then $\nu = (\nu_1, \dots, \nu_{n'})$
is written as
\begin{equation*}
\nu = (\underbrace{\nu_1, \dots, \nu_{|\ul\nu_1|}}_{\ul\nu_1\text{-part}},
 \underbrace{\nu_{|\ul\nu_1|+1}, \dots,
         \nu_{|\ul\nu_1| + |\ul\nu_2|}}_{\ul\nu_2\text{-part}}, \dots).
\end{equation*}
We write the part corresponding to $\ul\nu_k$
as $(\nu_s, \dots, \nu_{s'})$,
and the part corresponding to $\ul\nu_{\ell}$
as $(\nu_t, \dots, \nu_{t'})$.
Assume that $k < \ell$ and $w(k) > w(\ell)$. Then by definition of $w' \in S_{n'}$,
for any $i \in [s,s'], j \in [t,t']$, we have $i < j$ and $w'(i) > w'(j)$. 
The converse also holds. 
Since $\a_{\ul\nu_k} = \a_{\nu_s} + \cdots + \a_{\nu_{s'}},
               \a_{\ul\nu_{\ell}} = \a_{\nu_t} + \cdots + \a_{\nu_{t'}}$,
we have
\begin{equation*}
  (\a_{\ul\nu_k}, \a_{\ul\nu_{\ell}}) = \sum_{i \in [s,s'], j \in [t,t']}
         (\a_{\nu_i}, \a_{\nu_j}).
\end{equation*}  
It follows that
\begin{align*}
-A(\Bxi) =   
  -\sum_{\substack{ 1 \le k < \ell \le n  \\ w(k) > w(\ell)}}(\a_{\ul\nu_k}, \a_{\ul\nu_{\ell}})
    &= - \sum_{\substack{1 \le i < j \le n' \\ w'(i) > w'(j) }} (\a_{\nu_i}, \a_{\nu_j})
      = \deg\tau_{w'}e(\nu)
\end{align*}
by (5.17.5).  The lemma is proved.
\end{proof}  

\para{5.19.}
The weight of $P_{\Bj,\Bc}$ can be written as
$\weit(P_{\Bj,\Bc}) = \ul\a = \sum_{k=1}^mc_k\ul\a_{j_k} \in \ul Q_+$. 
We define $\d_{\ul\a} \in \BA$ by
\begin{equation*}
\tag{5.19.1}  
\d_{\ul\a} = \prod_{k=1}^m(1-q_{j_k}^2)^{c_k}.
\end{equation*}

We are now ready to prove

\begin{thm}  %%%%  Theorem.  5.20
Assume that
$\weit (P_{\Bj,\Bc}) = \weit (P_{\Bj',\Bc'}) = \ul\a$,
then    
\begin{equation*}
  ([P_{\Bj,\Bc}], [P_{\Bj', \Bc'}])
  = \biggl(\prod_{k=1}^m [c_k]_{j_k}^! [c'_k]_{j'_k}^!\biggr)\iv\d_{\ul\a}\iv
       \sum_{\Bxi \in \Xi(\ul\nu, \ul\nu')}q^{-A(\Bxi)}. 
\end{equation*}  
If $\weit (P_{\Bj,\Bc}) \ne \weit(P_{\Bj', \Bc'})$,
then $([P_{\Bj,\Bc}], [P_{\Bj',\Bc'}]) = 0$.
\end{thm}
\begin{proof}
It is enough to consider the case where $\weit(P_{\Bj,\Bc}) = \weit(P_{\Bj',\Bc'})$. 
We follow the discussion in 5.17. 
By the Mackey filtration in Theorem 5.15, we have
\begin{align*}
  \Res_{\a_{\ul\nu_1'}, \dots, \a_{\ul\nu'_n}}\Ind_{\a_{\ul\nu_1}, \dots, \a_{\ul\nu_n}}
        (P(\ul\nu_1)\otimes \cdots \otimes P(\ul\nu_n))
  &= \sum_{\Bxi \in \Xi(\ul\nu,\ul\nu')}\tau_{\Bxi}
                    (P(\ul\nu_1)\otimes \cdots \otimes P(\ul\nu_n)) 
\end{align*}
modulo the sum of traceless elements, and modulo the degree shift for each $\tau_{\Bxi}$. 
Here $\tau = \tau_{\Bxi}$ is as in Theorem 5.15.
Note that, in our case,  $e(\b_1, \dots, \b_s)$ given in the proof of Theorem 5.15
coincides with $e(\ul\nu)$, and $\tau_{\Bxi}$ coincides with $\tau_{w'}$.
By Lemma 5.18, $\tau_{w'}e(\ul\nu)$ is homogeneous with
$\deg\tau_{w'}e(\ul\nu) = -A(\Bxi)$. Hence 
\begin{align*}
\tau_{\Bxi}(P(\ul\nu_1)\otimes\cdots\otimes P(\ul\nu_n))
   &= \tau_{w'}e(\ul\nu)(P(\ul\nu_1)\otimes\cdots\otimes P(\ul\nu_n))  \\
   &= q^{-A(\Bxi)}P(\ul\nu_1')\otimes\cdots\otimes P(\ul\nu_n').  
\end{align*}

It follows, by (5.17.1),  that
\begin{align*}
\tag{5.20.1}  
(P_{\ul\nu}, P_{\ul\nu'}) &=  \sum_{\Bxi \in \Xi(\ul\nu,\ul\nu')}q^{-A(\Bxi)}
  \Bigl(P(\ul\nu'_1)\otimes \cdots \otimes P(\ul\nu'_n),
          P(\ul\nu'_1)\otimes \cdots \otimes P(\ul\nu'_n)\Bigr) \\
          &= \sum_{\Bxi \in \Xi(\ul\nu,\ul\nu')}q^{-A(\Bxi)}\prod_{i = 1}^n
             \bigl(P(\ul\nu'_i), P(\ul\nu'_i)\bigr).
\end{align*}  

It is known that $(P(j), P(j)) = (1- q_j^2)\iv$ for $j \in J$.  This was
proved in \cite{M} in the course of the proof of Theorem 5.12.  Thus we have 
\begin{align*}
\tag{5.20.2}  
\prod_{i=1}^n \bigl(P(\ul\nu'_i), P(\ul\nu'_i)\bigr)
            &= \prod_{i=1}^n \bigl(P(\ul\nu_i), P(\ul\nu_i)\bigr)  \\ 
                         &= \prod_{i = 1}^n(1 - q_{\ul\nu_i}^2)\iv
                               = \prod_{k = 1}^m (1 - q_{j_k}^2)^{-c_k} = \d_{\ul\a}\iv.
\end{align*}
The theorem now follows from (5.16.5), (5.20.1) and (5.20.2). 
\end{proof}  

\remark{5.21}
Theorem 5.20 can be proved directly by using the property of
the inner product of quantum groups given in 1.3.  The following discussion  
is due to the referee.  
Although Theorem 5.20 is proved
without using KLR algebras, we leave the discussion in this section unchanged,
since we think that the proof of Theorem 5.20 based on the Mackey formula
for KLR algebras has its own interest.
\par
Assume that $\ul\BU_q^-$ is the quantum group as in Theorem 5.20.
Let $f_{j_1}\cdots f_{j_m}$ and $f_{j'_1}\cdots f_{j'_m}$ be
monomials in $\ul\BU_q^-$.
The essential part of the formula in Theorem 5.20 is
the computation of the inner product $(f_{j_1}\cdots f_{j_m}, f_{j'_1}\cdots f_{j'_m})$.
Let $r : \ul\BU_q^- \to \ul\BU_q^- \otimes \ul\BU_q^-$ be the homomorphism
defined in 1.3.  
For $m \in \ZZ_{>0}$, let $r^{(m)} : \ul\BU_q^- \to (\ul\BU_q^-)^{\otimes (m+1)}$
be the map defined by
$r^{(m)} = (\id^{\otimes(m-1)}\otimes r)\circ \cdots \circ (\id\otimes r)\circ r$. 
Then by (1.3.2), we have
\begin{equation*}
\tag{5.21.1}  
  (f_{j_1}\cdots f_{j_m}, f_{j'_1}\cdots f_{j'_m})
              = (r^{(m-1)}(f_{j_1}\cdots f_{j_m}), f_{j'_1}\otimes \cdots \otimes f_{j'_m}).
\end{equation*}  
Note that $(\ul\BU_q^-)^{\otimes m+1}$ has a structure of an associative algebra,
defined similarly to $(\ul\BU_q^-)^{\otimes 2}$ in (1.3.1),
and $r^{(m)}$ turns out to be a homomorphism
with respect to this associative algebra structure.
We have $r^{(m-1)}(f_j) = \sum_{a = 1}^m f_{j;a}$, where
$f_{j;a} = 1^{\otimes a-1} \otimes f_j \otimes 1^{\otimes m-a}$.
By using the product rule in (1.3.1), it is easy to see, for $a > b$, that
\begin{equation*}
\tag{5.21.2}  
f_{j;a}f_{j';b} = q^{-(\a_j, \a_{j'})}f_{j';b}f_{j;a}.
\end{equation*}  

For $w \in S_m$, set
$A(w) = \sum_{1 \le k < l \le m, w(k)> w(l)}(\a_{j_k}, \a_{j_l})$. 
Then we have
\begin{equation*}
\tag{5.21.3}  
f_{j_1; w(1)}\cdots f_{j_m;w(m)} = q^{-A(w)}f_{j_{w\iv(1)};1} \cdots f_{j_{w\iv(m)};m}. 
\end{equation*}  

We prove (5.21.3) by induction on the length $l(w)$ of $w$.
Note that
\begin{equation*}
l(w) = \sharp\{ 1 \le k < l \le m \mid w(k) > w(l)\}. 
\end{equation*}
If $w = 1$, (5.21.3) certainly holds. Hence we assume that $w \ne 1$.
There exists $k' < l'$ such that $w(k) = l', w(k+1) = k'$ for some $k$.
Let $s = (k, k+1)$ be a transvection in $S_m$, and set $w' = ws$.
Then $l(w') = l(w) - 1$, and we have
\begin{align*}
\tag{5.21.4}  
  \{ 1 \le k_1 < l_1 &\le m \mid w(k_1) > w(l_1) \}  \\
  &= \{ 1 \le k_1 < l_1\le m  \mid w'(k_1) > w'(l_1)\}
          \sqcup \{k < k+1 \}.  
\end{align*}  
By applying (5.21.1) for $a = k+1, b = k$, we have
\begin{align*}
  &f_{j_1;w(1)}\cdots f_{j_m;w(m)}  \\
  &= f_{j_1;w(1)}\cdots f_{j_{k-1}; w(k-1)}\Bigl(f_{j_{k};l'}f_{j_{k+1};k'}\Bigr)
  f_{j_{k+2};w(k+2)}\cdots f_{j_m; w(m)}   \\
  &= f_{j_1;w(1)}\cdots f_{j_{k-1}; w(k-1)}\Bigl(q^{-(\a_{j_k}, \a_{j_{k+1}})}
      f_{j_{k+1};k'}f_{j_{k};l'}\Bigr)f_{j_{k+2};w(k+2)}\cdots f_{j_m; w(m)}  \\
  &= q^{-(\a_{j_{k}}, \a_{j_{k+1}})}
       f_{j_1;w'(1)}\cdots f_{j_{k-1}; w'(k-1)}\Bigl(
      f_{j_{k+1};w'(k)}f_{j_{k};w'(k+1)}\Bigr)f_{j_{k+2};w'(k+2)}\cdots f_{j_m; w'(m)}.  \\
\end{align*}  
In the last formula, by applying the induction hypothesis for $w'$, together with 
(5.21.4), we obtain (5.21.3) for $w$.  Hence (5.21.3) holds. 
\par
It follows from (5.21.3) that the coefficient of 
$f_{j'_1}\otimes \cdots\otimes f_{j'_m}$ in $r^{(m-1)}(f_{j_1}\cdots f_{j_m})$
is given by
\begin{equation*}
  \sum_{\substack{w \in S_m \\ (j_{w\iv (1)}, \dots, j_{w\iv(m)}) = (j_1', \dots, j'_m)}}
           q^{-A(w)}. 
\end{equation*}  
This coincides with the expression $\sum_{\Bxi \in \Xi(\ul\nu, \ul\nu')}q^{-A(\Bxi)}$
in Theorem 5.20. This implies Theorem 5.20, by using (1.3.2).

\par\bigskip
\section { Comparison of algorithms }

\para{6.1.}
We follow the setup in Section 4, and consider the monomial
basis $\SM_{\ul\Bh} = \{ m(\ul\Bc, \ul\Bh) \mid \ul\Bc \in \NN^{\ul N}\}$
of $\ul\BU_q^-$.  
We recall the algorithm of computing canonical bases discussed in Section 1.
Let $\ul H = (h_{\ul\Bc, \ul\Bc'})$ be the transition matrix from $\SX_{\ul\Bh}$ to 
$\SM_{\ul\Bh}$, $\ul Q = (q_{\ul\Bc, \ul\Bc'})$ the transition
matrix from $\bB_{\ul\Bh}$ to $\SM_{\ul\Bh}$, $\ul P = (p_{\ul\Bc,\ul\Bc'})$
the transition matrix from $\SX_{\ul\Bh}$ to $\bB_{\ul\Bh}$, as in 1.9.
Also set $\ul\vL = (m(\ul\Bc, \ul\Bh), m(\ul\Bc', \ul\Bh))$. Then we have a relation
\begin{equation*}
\tag{6.1.1}  
\ul\vL = {}^t\ul H \ul D\, \ul H, \qquad \ul H = \ul P \ul Q , 
\end{equation*}
and by Proposition 1.10, the matrix $\ul\vL$ determines the matrices $\ul P$ and $\ul Q$
uniquely.
\par
Let $F((\ul\Bc)) = m(\ul\Bc, \ul\Bh)$ be the monomial defined in (4.19.1),
and write it as $f_{j_1}^{(\ul d_1)}\cdots f_{j_m}^{(\ul d_m)}$ for $m = |J|\ul N$,
with $\Bj = (j_1, \dots, j_m) \in J^m$, $\ul\Bd = (\ul d_1, \dots, \ul d_m) \in \NN^m$.
Let $P_{\Bj, \ul\Bd}$ be an element in $\SP_{\ul\a}$ with $\ul\a \in \ul Q_+$
defined as in (5.16.1).
Then the isomorphism $\g$ in (5.13.1) maps $m(\ul\Bc, \ul\Bh)$ to $[P_{\Bj,\ul\Bd}]$,
and maps $m(\ul\Bc', \ul\Bh)$ to $[P_{\Bj', \ul\Bd'}]$. 
Hence by Theorem 5.12,
$(m(\ul\Bc, \ul\Bh), m(\ul\Bc', \ul\Bh))$ coincides with 
$(P_{\Bj,\ul\Bd}, P_{\Bj', \ul\Bd'})$, and this value is explicitly computable 
by the formula in Theorem 5.20.  Thus we have proved the following result,
which is a generalization of Antor \cite{A} for the non-symmetric case.
%%%%
\begin{thm}   %%%%  Theorem 6.2
Let $\ul\BU_q^-$ be a quantum group of general type.
Then (6.1.1) gives an algorithm of computing the transition matrix
$\ul P$ from the PBW basis to the canonical basis of $\ul\BU_q^-$. 
\end{thm}  

\para{6.3.}
Theorem 6.2 can be applied to the case where $\s$ is trivial,
which gives an algorithm of computing the transition matrix $P$
from the PBW basis to the canonical basis of $\BU_q^-$.
In the discussion below, we compare those two algorithms.
\par
Let $\wt\SM_{\Bh} = \{ \wt m(\Bc, \Bh) \mid \Bc \in \NN^N\}$
be the (modified) monomial basis of $\BU_q^-$ given in 4.22.
Let $H = (h_{\Bc,\Bc'})$ be the transition matrix from $\SX_{\Bh}$ to
$\wt\SM_{\Bh}$, $Q = (q_{\Bc,\Bc'})$ the transition matrix from  $\bB_{\Bh}$
to $\wt\SM_{\Bh}$, $P = (p_{\Bc,\Bc'})$ the transition matrix from  $\SX_{\Bh}$
to $\bB_{\Bh}$. Also set $\vL = \Bigl((\wt m(\Bc,\Bh), \wt m(\Bc',\Bh))\Bigr)$.  We have
a similar matrix equation
\begin{equation*}
\tag{6.3.1}  
\vL = {}^tH D H, \qquad H = PQ,
\end{equation*}
and the matrix $\vL$ determines $Q$ and $P$ uniquely.
The matrix $\vL$ can be computed by a similar formula as in Theorem 5.20
applied for the case $\s$ is trivial. Thus (6.3.1) gives an algorithm
of computing the transition matrix $P$ from  $\SX_{\Bh}$ to $\bB_{\Bh}$. 
\par
We write $\vL$ as $\vL = (\la_{\Bc, \Bc'})_{\Bc, \Bc'\in \NN^N}$, where
$\la_{\Bc, \Bc'} = (\wt m(\Bc, \Bh), \wt m(\Bc', \Bh))$.
Let $\vL^{\s}$ be the submatrix of $\vL$ consisting of $\la_{\Bc, \Bc'}$
such that $\Bc, \Bc' \in \NN^{N,\s}$. 
We want to compare the matrices $\vL^{\s}$ and $\ul\vL$, under the bijection
$\NN^{N,\s}$ and $\NN^{\ul N}$.
\par
Recall that $\wt m(\Bc, \Bh) = F(\wt\Bd^1) \cdots F(\wt\Bd^{\ul N})$
(see (4.22.1) and (4.22.2)), which we express as 
$\wt m(\Bc, \Bh) = f_{i_1}^{(d_1)} \cdots f_{i_{m'}}^{(d_{m'})}$,
where $m' = |I|\ul N$, and $\Bi = (i_1, \dots, i_{m'}) \in I^{m'}$,
$\Bd = (d_1, \dots, d_{m'}) \in \NN^{m'}$.

\para{6.4.}
Assume that $\Bc \in \NN^{N, \s}$ corresponds to $\ul \Bc \in \NN^{\ul N}$. 
Then $\Bi$ and $\Bd$ are determined from $\Bj$ and $\ul\Bd$ as follows. 
For each $j \in J$, we fix
an order of the orbit $j$ as $j = \{ k_1, \dots, k_{|j|} \}$.  Then $\Bi$ is given 
as a juxtaposition of $j_1, \dots, j_m$ such as
\begin{equation*}
  \Bi = (\underbrace{i_1, \dots, i_{|j_1|}}_{j_1},
         \underbrace{i_{|j_1|+1}, \dots, i_{|j_1|+ |j_2|}}_{j_2}, \dots,
         \underbrace{i_{|j_1| + \cdots + |j_{m-1}| +1}, \dots, i_{m'}}_{j_m}),
\end{equation*}
and $\Bd$ is given as 
\begin{equation*}
  \Bd = (d_1, \dots, d_{m'})
  = (\underbrace{\ul d_1, \dots, \ul d_1}_{|j_1|\text{-times}},
  \underbrace{\ul d_2, \dots, \ul d_2}_{|j_2|\text{-times}},
  \dots, \underbrace{\ul d_m, \dots, \ul d_m}_{|j_m|\text{-times}}). 
\end{equation*}  

\par
Let $P_{\Bi, \Bd}$ be an element in $\SP_{\a}$ with $\a \in Q_+$
corresponding to $\wt m(\Bc, \Bh)$,
and define $P_{\Bi', \Bd'}$ corresponding to $\wt m(\Bc',\Bh)$, similarly.
Then we have $([P_{\Bi, \Bd}], [P_{\Bi', \Bd'}]) = (\wt m(\Bc, \Bh), \wt m(\Bc', \Bh))$. 
\par
Assume that $\Bc \in \NN^{N,\s}$ (resp. $\Bc' \in \NN^{N,\s}$), and
$(\Bi, \Bd)$ (resp. $(\Bi',\Bd')$) is obtained from $(\Bj, \ul\Bd)$
(resp. from $(\Bj', \ul\Bd')$). 
The inner product $([P_{\Bj,\ul\Bd}], [P_{\Bj', \ul\Bd'}])$
can be computed by making use of Theorem 5.20. 
Assume that $\weit(P_{\Bj,\ul\Bd}) = \weit(P_{\Bj', \ul\Bd'}) = \ul\a \in \ul Q_+$.  
Set $n = \sum_{k=1}^m\ul d_k$, and let $\ul\nu = (\ul\nu_1, \dots, \ul\nu_n) \in J^n$
be the element attached to $(\Bj, \ul\Bd)$ as in (5.16.2). Similarly,
we define $\ul\nu' = (\ul\nu'_1, \dots, \ul\nu'_n)$ attached to $(\Bj',\ul\Bd')$.
Then by Theorem 5.20, we have

\begin{align*}
\tag{6.4.1}  
([P_{\Bj,\ul\Bd}], [P_{\Bj', \ul\Bd'}])
  = \d_{\ul\a}\iv \prod_{k=1}^m ([\ul d_k]_{j_k}^! [\ul d'_k]_{j'_k}^!)\iv \cdot
         \sum_{\ul\Bxi \in \Xi(\ul\nu, \ul\nu')}q^{-A(\ul\Bxi)}, 
\end{align*}
where by (5.19.1), 
\begin{equation*}
\d_{\ul\a} = \prod_{k=1}^m(1-q_{j_k}^2)^{\ul d_k}.
\end{equation*}  

On the other hand, we consider the inner product $([P_{\Bi,\Bd}], [P_{\Bi',\Bd'}])$.
Assume that $\weit(P_{\Bi,\Bd}) = \weit(P_{\Bi',\Bd'}) = \a \in Q^{\s}_+$.
Set $n' = \sum_{k=1}^{m'}d_k$, and let $\nu = (\nu_1, \dots, \nu_{n'}) \in I^{n'}$
attached to $(\Bi, \Bd)$ as in (5.16.2). $\nu' = (\nu'_1, \dots, \nu'_{n'})$ attached
to $(\Bi',\Bd')$ is defined similarly. 
By applying Theorem 5.20 for $P_{\Bi, \Bd}$, we have
the following.  
\begin{align*}
\tag{6.4.2}  
([P_{\Bi,\Bd}], [P_{\Bi', \Bd'}])
&= \d_{\a}\iv \prod_{k=1}^{m '}([d_k]^! [d'_k]^!)\iv \cdot
             \sum_{\Bxi \in \Xi(\nu, \nu')}q^{-A(\Bxi)}  \\ 
&= \d_{\a}\iv \prod_{k=1}^m ([\ul d_k]^!)^{-|j_k|} ([\ul d'_k]^!)^{-|j'_k|} \cdot
        \sum_{\Bxi \in \Xi(\nu, \nu')}q^{-A(\Bxi)}, 
\end{align*}
where
\begin{equation*}
\d_{\a} = \prod_{k = 1}^{m'}(1 - q^2)^{d_k} = \prod_{k=1}^m(1-q^2)^{\ul d_k|j_k|}.
\end{equation*}  

\par
Assume that $\wt m(\Bc, \Bh)$ corresponds to $P_{\Bi,\Bd}$ for $\Bc \in \NN^{N,\s}$,
and $m(\ul\Bc, \ul\Bh)$ corresponds to $P_{\Bj, \ul\Bd}$ for $\ul\Bc \in \NN^{\ul N}$.
Set
\begin{equation*}
\tag{6.4.3}  
  \g_{\Bc} = \prod_{k=1}^m([\ul d_k]^!)^{|j_k|}, \qquad
  \g_{\ul\Bc} = \prod_{k=1}^m[\ul d_k]_{j_k}^!. 
\end{equation*}  
\para{6.5.}
For each $j \in J$, we fix a total order
$j = \{ k_1, \dots, k_{|j|}\}$. 
For each $\ul\nu = (\ul\nu_1, \dots, \ul\nu_n) \in J^n$, we define
$\nu = (\nu_1, \dots, \nu_{n'}) \in I^{n'}$ by juxtaposition of
$j = \ul\nu_k \lra (k_1, \dots, k_{|j|}) \in I^{|j|}$, where
$n' = \sum_{1 \le k \le n}|\ul\nu_k|$.
For each $\ul\Bxi = (\ul\xi_{st})_{1 \le s,t \le n} \in \Xi(\ul\nu, \ul\nu')$, we define
a block matrix $\Bxi = (\wt\Bxi_{st})_{1 \le s,t \le n}$, where
$\wt\Bxi_{st}$ is a matrix of size $|\ul\nu_s| \times |\ul\nu'_t|$ such that  
\begin{equation*}
  \wt\Bxi_{st} = \begin{cases}
          \Diag(\a_{k_1}, \dots, \a_{k_{|j|}})
                 &\quad\text{ if } \ul\xi_{st}
                = \ul\nu_s = \ul\nu'_t = j,  \\  
            0    &\quad\text{ if } \ul\xi_{st} = 0.
               \end{cases} 
\end{equation*}
Then $\Bxi$ gives an element in $\Xi(\nu, \nu')$, and
one can define an injective map $\vf : \Xi(\ul\nu, \ul\nu') \to \Xi(\nu,\nu')$
by the assignment $\ul\Bxi \mapsto \Bxi$.
We consider $w = w(\ul\Bxi) \in S_n$ and $w' = w(\Bxi) \in S_{n'}$ associated to $w$.
Then this $w'$ coincides with $w'$ appeared in (5.17.5).
By applying Lemma 5.18 for $e(\nu)$ and $e(\ul\nu)$, we see that
$A(\Bxi) = A(\ul\Bxi)$.
Summing up the above discussion, we have the following.

\begin{lem}  %%%  Lemma 6.6
There exists an injective map
$\vf : \Xi(\ul\nu, \ul\nu') \to \Xi(\nu, \nu'), \ul\Bxi \mapsto \Bxi$
such that $A(\ul\Bxi) = A(\Bxi)$.
In particular, we have
\begin{equation*}
\tag{6.6.1}  
  \sum_{\ul\Bxi \in \Xi(\ul\nu, \ul\nu')}q^{-A(\ul\Bxi)}
          = \sum_{\Bxi \in \Im \vf}q^{-A(\Bxi)}.
\end{equation*}  
\end{lem}

By comparing (6.4.1) and (6.4.2), and by applying Lemma 6.6,
we have the following result.  
%%%%
%%%%
\begin{thm}  %%%%   Theorem 6.7
The inner product $([P_{\Bj,\ul\Bd}], [P_{\Bj',\ul\Bd'}])$ can be computed from
the inner product $([P_{\Bi,\Bd}], [P_{\Bi', \Bd'}])$ in (6.4.2), by replacing
$\d_{\a}$ by $\d_{\ul\a}$, $([\ul d_k]^!)^{|j_k|}([\ul d'_k]^!)^{|j'_k|}$ by
$[\ul d_k]_{j_k}^![\ul d'_k]_{j'_k}^!$, and
$\sum_{\Bxi \in \Xi(\nu,\nu')}q^{-A(\Bxi)}$ by $\sum_{\Bxi \in \Im \vf}q^{-A(\Bxi)}$. 
Hence the matrix $\ul\vL$ can be computed from the data related to the matrix $\vL$.
In particular, the matrix $\ul P$ can be computed from the monomial basis for $\BU_q^-$. 
\end{thm}

\remark{6.8.}
In order to apply Lemma 6.6 for the proof of Theorem 6.7, we need to
replace $P_{\Bj, \ul\Bd}$ by the sequence $\ul\nu$.  If $\nu$ is the sequence
as given in 6.5, it not necessarily coincide with the sequence obtained from $P_{\Bi,\Bd}$.
However, the difference of those two sequences consists of $i$ which are not joined.
Hence the corresponding monomial bases are the same, and one can ignore
the discrepancy.  For example, let $P_{\Bj, \ul\Bd}$ be such that
$\Bj = (j_1), \ul\Bd = (\ul d_1)$ with $\ul d_1 = 2$.  Assume that
$j_1 = \{ i_1, i_2\} \in J$.  Then the corresponding $\s$-stable $P_{\Bi, \Bd}$
is given by $\Bi = (i_1, i_2), \Bd = (2,2)$.
In this case, $\ul\nu = (j_1, j_1)$ and the sequence $\nu$ obtained from $\ul\nu$
as in 6.5 is $\nu = (i_1, i_2, i_1, i_2)$. 
But the sequence obtained from $P_{\Bi,\Bd}$ is $\nu' = (i_1, i_1, i_2, i_2)$,
and we have $\nu \ne \nu'$.

\remarks{6.9.}
(i) \
Let $P^{\s}$ be the submatrix of $P$ consisting of $p_{\Bc, \Bc'}$ with
$\Bc, \Bc'$ : $\s$-stable.  Then Proposition 3.7 asserts that $P^{\s}$ coincides with $\ul P$
over $\BA' = (\ZZ/p\ZZ)[q,q\iv]$.
In the setup of Theorem 6.7, this is explained as follows.
We have $\d_{\a} = \d_{\ul \a}$ and $([\ul d_k]^!)^{|j_k|} = [\ul d_k]_{j_k}^!$ on
$\BA'$ (note that $|j| = 1$ or $p$ for $j \in J$). Furthermore
it is likely that $\sum_{\Bxi \in \Xi(\nu,\nu')}q^{-A(\Bxi)}
   \equiv \sum_{\Bxi \in \Im \vf}q^{-A(\Bxi)} \pmod p$.
Then we have $\vL^{\s} = \ul\vL$ over $\BA'$, and $P^{\s} = \ul P$ over $\BA'$.
\par
(ii) \ Theorem 6.7 together with Theorem 6.2 provides an algorithm
of computing $\ul P$ over $\BA$, 
but in order to determine $\ul P$, the matrix $\vL$ or $H$ is needed besides $P$.

\par\bigskip
\section{ Examples }

\para{7.1.}
Assume that $\BU_q^-$ is of type $A_3$, with $I = \{ 1,2,1'\}$ such that
$\s : 1 \lra 1', 2\lra 2$.
We assume that $(\a_i, \a_i) = 2$ for $i \in I$, and $(\a_i, \a_j) = -1$ if $i$ and $j$
are joined.  
We follow the setup in 4.3 (A).
We have $J = \{ \ul 1, \ul 2\}$ and $\ul\BU_q^-$
is of type $B_2$.  Thus $(\a_{\ul 1}, \a_{\ul 1}) = 4, (\a_{\ul 2}, \a_{\ul 2}) = 2$,
and $(\a_{\ul 1}, \a_{\ul 2}) = -2$.  
We take a reduced sequence $\ul\Bh = (\ul 1, \ul 2, \ul 1, \ul 2)$ of
$\ul w_0$ of $W(B_2)$. Let $\Bh = (1,1',2,1,1',2)$ be a reduced sequence of
$w_0 \in W(A_3) = S_4$. Then $\Bh$ is adapted
with respect to the quiver $\ora{Q} = 1 \leftarrow 2 \rightarrow 1'$,
and the total order of $I = \{ 1,1',2\}$ satisfies the condition (*) (see 4.3 (A)).
The total order of $\vD^+$ associated to $\Bh$
is given by $\vD^+ = \{ 1,1', 11'2, 1'2, 12, 2\}$.
The PBW basis $\SX_{\Bh}$
is given by
\begin{equation*}
\SX_{\Bh} = \{ L(\Bc,\Bh) =
  F_1^{(c_1)}F_{1'}^{(c_2)}F_{11'2}^{(c_3)}F_{1'2}^{(c_4)}F_{12}^{(c_5)}F_2^{(c_6)}
       \mid \Bc = (c_1, \dots, c_6) \in \NN^6 \}.
\end{equation*}
$\s$ acts on $\vD^+$ by $\s: 1 \lra 1', 12 \lra 1'2$, and leaves $11'2, 2$ stable.
Correspondingly, we have an action of $\s$ on $\Bc = (c_1, \dots, c_6)$ by
\begin{equation*}
\s : (c_1,c_2,c_3,c_4,c_5,c_6) \mapsto (c_2, c_1, c_3, c_5,c_4, c_6).
\end{equation*}
The total order of $\ul\vD^+$ associated to $\ul\Bh$ is given by
$\ul\vD^+ = \{ \ul 1, \ul 1\ul 2, \ul 1\ul 2\ul 2, \ul 2\}$, and 
the PBW basis $\SX_{\ul\Bh}$ is given by
\begin{equation*}
  \SX_{\ul\Bh} = \{ L(\ul\Bc, \ul\Bh)
  = F_{\ul 1}^{(\ul c_1)}F_{\ul 1\ul 2}^{(\ul c_2)}
  F_{\ul 1\ul 2\ul 2}^{(\ul c_3)}F_{\ul 2}^{(\ul c_4)} \mid
     \ul\Bc = (\ul c_1, \dots, \ul \Bc_4) \in \NN^4\}.
\end{equation*}
Let $\wt m(\Bc, \Bh) = F(\wt\Bd^1)F(\wt\Bd^2)F(\wt\Bd^3)F(\wt\Bd^4)$ be the monomial basis.
We have
\begin{equation*}
\begin{aligned}
 \wt\Bc_1 &= (c_1, c_2,0,0,0,0) &\quad &\longleftrightarrow &\quad \wt\Bd^1 &= (c_1, c_2, 0), \\ 
  \wt\Bc_2 &= (0,0, c_3, 0,0,0)  &\quad &\longleftrightarrow &\quad \wt\Bd^2 &= (c_3, c_3, c_3),  \\
  \wt\Bc_3 &= (0, 0,0,c_4, c_5,0) &\quad &\longleftrightarrow
                  &\quad \wt\Bd^3 &= (c_4, c_5, c_4+c_5),  \\
  \wt\Bc_4 &= (0,0,0,0, 0,c_6)   &\quad  &\longleftrightarrow  &\quad \wt\Bd^4 &= (0,0,c_6).
\end{aligned}
\end{equation*}

Hence $\wt m(\Bc,\Bh)$ is given by
\begin{equation*}
\tag{7.1.1}  
  \wt m(\Bc, \Bh) = f_{1'}^{(c_2)}f_1^{(c_1)}\cdot f_2^{(c_3)}f_{1'}^{(c_3)}f_1^{(c_3)}\cdot
                      f_{2}^{(c_4 + c_5)}f_{1'}^{(c_5)}f_1^{(c_4)}\cdot f_2^{(c_6)}.
\end{equation*}  

In turn, for $\ul\Bc = (\ul c_1, \ul c_2, \ul c_3, \ul c_4)$, we have

\begin{equation*}
\begin{aligned}
\ul\Bc_1 &= (\ul c_1, 0, 0, 0) &\quad &\longleftrightarrow &\quad \ul\Bd^1 &= (\ul c_1, 0), \\     
\ul\Bc_2 &= (0, \ul c_2, 0,0)  &\quad &\longleftrightarrow &\quad \ul\Bd^2 &= (\ul c_2, \ul c_2), \\ 
\ul\Bc_3 &= (0,0, \ul c_3,0)   &\quad &\longleftrightarrow &\quad \ul\Bd^3 &= (\ul c_3, 2\ul c_3), \\
\ul\Bc_4 &= (0,0,0,\ul c_4)    &\quad &\longleftrightarrow &\quad \ul\Bd^4 &= (0, \ul c_4).
\end{aligned}
\end{equation*}  

Hence $m(\ul\Bc, \ul\Bh) = F(\ul \Bd^1)F(\ul\Bd^2)F(\ul\Bd^3)F(\ul\Bd^4)$
is given by

\begin{equation*}
\tag{7.1.2}  
m(\ul\Bc, \ul\Bh) = f_{\ul 1}^{(\ul c_1)}\cdot f_{\ul 2}^{(\ul c_2)}f_{\ul 1}^{(\ul c_2)}\cdot
                       f_{\ul 2}^{(2\ul c_3)}f_{\ul 1}^{(\ul c_3)}\cdot f_{\ul 2}^{(\ul c_4)}.
\end{equation*}

Thus in the case where
$\Bc = (c_1, \dots, c_6) = (\ul c_1, \ul c_1, \ul c_2, \ul c_3, \ul c_3, \ul c_4) \in \NN^{6,\s}$, 
we have

\begin{equation*}
\tag{7.1.3}  
\Phi(m(\ul\Bc, \ul\Bh)) = \pi(\wt m(\Bc, \Bh)). 
\end{equation*}

\para{7.2.}
We keep the setting in 7.1.
For $L(\Bc,\Bh)$, the weight is given by
$\a = \weit(L(\Bc,\Bh)) = \sum_{k=1}^6 c_k\b_k$.
We consider the case where $\a = 2\a_1 + 2\a_{1'} + \a_2$.
Then $\Bc = (c_1, \dots, c_6)$ such that $\weit(L(\Bc,\Bh)) = \a$ are given by
\begin{align*}
  \Bc_1 &= (1,1,1,0,0,0), \\
  \Bc_2 &= (1,2,0,0,1,0), \\
  \Bc_3 &= (2,1,0,1,0,0), \\
  \Bc_4 &= (2,2,0,0,0,1).
\end{align*}
By (7.1.1), the corresponding $\wt m(\Bc, \Bh)$ are given by
\begin{align*}
\wt m(\Bc_1, \Bh) &= f_{1'}f_1\cdot f_2f_{1'}f_1, \\
\wt m(\Bc_2, \Bh) &= f_{1'}^{(2)}f_1\cdot f_2f_1, \\
\wt m(\Bc_3, \Bh) &= f_{1'}f_1^{(2)}\cdot f_2f_{1'}, \\
\wt m(\Bc_4,\Bh)  &= f_{1'}^{(2)}f_1^{(2)}\cdot f_2. 
\end{align*}
Here $\wt m(\Bc_1, \Bh), \wt m(\Bc_4, \Bh)$ are $\s$-invariant, and
$\s$ permutes $\wt m(\Bc_2, \Bh)$ with $\wt m(\Bc_3, \Bh)$.
\par
Let $\Bi = (1', 1, 2, 1',1) \in I^5$ and $\Ba = (a_1, a_2, a_3, a_4, a_5) \in \NN^5$.
Then $\wt m(\Bc,\Bh)$ can be written as
$F_{\Bi, \Ba} = f_{1'}^{(a_1)}f_1^{(a_2)}f_2^{(a_3)}f_{1'}^{(a_4)}f_1^{(a_5)}$, 
\begin{align*}
 \wt m(\Bc_1, \Bh) &= F_{\Bi, \Ba_1} \quad\text{ with }\quad \Ba_1 = (1,1,1,1,1), \\     
 \wt m(\Bc_2, \Bh) &= F_{\Bi, \Ba_2} \quad\text{ with }\quad \Ba_2 = (2,1,1,0,1), \\
 \wt m(\Bc_3, \Bh) &= F_{\Bi, \Ba_3} \quad\text{ with }\quad \Ba_3 = (1,2,1,1,0), \\
 \wt m(\Bc_4, \Bh) &= F_{\Bi, \Ba_4} \quad\text{ with }\quad \Ba_4 = (2,2,1,0,0). 
\end{align*}  

We consider the matrix $\vL = (F_{\Bi,\Ba_k}, F_{\Bi, \Ba_{k'}})_{1 \le k,k' \le 4}$. 
Since $\a = 2\a_1 + 2\a_{1'} + \a_2$, $\d_{\a} = (1-q^2)^5$. Furthermore
\begin{align*}
\g_{\Bi,\Ba_1} &= 1, \\ 
\g_{\Bi,\Ba_2} &= [2]^!      = q + q\iv,  \\
\g_{\Bi,\Ba_3} &= [2]^!      = q + q\iv,  \\
\g_{\Bi,\Ba_4} &= [2]^![2]^! = (q + q\iv)^2. 
\end{align*}

Then $\vL$ can be computed by making use of the formula (6.4.2), 
by replacing $F_{\Bi, \Ba_k}, F_{\Bi, \Ba_{k'}}$ by
$[P_{\Bi, \Ba_k}], [P_{\Bi,\Ba_{k'}}]$.  We have

{\footnotesize 
\begin{align*}
\tag{7.2.1}  
\vL &= \d_{\a}\iv \vG_{\Bi,\Ba}\iv
            \begin{pmatrix}    
               4  &   2(q+q\iv)  &  2(q+q\iv)  &   (q + q\iv)^2   \\ 
               2(q+q\iv)    &   2q\iv (q + q\iv)  &  (q+q\iv)^2  &  q\iv (q + q\iv)^2  \\
               2(q+q\iv)    &   (q + q\iv)^2   &   2q\iv(q + q\iv)  &   q\iv (q + q\iv)^2  \\
               (q + q\iv)^2  &  q\iv(q + q\iv)^2  &  q\iv(q + q\iv)^2  &  q^{-2}(q + q\iv)^2
             \end{pmatrix}\vG_{\Bi, \Ba}\iv,  \\   \\
    &=   \frac{1}{(1-q^2)^5(1+q^2)^2}
          \begin{pmatrix}
            4(1+q^2)^2   &  2(1+q^2)^2  &   2(1+q^2)^2   &  (1 + q^2)^2  \\
            2(1+q^2)^2   &  2(1+q^2)  &  (1 + q^2)^2   &   1 + q^2  \\
            2(1+q^2)^2   &  (1 + q^2)^2    &  2(1+q^2)  &   1+q^2   \\
           (1+q^2)^2    &  1 + q^2   &  1 + q^2    &   1
          \end{pmatrix}, 
\end{align*}
}
where $\vG_{\Bi,\Ba} = \Diag(\g_{\Bi, \Ba_1}, \g_{\Bi,\Ba_2}, \g_{\Bi,\Ba_3}, \g_{\Bi, \Ba_4})$.  
\par\medskip
As an example, we compute
$(F_{\Bi,\Ba_3}, F_{\Bi, \Ba_4}) =
     (f_{1'}f_1^{(2)}f_2f_{1'}, f_{1'}^{(2)}f_1^{(2)}f_2)$.
In this case, $\nu = (1',1,1,2,1')$ and $\nu' = (1',1',1,1,2)$.
$\Xi(\nu,\nu') = \{ \Bxi_1, \Bxi_2, \Bxi_3, \Bxi_4\}$, where
{\footnotesize 
\begin{align*}
\Bxi_1 &= \begin{pmatrix}
             \a_{1'} &  &  &  &   \\
                     &  &  \a_1 &   &  \\
                     &  &       &  \a_1  &  \\
                     &  &       &        &  \a_2  \\
                     &  \a_{1'}  &   &   &  
           \end{pmatrix},  \qquad
\Bxi_2 = \begin{pmatrix}
                  &  \a_{1'} &  &  &   \\
                  &  &  \a_1 &   &  \\
                  &  &       &  \a_1  &  \\
                  &  &       &        &  \a_2  \\
        \a_{1'}   &  &   &   &  
           \end{pmatrix},  \\
\Bxi_3 &= \begin{pmatrix}
         \a_{1'}  &  &  &  &   \\
                  &  &  &  \a_1  &  \\
                  &  &  \a_1 &   &  \\
                  &  &       &   &  \a_2  \\
                  &  \a_{1'} &   &   &  
           \end{pmatrix},  \qquad 
\Bxi_4 = \begin{pmatrix}
                  &  \a_{1'} &  &  &   \\
                  &  &  &  \a_1  &  \\
                  &  &  \a_1 &   &  \\
                  &  &       &   &  \a_2  \\
           \a_{1'}&  &   &   &  
           \end{pmatrix}. 
\end{align*}
}

The corresponding $w_i = w(\Bxi_i) \in S_5$ are given by
\begin{equation*}
  w_1 = (1,3,4,5,2), \quad w_2 = (2,3,4,5,1), \quad
  w_3 = (1,4,3,5,2), \quad w_4 = (2,4,3,5,1).
\end{equation*}  
Since $A(\Bxi) = \sum_{k<\ell; w(k) > w(\ell)}(\a_{\nu_k}, \a_{\nu_{\ell}})$
for $w = w(\Bxi)$, we have
\begin{align*}
  A(\Bxi_1) = -1, \quad A(\Bxi_2) = 1, \quad A(\Bxi_3) = 1, \quad A(\Bxi_4) = 3.
\end{align*}  
Thus $\sum_{\Bxi \in \Xi(\nu,\nu')}q^{-A(\Bxi)} = q\iv(q + q\iv)^2$. 
It follows that
\begin{align*}
\tag{7.2.2}
(F_{\Bi,\Ba_3}, F_{\Bi, \Ba_4})
      &= \d_{\a}\iv \g_{\Bi, \Ba_3}\iv \g_{\Bi,\Ba_4}\iv \times q\iv(q + q\iv)^2
       = \frac{1 + q^2}{(1 - q^2)^5(1 + q^2)^2}.
\end{align*}

All other inner products $(F_{\Bi, \Ba_k}, F_{\Bi, \Ba_{k'}})$ can be computed
similarly. 

\par\medskip
We have a matrix equation $\vL = {}^tHDH$, where $H$ is lower unitriangular,
and $D$ is diagonal. 
By using (7.2.1), $H$ and $D$ are determined from this equation as follows.
{\footnotesize
\begin{align*}
\tag{7.2.3}
  H &= \begin{pmatrix}
        1          &               &         &    \\
        1+q^2      &   1           &         &    \\
        1+q^2      &   0           &   1     &    \\
        (1+q^2)^2  &   1 + q^2     &  1+q^2  &   1
      \end{pmatrix},  \\
  D &=  \frac{1}{(1-q^2)^5(1 + q^2)^2}
           \begin{pmatrix}
             (1-q^4)^2   &    &   &   \\
                         &   1-q^4  &    &  \\
                         &      &   1-q^4  &    \\
                         &      &     &   1
           \end{pmatrix}.
\end{align*}
}

Note that $D = (L(\Bc_k, \Bh), L(\Bc_{k'}, \Bh))$ can be computed directly
by using the formula (1.4.1), which coincides with our computation.
\par
Now $H$ is decomposed as $H = PQ$, where $P,Q$ are lower unitriangular,
the coordinates of $Q$ are bar-invariant, and off-diagonals of $P$ are
contained in $q\ZZ[q]$.  Hence we have
{\footnotesize
\begin{align*}
\tag{7.2.4}  
 P = \begin{pmatrix}
       1   &        &      &   &    \\
      q^2  &  1     &      &   &    \\
      q^2  &  0     &  1    &   &    \\
       q^4  &  q^2  &  q^2  &  1
     \end{pmatrix}, \qquad
 Q = \begin{pmatrix}
      1  &     &     &  &  \\
      1  &  1  &     &  &  \\
      1  &  0  &  1  &   &   \\
      1  &  1  &  1  &  1  &
     \end{pmatrix}. 
\end{align*}  
}

Note that these matrices are indexed by $\Ba_1, \dots, \Ba_4$.
$\s$ leaves $\Ba_1, \Ba_4$ stable, and permutes $\Ba_2, \Ba_3$. 
It follows that
{\footnotesize
\begin{align*}
\tag{7.2.5}
  \vL^{\s} &= \d_{\a}\iv \begin{pmatrix}
              \g_{\Bi, \Ba_1}\iv   &   \\
                          &  \g_{\Bi, \Ba_4}\iv
            \end{pmatrix}
            \begin{pmatrix}
             4  &   (q + q\iv)^2  \\
             (q + q\iv)^2        &   (1 + q^{-2})^2 
            \end{pmatrix}
            \begin{pmatrix}
              \g_{\Bi,\Ba_1}\iv    &    \\
                           &  \g_{\Bi,\Ba_4}\iv
            \end{pmatrix}  \\
         &= \frac{1}{(1-q^2)^3(1 - q^4)^2}
              \begin{pmatrix}
               4(1+q^2)^2 &    (1 + q^2)^2   \\
              (1+q^2)^2   &     1 
              \end{pmatrix}.     
\end{align*}
}

\para{7.3.}
We keep the setup in 7.1, and consider $\ul\BU_q^-$ of type $B_2$ with $J = \{ \ul 1, \ul 2\}$,
where $\a_{\ul 1}$ is a long root
and $\ul\a_2$ is a short root. 

Let $\Bj = (\ul 1, \ul 2, \ul 1)$, and $\ul \a = 2\a_{\ul 1} + \a_{\ul 2}$.
For $\Bj$ and $\Ba = (a_1, a_2, a_3)$, set
$F_{\Bj,\Ba} = f_{\ul 1}^{(a_1)}f_{\ul 2}^{(a_2)}f_{\ul 1}^{(a_3)}$. 
Consider the monomial basis
\begin{align*}
  m(\ul\Bc_1, \ul\Bh) &= F_{\Bj, \ul\Ba_1} = f_{\ul 1}f_{\ul 2}f_{\ul 1}, \\
  m(\ul\Bc_2, \ul\Bh) &= F_{\Bj, \ul\Ba_2} = f_{\ul 1}^{(2)}f_{\ul 2}, 
\end{align*}
where $\ul\Ba_1 = (1,1,1), \ul\Ba_2 = (2,1,0) \in \NN^3$,
and $\ul\Bc_1 = (1,1,0,0), \ul\Bc_2 = (2,0,0,1)$.    
We have
\begin{equation*}
  \d_{\ul\a} = (1-q^4)^2(1-q^2), \quad
      \g_{\Bj, \ul\Ba_1} = 1, \quad \g_{\Bj,\ul\Ba_2} = q^2 + q^{-2}. 
\end{equation*}
\par
The inner products $(F_{\Bj, \ul\Ba_k}, F_{\Bj, \ul\Ba_{k'}})$ are computed similarly,
and we have

\begin{align*}
\tag{7.3.1}
\ul\vL = \d_{\ul\a}\iv \begin{pmatrix}
                         \g_{\Bj,\ul\Ba_1}\iv   &    \\
                             &    \g_{\Bj,\ul\Ba_2}\iv
                       \end{pmatrix}
                       \begin{pmatrix}
                            2           &    q^2 + q^{-2}   \\
                           q^2 + q^{-2}    &   1 + q^{-4}
                        \end{pmatrix}
                       \begin{pmatrix}
                             \g_{\Bj, \ul\Ba_1}\iv   &     \\
                                  &     \g_{\Bj,\ul\Ba_2}\iv
                        \end{pmatrix}.     
\end{align*}

The matrix equation $\ul\vL = {}^t\ul H \ul D \ul H$ determines $\ul H$ and $\ul D$
as follows.

{\footnotesize
\begin{align*}
\tag{7.3.2}  
\ul H = \begin{pmatrix}
           1  &   \\
           1 + q^4  &    1
        \end{pmatrix},   \qquad
\ul D = \begin{pmatrix}
            \frac{1}{(1-q^2)(1-q^4)}  &   \\
                   &   \frac{1}{(1-q^2)(1-q^4)(1-q^8)}
        \end{pmatrix}.
\end{align*}  
}

Finally we have
{\footnotesize
\begin{equation*}
\tag{7.3.3}  
\ul H = \begin{pmatrix}
          1  &     \\    
          1 + q^4  &  1 
         \end{pmatrix}
      =  \ul P \ul Q  = \begin{pmatrix}
                            1   &   \\
                            q^4  &   1
                       \end{pmatrix}     
                       \begin{pmatrix}
                          1   &    \\
                          1   &   1 
                       \end{pmatrix}.
\end{equation*}  
}

\remark{7.4.}
The matrix {\footnotesize $\begin{pmatrix}
               2  &   q^2 + q^{-2}  \\
               q^2 + q^{-2}  &  1 + q^{-4}
             \end{pmatrix}$}
appearing in $\ul\vL$ in (7.3.1) can be computed from the matrix
             {\footnotesize
             $\begin{pmatrix}
                4  &  q^2 + q^{-2} + 2  \\
                q^2 + q^{-2} + 2   &  1 + q^{-4} + 2q^{-2}
             \end{pmatrix}$}
appearing in $\vL^{\s}$ in (7.2.5) by using the injection
$\vf : \Xi(\ul\nu, \ul\nu') \to \Xi(\nu,\nu')$ in Lemma 6.6.    
Also note that these two matrices coincide each other over $(\BZ/2\BZ)[q, q\iv]$. 

\para{7.5}
Assume that $\BU_q^-$ is of type $D_4$ with $I = \{ 1,1',1'',2\}$,
such that $\s : 1 \mapsto 1' \mapsto 1'' \mapsto 1$, and $\s(2) = 2$.
We assume that $(\a_i, \a_i) = 2$ for $i \in I$, and $(\a_i,\a_j) = -1$ if
$i$ and $j$ are joined. 
We have $J = \{ \ul 1, \ul 2\}$ and $\ul\BU_q^-$ is of type $G_2$. 
Thus $(\a_{\ul 1}, \a_{\ul 1}) = 6, (\a_{\ul 2}, \a_{\ul 2}) = 2$, and
$(\a_{\ul 1}, \a_{\ul 2}) = -3$. 
We follow the notation in 4.3 (D).
In particular, $\Bh$ and $\ul \Bh$ are given by (4.3.12) and (4.3.13).
The total order of $\vD^+$ associated to $\Bh$
is given (see the proof of Lemma 4.13, (D)) as 
\begin{equation*}
\vD^+ = \{ 1, 1', 1'', 11'1''2, 1'1''2, 11''2, 11'2, 11'1''22, 12, 1'2, 1''2, 2\}.
\end{equation*}
$\s$ acts on $\vD^+$ by $1 \mapsto 1' \mapsto 1'' \mapsto 1$,
$1'1''2 \mapsto 11''2 \mapsto 11'2 \mapsto 1'1''2$, $12 \mapsto 1'2 \mapsto 1''2 \mapsto 12$,
and leaves $11'1''2, 11'1''22, 2$ stable. 
Correspondingly, we have an action of $\s$ on $\Bc = (c_1, \dots, c_{12}) \in \NN^{12}$ by 
$c_1 \mapsto c_2 \mapsto c_3 \mapsto c_1$, 
$c_5 \mapsto c_6 \mapsto c_7 \mapsto c_5$, $c_9 \mapsto c_{10} \mapsto c_{11} \mapsto c_9$,
and $\s$ leaves $c_4, c_8, c_{12}$ stable. 
\par
The total order of $\ul\vD^+$ associated to $\ul\Bh$ is given by
$\ul\vD^+ = \{ \ul 1, \ul{12}, \ul{11222}, \ul{122}, \ul{1222}, \ul 2\}$. 
\par
Let $\wt m(\Bc, \Bh) = F(\wt \Bd^1)\cdots F(\wt \Bd^6)$ be the monomial basis of $\BU_q^-$.
We have 
{\footnotesize
\begin{equation*}
\begin{aligned}
  \wt\Bc_1 &= (c_1, c_2, c_3, 0, \dots, 0)  &\quad &\longleftrightarrow
                \quad   \wt \Bd^1= (c_1, c_2, c_3, 0),  \\
  \wt\Bc_2 &= (0,0,0, c_4, 0, \dots, 0) &\quad &\longleftrightarrow
                \quad    \wt \Bd^2 = (c_4, c_4, c_4, c_4), \\
  \wt\Bc_3 &= (0, \dots, 0, c_5, c_6, c_7, 0, \dots, 0)  &\quad &\longleftrightarrow
                \quad    \wt \Bd^3 = (c_6 + c_7, c_5 + c_7, c_5+c_6, c_5 + c_6 + c_7),    \\
  \wt\Bc_4 &= (0,\dots, 0, c_8, 0, \dots, 0) &\quad &\longleftrightarrow
                 \quad    \wt \Bd^4 = (c_8, c_8, c_8, 2c_8),  \\
  \wt\Bc_5 &= (0, \dots, 0, c_9, c_{10}, c_{11}, 0)  &\quad &\longleftrightarrow
                  \quad   \wt \Bd^5 = (c_9, c_{10}, c_{11}, c_9 + c_{10} + c_{11}),  \\
\wt\Bc_6 &= (0,\dots, 0, c_{12}) &\quad  &\longleftrightarrow 
                 \quad    \wt \Bd^6 = (0,0,0,c_{12}). 
\end{aligned}
\end{equation*}
}

Hence $\wt m(\Bc, \Bh) = F(\wt\Bd^1)\cdots F(\wt\Bd^6)$ is expressed in a similar
formula as (7.1.1).

In turn, for $\ul\Bc = (\ul c_1, \dots, \ul c_6)$, we have

{\footnotesize
\begin{equation*}
\begin{aligned}
\ul\Bc_1 &= (\ul c_1, 0, \dots, 0) &\quad &\longleftrightarrow
                  &\quad \ul\Bd^1 &= (\ul c_1, 0),   \\
\ul\Bc_2 &= (0, \ul c_2, 0, \dots, 0)  &\quad &\longleftrightarrow
            &\quad \ul \Bd^2 &= (\ul c_2, \ul c_2), \\
\ul\Bc_3 &= (0,0, \ul c_3, 0, 0, 0) &\quad &\longleftrightarrow
             &\quad \ul \Bd^3 &= (2\ul c_3, 3 \ul c_3),  \\
\ul\Bc_4 &= (0, 0,0, \ul c_4, 0,0)  &\quad &\longleftrightarrow
              &\quad \ul\Bd^4 &= (\ul c_4, 2\ul c_4),  \\
\ul\Bc_5 &= (0,\dots, 0, \ul c_5, 0) &\quad &\longleftrightarrow
               &\quad \ul\Bd^5 &= (\ul c_5, 3\ul c_5),  \\
\ul\Bc_6 &= (0, \dots, 0, \ul c_6)   &\quad &\longleftrightarrow
               &\quad \ul\Bd^6 &= (0, \ul c_6). 
\end{aligned}
\end{equation*}
}

Hence $m(\ul\Bc, \ul\Bh) = F(\ul\Bd^1)\cdots F(\ul\Bd^6)$ is expressed as in (7.1.2). 

\para{7.6}
Under the setting in 7.5, we consider the case where the weight is given by
$\a = 2\a_1 + 2\a_{1'} + 2\a_{1''} + \a_2$. We can compute the matrix $\vL$
similarly to the case $A_3$.
But since $\vL$ is big (the size of $\vL$ is 8),
we concentrate our attention to the $\s$-fixed part
$\vL^{\s}$.
Then $\s$-invariant PBW bases $L(\Bc, \Bh)$ such that $\weit (L(\Bc,\Bh)) = \a$ are given by
\begin{equation*}
\begin{aligned}
  L(\Bc_1, \Bh) &= F_1F_{1'}F_{1''}F_{11'1''2}  &\quad\text{ with } &\quad
                   \Bc_1 = (1,1,1,1, 0, \dots, 0), \\
L(\Bc_2, \Bh) &= F_1^{(2)}F_{1'}^{(2)}F_{1''}^{(2)}F_2 &\quad\text{ with } &\quad
                    \Bc_2 = (2,2,2, 0,\dots, 0,1).
\end{aligned}
\end{equation*}
Accordingly, $\s$-invariant monomial bases are given by
\begin{equation*}
\begin{aligned}
  \wt m(\Bc_1, \Bh) &= f_{1''}f_{1'}f_1\cdot f_2f_{1''}f_{1'}f_1 &= F_{\Bi, \Ba_1}, \\
  \wt m(\Bc_2, \Bh) &= f_{1''}^{(2)}f_{1'}^{(2)}f_1^{(2)}\cdot f_2 &= F_{\Bi, \Ba_2},
\end{aligned}
\end{equation*}
where $\Bi = (1'',1',1,2,1'',1',1)$ and $\Ba_ 1= (1,1,1,1,1,1,1), \Ba_2 = (2,2,2,1,0,0,0)$. 
\par
We now compute $\vL^{\s}$. In this case, we have
\begin{equation*}
\d_{\a} = (1-q^2)^7, \quad \g_{\Bi, \Ba_1} = 1, \quad \g_{\Bi, \Ba_2} = [2]^3 = (q + q\iv)^3.
\end{equation*}

Then

\begin{equation*}
\tag{7.6.1}  
  \vL^{\s} = \d_{\a}\iv \begin{pmatrix} 
                         \g_{\Bi, \Ba_1}\iv  &    \\
                                 &   \g_{\Bi, \Ba_2}\iv
                         \end{pmatrix}
                         \begin{pmatrix}
                                  8  &  (q + q\iv)^3  \\
                         (q + q\iv)^3  & (1 + q^{-2})^3
                         \end{pmatrix}
                         \begin{pmatrix}
                               \g_{\Bi, \Ba_1}\iv  &    \\
                                   &    \g_{\Bi,\Ba_2}\iv
                         \end{pmatrix}.      
\end{equation*}

\para{7.7.}
We keep the setup in 7.5, and consider $\ul\BU_q^-$ of type $G_2$.
Let $\ul\a = 2\a_{\ul 1} + \a_{\ul 2}$. Then the PBW bases $L(\ul\Bc, \ul\Bh)$ of
weight $\ul\a$ are given by
\begin{equation*}
\begin{aligned}
L(\ul\Bc_1, \ul\Bh) &= F_{\ul 1}F_{\ul{12}} &\quad\text{ with } &\quad \ul\Bc_1 = (1,1,0,0,0,0), \\  
L(\ul\Bc_2, \ul\Bh) &= F_{\ul 1}^{(2)}F_{\ul 2}  &\quad\text{ with }
                         &\quad\ul\Bc_2 = (2,0,0,0,0,1).
\end{aligned}
\end{equation*}

The corresponding monomial bases are given by
\begin{equation*}
\begin{aligned}
  m(\ul\Bc_1, \ul\Bh_1) &= f_{\ul 1}f_{\ul 2}f_{\ul 1} &= F_{\Bj, \ul\Ba_1},  \\     
  m(\ul\Bc_2, \ul\Bh_2) &= f_{\ul 1}^{(2)}f_{\ul 2}    &= F_{\Bj, \ul\Ba_2},
\end{aligned}
\end{equation*}
where $\Bj = (\ul 1, \ul 2, \ul 1)$ and $\ul\Ba_1 = (1,1,1), \ul\Ba_2 = (2,1,0)$.
Then $\ul\vL$ is computed as follows,

\begin{equation*}
\tag{7.7.1}  
  \ul\vL = \d_{\ul\a}\iv \begin{pmatrix}
                           \g_{\Bj, \ul\Ba_1}\iv  &    \\
                                &   \g_{\Bj, \ul\Ba_2}\iv
                          \end{pmatrix}
                          \begin{pmatrix}
                            2   &  q^3 + q^{-3}  \\
                            q^3 + q^{-3}  &  1 + q^{-6}
                          \end{pmatrix}  
                          \begin{pmatrix}
                            \g_{\Bj, \ul\Ba_1}\iv  &   \\
                                    &  \g_{\Bj, \ul\Ba_2}\iv
                          \end{pmatrix}.  
\end{equation*}  

By comparing (7.6.1) and (7.7.1), we see that Remark 7.4 still holds for
$D_4$ and $G_2$. 
\par
The matrix equation $\ul\vL = {}^t\ul H \ul D \ul H$ determines $\ul H$ and $\ul D$
as follows.

{\footnotesize
\begin{align*}
\tag{7.7.2}  
\ul H = \begin{pmatrix}
           1  &   \\
           1 + q^6  &    1
        \end{pmatrix},   \qquad
\ul D = \begin{pmatrix}
            \frac{1}{(1-q^3)(1-q^6)}  &   \\
                   &   \frac{1}{(1-q^3)(1-q^6)(1-q^9)}
        \end{pmatrix}.
\end{align*}  
}

Finally we have
{\footnotesize
\begin{equation*}
\tag{7.7.3}  
\ul H = \begin{pmatrix}
          1  &     \\    
          1 + q^6  &  1 
         \end{pmatrix}
      =  \ul P \ul Q  = \begin{pmatrix}
                            1   &   \\
                            q^6  &   1
                       \end{pmatrix}     
                       \begin{pmatrix}
                          1   &    \\
                          1   &   1 
                       \end{pmatrix}.
\end{equation*}  
}

\par

\par\vspace{1.5cm}
\noindent
T. Shoji \\
School of Mathematical Sciences, Tongji University \\ 
1239 Siping Road, Shanghai 200092, P.R. China  \\
E-mail: \verb|shoji@tongji.edu.cn|

\par\vspace{0.5cm}
\noindent
Z. Zhou \\
School of Digital Science, Shanghai Lida University \\ 
1788 Cheting Road, Shanghai 201608, P.R. China  \\
E-mail: \verb|forza2p2h0u@163.com|

\end{document}